\PassOptionsToPackage{table}{xcolor}

\documentclass[envcountsame]{llncs}
\newif\ifllncs
\llncstrue
\newif\ifsubmission
\submissionfalse

\usepackage[english]{babel}

\usepackage{tcolorbox}
\usepackage{lineno}
\usepackage{amsmath,amsfonts,amssymb,stmaryrd}
\usepackage[noend]{algpseudocode}
\usepackage{algorithm}
\usepackage{tikz}
\usepackage{tikz-cd}
\usepackage{graphics}
\usepackage{array}
\usepackage{booktabs}
\usepackage{enumitem}
\usepackage{thmtools, thm-restate}

\usepackage[colorlinks=true,linktocpage]{hyperref}
\providecolor{DarkBlue}{rgb}{0,0,.545}
\providecolor{DarkGreen}{rgb}{0,.392,0}
\hypersetup{citecolor=DarkGreen}
\hypersetup{linkcolor=DarkBlue}
\hypersetup{urlcolor=DarkBlue}
\usepackage{cleveref}
\crefname{problem}{problem}{problems}
\Crefname{problem}{Problem}{Problems}
\Crefname{equation}{Eq.}{Eqs.}
\usepackage{xspace}

\newcommand\unnumberedfootnote[1]{%
  \begingroup
  \renewcommand\thefootnote{}\footnote{#1}%
  \addtocounter{footnote}{-1}%
  \endgroup
}

\newcommand{\ie}{i.e.\xspace}

\newcommand{\cryptoalg}[3]{%
	\newcommand{#1}{\hyperref[#3]{\ensuremath{#2}}\xspace}
}

\newcommand{\GB}[1]{\textcolor{purple}{Giacomo: #1}\ }
\newcommand{\Sina}[1]{\textcolor{blue}{Sina: #1}\ }
\newcommand{\Luca}[1]{\textcolor{green}{Luca: #1}\ }
\newcommand{\Guido}[1]{\textcolor{orange}{Guido: #1}\ }

\newcommand{\tomatrix}[2]       {\mathsf{M}_{#1}}%
\newcommand{\dual}[1]           {\widehat{#1}}
\newcommand{\conjugate}[1]      {\overline{#1}}
\newcommand{\adjoint}           {\dagger}
\newcommand{\spanned}[1]         {\left\langle #1 \right\rangle}%

\newcommand{\prob}              {\ensuremath{\mathsf{Pr}}\xspace}
\newcommand{\randgets}          {\ensuremath{\gets_\$}\xspace}

\newcommand{\bigO}              {\ensuremath{O}\xspace}
\newcommand{\action}            {\star}

\newcommand{\secpar}            {\ensuremath{\lambda}\xspace}

\newcommand{\pk}                {\ensuremath{\mathsf{pk}}\xspace}
\newcommand{\sk}                {\ensuremath{\mathsf{sk}}\xspace}

\newcommand{\challenge}         {\ensuremath{\mathsf{ch}}\xspace}
\newcommand{\cmt}               {\ensuremath{\mathsf{com}}\xspace}
\newcommand{\commit}            {\ensuremath{\cmt}\xspace}
\newcommand{\ch}                {\ensuremath{\challenge}\xspace}
\newcommand{\resp}              {\ensuremath{\mathsf{rsp}}\xspace}
\newcommand{\aux}               {\ensuremath{\mathsf{aux}}\xspace}
\newcommand{\msg}               {\ensuremath{\mathsf{msg}}\xspace}

\newcommand{\Prot}              {\ensuremath{\mathsf{\Sigma}}\xspace}
\newcommand{\mycomment}[1]      {\quad \textcolor{gray}{// #1}}

\newcommand{\Adv}               {\ensuremath{\mathsf{Adv}}\xspace}
\newcommand{\Advantage}         {\Adv}
\newcommand{\adv}               {\calA}
\newcommand{\bdv}               {\calB}
\newcommand{\Sim}               {\ensuremath{\mathsf{Sim}}\xspace}

\newcommand{\Gen}               {\ensuremath{\mathsf{Gen}}\xspace}

\newcommand{\calA}              {\ensuremath{\mathcal{A}}\xspace}
\newcommand{\calB}              {\ensuremath{\mathcal{B}}\xspace}

\newcommand{\calD}              {\ensuremath{\mathcal{D}}\xspace}
\newcommand{\calE}              {\ensuremath{\mathcal{E}}\xspace}
\newcommand{\negl}              {\ensuremath{\mathsf{negl}}\xspace}
\newcommand{\poly}              {\ensuremath{\mathsf{poly}}\xspace}
\newcommand{\SIG}               {\ensuremath{\mathsf{SIG}}\xspace}
\newcommand{\EUFCMA}            {\ensuremath{\mathsf{EUF\text{-}CMA}}\xspace}

\newcommand{\alg}               {\ensuremath{{\mathcal{B}_{p,\infty}}}\xspace}
\newcommand{\QuaternionAlgebra} {\ensuremath{{\alg}}\xspace}
\newcommand{\Bpinf}             {\ensuremath{\QuaternionAlgebra}\xspace}
\newcommand{\Zmod}              {\ensuremath{\Z_{/N}}\xspace}

\newcommand{\GL}                {\ensuremath{\mathsf{GL}}\xspace}
\newcommand{\Mat}               {\ensuremath{\mathsf{M}}\xspace}

\newcommand{\GLtwo}             {\ensuremath{\GL_2(\Zmod)}\xspace}
\newcommand{\Mattwo}            {\ensuremath{\Mat_2(\Zmod)}\xspace}

\newcommand{\Q}                 {\ensuremath{\mathbb{Q}}\xspace}
\newcommand{\Z}                 {\ensuremath{\mathbb{Z}}\xspace}

\newcommand{\F}                 {\ensuremath{\mathbb{F}}\xspace}
\renewcommand{\O}               {\ensuremath{\mathcal{O}}\xspace}
\newcommand{\LatticeOfSolutions}{\ensuremath{\Lambda}\xspace}

\newcommand{\commitIsogeny}     {\ensuremath{\varphi_{\commit}}\xspace}
\newcommand{\commitCurve}       {\ensuremath{E_{\commit}}\xspace}
\newcommand{\commitIdeal}       {\ensuremath{I_{\commit}}\xspace}

\newcommand{\commitMatrix}      {\ensuremath{\Mat_{\commit}}\xspace}
\newcommand{\auxIsogeny}        {\ensuremath{\varphi_{\aux}}\xspace}
\newcommand{\auxCurve}          {\ensuremath{E_{\aux}}\xspace}
\newcommand{\respIsogeny}       {\ensuremath{\varphi_{\resp}}\xspace}

\newcommand{\chCurve}           {\ensuremath{E_\ch}\xspace}
\newcommand{\chIsogeny}         {\ensuremath{\varphi_{\challenge}}\xspace}
\newcommand{\generalIsogeny}    {\ensuremath{\zeta}\xspace}
\newcommand{\generalIdeal}      {\ensuremath{I_\zeta}\xspace}
\newcommand{\secretIsogeny}     {\ensuremath{\varphi_{\sk}}\xspace}
\newcommand{\secretIdeal}       {\ensuremath{I_{\sk}}\xspace}

\newcommand{\publicCurve}       {\ensuremath{E_\pk}\xspace}

\newcommand{\hint}[1]          {\ensuremath{\mathcal{H}_{#1}}\xspace}
\newcommand{\hintsim}[1]       {\hyperref[dist:hint_sim]{\ensuremath{\hint{#1}^{\sf sim}}}\xspace}

\newcommand{\hintunif}[1]      {\hyperref[dist:hint_unif]{\ensuremath{\hint{#1}^{\sf unif}}}\xspace}
\newcommand{\hintdistr}        {\ensuremath{\mathcal{H}}\xspace}

\newcommand{\probcond}[2]{\ensuremath{\Pr\left[\begin{aligned}#1\end{aligned}\;\middle|\; \begin{aligned}#2\end{aligned}\right]}}

\newcommand{\degreeBound}   {\ensuremath{B_{\resp}}\xspace}

\newcommand{\torsion}       {\ensuremath{N}\xspace}
\newcommand{\level}         {\ensuremath{N}\xspace}

\newcommand{\etorsion}      {\ensuremath{e_\mathsf{tor}}\xspace}

\newcommand{\respDeg}       {\ensuremath{D_{\resp}}\xspace}

\newcommand{\tor}           {\mathsf{tor}}
\newcommand{\B}             {\ensuremath{\mathcal{B}}\xspace}

\newcommand{\Ppk}           {\ensuremath{P_{\pk}}\xspace}
\newcommand{\Qpk}           {\ensuremath{Q_{\pk}}\xspace}
\newcommand{\Paux}          {\ensuremath{P_{\aux}}\xspace}
\newcommand{\Qaux}          {\ensuremath{Q_{\aux}}\xspace}
\newcommand{\Ppktor}        {\ensuremath{P_{\pk}}\xspace}
\newcommand{\Qpktor}        {\ensuremath{Q_{\pk}}\xspace}
\newcommand{\Pcmttor}       {\ensuremath{P_{\commit}}\xspace}
\newcommand{\Qcmttor}       {\ensuremath{Q_{\commit}}\xspace}
\newcommand{\Pcomtor}       {\Pcmttor}
\newcommand{\Qcomtor}       {\Qcmttor}

\newcommand{\OrderEnd}                      {\ensuremath{\O}\xspace}

\newcommand{\End}                           {\ensuremath{\mathsf{End}}\xspace}

\newcommand{\norm}                          {\ensuremath{\mathsf{n}}\xspace}
\newcommand{\normklpt}[1]                   {\norm{(#1)}}
\newcommand{\tr}                            {\ensuremath{\mathrm{tr}}\xspace}

\newcommand{\publicLevelStructure}          {\ensuremath{S_\pk}\xspace}
\newcommand{\chLevelStructure}              {\publicLevelStructure}
\newcommand{\commitLevelStructure}          {\ensuremath{S_{\commit}}\xspace}
\newcommand{\chMatrix}                      {\ensuremath{M_{\ch}}\xspace}

\newcommand{\curvelevel}                    {\ensuremath{\mathbb{E}}\xspace}

\newcommand{\SQI}                       {\ensuremath{\mathrm{SQI}}\xspace}
\newcommand{\SQIsign}                   {\ensuremath{\mathrm{SQIsign}}\xspace}
\newcommand{\SQInstructor}              {\ensuremath{\mathrm{SQInstructor}}\xspace}
\newcommand{\SQItorsion}                {\SQInstructor}
\newcommand{\SQIprime}                  {\ensuremath{\mathrm{SQIprime}}\xspace}

\newcommand{\CRT}                       {\ensuremath{\mathsf{CRT}}\xspace}
\newcommand{\OneEndp}                   {\ensuremath{\mathbf{OneEnd}_p}\xspace}
\newcommand{\ROneEnd}                   {\mathcal{R}_{\Gamma\text{-}\mathrm{OneEnd}}}

\newcommand{\hintOneEndp}               {\ensuremath{\mathbf{hint}\text{-}\OneEndp}\xspace}
\newcommand{\hintdist}                  {\ensuremath{\mathbf{hint}\text{-}\mathbf{dist}}\xspace}

\cryptoalg{\SigningKLPT}				{\mathsf{SigningKLPT}}{alg:signing_klpt}
\cryptoalg{\GenericKLPT}				{\mathsf{GenericKLPT}}{alg:generic_klpt}
\cryptoalg{\ScalarSigningKLPT}          {\mathsf{Scalar}\allowbreak\mathsf{Signing}\allowbreak\mathsf{KLPT}}{alg:scalar_klpt}
\cryptoalg{\BorelSigningKLPT}			{\mathsf{Borel}\allowbreak\mathsf{Signing}\allowbreak\mathsf{KLPT}}{alg:borel_klpt}
\cryptoalg{\InitialRerandomization}     {\ensuremath{\mathsf{Initial}\allowbreak\mathsf{Rerandomization}}}{alg:scalar_rerandomization}
\newcommand{\RepresentInteger}			{\ensuremath{\mathsf{Represent}\allowbreak\mathsf{Integer}}\xspace}
\newcommand{\StrongApproximation}		{\ensuremath{\mathsf{Strong}\allowbreak\mathsf{Approximation}}\xspace}
\newcommand{\Full}				        {\ensuremath{\mathsf{Full}}}
\cryptoalg{\FullRepresentInteger}       {\Full\allowbreak\RepresentInteger}{alg:FullRepresentInteger}
\cryptoalg{\FullStrongApproximation}    {\Full\allowbreak\StrongApproximation}{alg:FullStrongApproximation}
\cryptoalg{\GeneralizedScalarSigningKLPT}{\mathsf{Generalized}\allowbreak\mathsf{Scalar}\allowbreak\mathsf{Signing}\allowbreak\mathsf{KLPT}}{alg:gen_scalar_klpt}
\cryptoalg{\IdealEichlerNorm}           {\mathsf{Ideal}\allowbreak\mathsf{Eichler}\allowbreak\mathsf{Norm}}{alg:IdealEichlerNorm}
\cryptoalg{\IdealSuborderEichlerNorm}   {\mathsf{Ideal}\allowbreak\mathsf{Suborder}\allowbreak\mathsf{Eichler}\allowbreak\mathsf{Norm}}{alg:IdealSuborderEichlerNorm}
\cryptoalg{\EquivalentPrimeIdeal}       {\mathsf{Equivalent}\allowbreak\mathsf{Prime}\allowbreak\mathsf{Ideal}}{alg:EquivalentPrimeIdeal}

\cryptoalg{\ComputeSolutions}			{\mathsf{InStruct}\-\mathsf{Ideal}}{alg:compute_solutions}
\cryptoalg{\ResponseIdeal}			{\mathsf{Response}\-\mathsf{Ideal}}{alg:response_ideal}

\cryptoalg{\FIDIO}                      {\mathsf{FIDIO}}{def:fidio}
\cryptoalg{\UTO}                        {\Gamma\text{-}\mathsf{UTO}}{def:uto}

\renewcommand{\GB}[1]{}
\renewcommand{\Sina}[1]{}
\renewcommand{\Luca}[1]{}
\renewcommand{\Guido}[1]{}

\newcommand{\fakepar}[1]{\noindent\textit{\underline{#1}.}}

\newsavebox\myboxA
\newsavebox\myboxB
\newlength\mylenA

\newcommand*{\conjugated}[2][0.75]{%
    \sbox{\myboxA}{$\m@th#2$}%
    \setbox\myboxB\null%
    \ht\myboxB=\ht\myboxA%
    \dp\myboxB=\dp\myboxA%
    \wd\myboxB=#1\wd\myboxA%
    \sbox\myboxB{$\m@th\overline{\copy\myboxB}$}%
    \setlength\mylenA{\the\wd\myboxA}%
    \addtolength\mylenA{-\the\wd\myboxB}%
    \ifdim\wd\myboxB<\wd\myboxA%
        \rlap{\hskip 0.5\mylenA\usebox\myboxB}{\usebox\myboxA}%
    \else
        \hskip -0.5\mylenA\rlap{\usebox\myboxA}{\hskip 0.5\mylenA\usebox\myboxB}%
    \fi
}

\newcommand{\lto}{\longrightarrow}
\newcommand{\ol}        {\ensuremath{\overline}\xspace}

\newcommand{\bbE}       {\ensuremath{\mathbb E}\xspace}

\newcommand{\Fpsq}      {\ensuremath{\mathbb F_{p^2}}\xspace}

\newcommand{\calO}      {\ensuremath{\mathcal{O}}\xspace}
\newcommand{\Norm}{\norm}
\newcommand{\Hom}{\mathrm{Hom}}

\newcommand{\covolume}{\mathrm{Cov}}

\newcommand{\com}{\cmt}
\newcommand{\rsp}{\mathrm{rsp}}
\newcommand{\chl}{\ch}

\newcommand{\Id}{\mathrm{Id}}

\newcommand{\smt}[4]{\left[\begin{smallmatrix}{#1}&{#2}\\{#3}&{#4}\end{smallmatrix}\right]}
\newcommand{\vv}[2]{\left(\begin{smallmatrix}{#1}\\{#2}\end{smallmatrix}\right)}

\newlist{thmlist}{enumerate}{1}
\setlist[thmlist]{label=(\roman{thmlisti}),
                  ref=\thetheorem(\roman{thmlisti}),
                  noitemsep}
\newlist{lemlist}{enumerate}{1}
\setlist[lemlist]{label=(\roman{lemlisti}),
                  ref=\thelemma(\roman{lemlisti}),
                  noitemsep}
\newlist{proplist}{enumerate}{1}
\setlist[proplist]{label=(\roman{proplisti}),
                  ref=\theproposition(\roman{proplisti}),
                  noitemsep}
\Crefname{listthm}{Theorem}{Theorems}
\Crefname{listlem}{Lemma}{Lemmas}
\Crefname{listprop}{Proposition}{Propositions}

\algrenewcommand\algorithmicrequire{\textbf{Input: }}
\algrenewcommand\algorithmicensure{\textbf{Output: }}

\newcommand{\idprot} {\ensuremath{\Prot_\SQItorsion}\xspace}
\newcommand{\sigprot}{\ensuremath{\mathsf{SIG}[\idprot]}\xspace}
\newcommand{\idprotoriginal}{\ensuremath{\Prot_\SQI}\xspace}

\newcommand{\sqinstrctrepo}{\href{https://github.com/SQInstructors/SQInstructor}{github.com/SQInstructors/SQInstructor}}
\newcommand{\dataurl}{\sqinstrctrepo}

\newcommand{\klpturl}{\sqinstrctrepo}

\title{The SQInstructor: a guide to SQIsign and the Deuring Correspondence with level structures}

\author{
    Giacomo Borin\inst{1,2}
    \and
    Luca De Feo\inst{1}
    \and
    Guido Lido\inst{3}
    \and
    Sina Schaeffler\inst{1,4}
}

\institute{
  IBM Research Europe, Switzerland
  \and
  University of Zurich, Switzerland
  \and
  University of Roma Tor Vergata, Italy
  \and
  ETH Zurich, Switzerland
}
\authorrunning{Giacomo Borin, Luca De Feo, Guido Maria Lido and Sina Schaeffler}

\begin{document}

\maketitle\unnumberedfootnote{Author list in alphabetical order; see \url{https://ams.org/profession/leaders/CultureStatement04.pdf.}}

\begin{abstract}
We explore the use of level structures to generalize the SQIsign signature scheme.
We give a general framework where, given the public key and the commitment, the challenge is to exhibit an isogeny between them with an additional requirement, namely to map a chosen level structure to another.
We then instantiate the framework using 1-dimensional and 2-dimensional isogenies.
In doing that we provide a new explicit Deuring correspondence for supersingular elliptic curves with level structures and solve new constrained norm equations.
\end{abstract}

\section{Introduction}
Isogeny-based cryptography is a promising candidate for post-quantum
cryptography, with several proposals for key exchange
and digital signatures, but also for more advanced primitives such as
threshold and ring signatures.

A mainstay of isogeny-based cryptography is
\SQIsign~\cite{AC:DKLPW20}, which is so far the most compact signature
scheme participating in NIST's post-quantum standardization
process~\cite{NIST-call}. %
At its heart, it is based on the \emph{Deuring correspondence}, a
mathematical and algorithmic equivalence between supersingular curves
and orders of a quaternion algebra.

The original instantiation of \SQIsign constructs signatures by
solving some quadratic equations in the quaternion algebra, using an
algorithm known as KLPT~\cite{klpt,EC:DLLW23}. %
This approach gives very compact signatures, with a relatively simple
verification procedure, but a quite involved and slow signing
procedure. %
It also presents an intrinsic difficulty in proving EUF-CMA security,
owing to the limitations of KLPT.
A paradigm shift for \SQIsign,
first proposed in~\cite{EC:DLRW24} and then improved~\cite{AC:BDDLMP24,AC:DupFou24,AC:NOCCIL24}, 
came with the 
introduction of \emph{higher-dimensional (HD) isogenies}.
Building upon a line of work started with the SIDH
attacks~\cite{EC:CasDec23,EC:MMPPW23,EC:Robert23},
these techniques significantly sped up \SQIsign and improved its provable security.

Taking a step back, 
\emph{level structures} are bases of torsion subgroups of elliptic curves, up to 
transformation by a well chosen group of matrices.
They have been used to prove indistinguishability of certain distributions 
on supersingular isogeny graphs~\cite{EC:BCCDFL23}, and to classify hardness assumptions in isogeny-based cryptography~\cite{EC:DeFFouPan24}.
However level structures are implicitly featured in many more
isogeny-based protocols, for example in SIKE~\cite{PQCRYPTO:JaoDeF11}
or FESTA~\cite{AC:BasMaiPop23}. %
Even \SQIsign itself can be understood in terms of level structures. %
Indeed, a \emph{response} in the \SQIsign interactive identification
protocol is generally understood to be an isogeny connecting a
\emph{public key curve} to a \emph{challenge curve}. %
The challenge curve is generated as the codomain of a
\emph{challenge isogeny}, which is itself defined in terms of its
\emph{kernel}, a form of level structure. %
Thus, an equivalent characterization of a \SQIsign signature is as an
isogeny (the composition of response and challenge isogeny) vanishing
on a level structure (the kernel). %
This point of view is especially visible in the variant of \SQIsign
named \SQIprime~\cite{AC:DupFou24}.

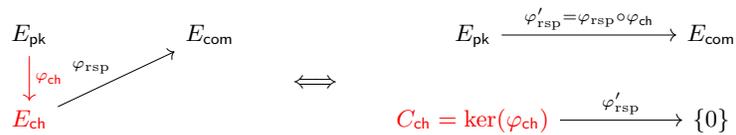
\begin{figure}[!ht]
	\centering
$$
\intextsep 0pt
\begin{tikzcd}[]
	E_{\pk} \arrow[d, "\varphi_{\ch}", red] 
	&& E_{\com} 
	& \, \arrow[d, phantom, "\iff"] &
	E_{\pk} \arrow[rr, "\varphi_{\rsp}' = \varphi_\rsp \circ \varphi_\ch	"] 
	&& E_{\com} 
	\\
	{\color{red}E_{\ch}} \arrow[urr, "\varphi_{\rsp}"] 
	&& 
	&\, &
	{\color{red}C_{\ch}=\ker(\varphi_{\ch})} \arrow[rr, "\varphi_{\rsp}' "] 
	&& \{0\}
\end{tikzcd}
$$
\caption{\small SQIsign simplified (generation of key and of commitment is omitted)}
\label{fig:SQISign_simplified_1}
\end{figure}
With this point of view, it is very natural to consider signatures where the response is requested to send a certain cyclic subgroup (or a suitable level structure) to another one, instead of annihilating it. This is exactly the idea of our new proposal: \SQItorsion is a general framework following this idea, with freedom to choose the kind of level structure 
and whether the response should be computed with one-dimensional or 
with HD techniques.

\begin{figure}[!ht]
	\centering
$$
\begin{tikzcd}[column sep=large]
	E_{\pk} \arrow[rr, "\varphi_{\rsp}"] 
	&& E_{\com} 
	\\
	{\color{red} C_{\ch} \,\,\text{ or }\,\, \phi_\ch } \arrow[rr, mapsto] 
	&& C_\com \,\,\text{ or }\,\, \phi_\com 
\end{tikzcd}
$$
\abovecaptionskip -1pt
\caption{SQInstructor simplified (generation of key and of commitment is omitted)}
\label{fig:SQISign_simplified_2}
\end{figure}
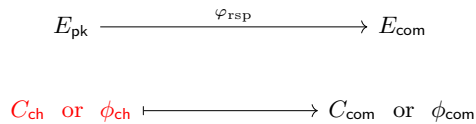

\paragraph{Contributions.}
In this work we instantiate and explore in detail the use of level structures 
for the \SQItorsion protocol. Progressively in each section we:
\begin{itemize}
    \item Generalize the well-known Deuring Correspondence to the setting of
        supersingular elliptic curves with level structure, with some concrete 
        examples. Remarkably, we show how to make this correspondence explicit, 
        even for large non-smooth levels.
    \item We describe the new identification scheme \SQItorsion in full generality, 
        leaving open the choice of the level structure and isogeny 
        representation for the specific instantiation of the protocol.
    \item Instantiating the framework using 1-dimensional isogenies, \emph{a 
        la}~\cite{AC:DKLPW20} requires solving specific \emph{constrained} norm 
        equations in the quaternion algebra \Bpinf.
        We show how to modify and combine previous algorithms in the literature 
        to efficiently solve these equations in practice. We believe this is of
        independent interest. We also give a proof-of-concept implementation of 
        these algorithms.
    \item We instantiate the framework using 2-dimensional isogenies, prove its
        security under assumptions similar to those of \SQIsign, and give an 
        optimized implementation based on~\cite{NISTPQC-ADD-R2:SQIsign24}. 
\end{itemize}

\paragraph{Related work.}
Arpin~\cite{arpin-sarah-borel-lvl-structures} initiated the study of
supersingular isogeny graphs with Borel level structure and
generalized the Deuring correspondence to them. %
Here we extend the correspondence to even more general level
structures, and explain for the first time how to build SQIsign-like
identification protocols from it.

In~\cite{EC:DeFFouPan24}, level structures were used to classify
SIDH-like hardness assumptions and prove reductions between them,
showing in particular that the SIDH
attacks~\cite{EC:CasDec23,EC:MMPPW23,EC:Robert23} apply to a broader
class of problems than initially thought. %
Although we use level structures in this work, we only do so
constructively and we never rely on the kind of assumptions studied
in~\cite{EC:DeFFouPan24}. %
The security of \SQInstructor rests on the same foundations as SQIsign
itself.

Basso, De Feo, Patranabis, Radulescu and Wesolowski~\cite{forensic} introduce an abstract framework for
SQIsign-like schemes called \emph{forensic categories}. %
Among the proposed instantiations of the framework, one is based on
level structures and yields an identification protocol equivalent to
\SQInstructor's\footnote{Precisely, equivalent to the 1-dimensional
  variant of \Cref{sec:klpt}.} %
Their work is complementary to ours: whereas they focus on protocols
and show that one can construct ring signatures and chameleon hashes
from any forensic category, we focus on algorithms, carefully
generalizing the Deuring correspondence and explaining how to
implement \SQInstructor's 1- and 2-dimensional variants.

\paragraph{Acknowledgments.}
We thank Andrea Basso, Sikhar Patranabis, Ilinca Radulescu and Benjamin Wesolowski for the useful discussions, and Antonin Leroux for pointers to KLPT variants.

 The third author is supported by the MIUR Excellence Department Project MatMod@TOV awarded to the Department of Mathematics, University of Rome Tor Vergata, by the PRIN PNRR 2022 ``Mathematical Primitives for Post Quantum Digital Signatures'' and by GNSAGA - INdAM. All other authors are supported by SNSF Consolidator Grant CryptonIs 213766.

\section{The Deuring Correspondence in Cryptography}
\label{sec:preliminaries}

We recall here some basic facts about the Deuring correspondence and
its applications to cryptography, in particular to SQIsign. %
We assume some familiarity with elliptic curves, isogenies and
quaternion algebras. %
For more details, we refer
to~\cite{silverman2009arithmetic,quaternion_book,de2017mathematics,boneh2020graduate}.

Throughout this document we fix a prime $p$ and we write $\F_p$ for
the finite field with $p$ elements, $\F_{p^2}$ for a field extension
of degree $2$, and $\overline{\F}_p$ for their algebraic closure.  %
Finally, we write $\Bpinf$ for the quaternion algebra ramified at $p$
and at infinity. %

\subsubsection{An equivalence of categories.}
If $E,E'$ are elliptic curves, we write $\Hom(E,E')$ for the group of
homomorphisms $E\to E'$, and $\End(E) := \Hom(E,E)$. %
As a $\Z$-module, $\Hom(E,E')$ has finite rank, and the usual degree
map is a positive definite quadratic form on it. %
In short, $\Hom(E,E')$ is a Euclidean lattice. %
On top of that, composition of morphisms makes $\End(E')$ a ring and
$\Hom(E,E')$ a left $\End(E')$-module.

Deuring proved that if $E,E'$ are supersingular curves in
characteristic $p$, then $\Hom(E,E')$ has rank $4$ and $\End(E')$ is
isomorphic to a maximal order in $\Bpinf$. %
But his correspondence goes even deeper: it shows that isomorphism
classes of supersingular elliptic curves, up to Galois conjugation in
$\F_{p^2}/\F_p$, are in one-to-one correspondence with isomorphism
types of maximal orders of $\Bpinf$ (both sets are finite), and that
the integral left ideals of $\End(E)$ are in one-to-one
correspondence with the isogenies of $E$, up to post-composition with
isomorphisms. %
In modern terms, fixing a curve $E_0$ and an order
$\O_0 \simeq \End(E_0)$, the functor $\Hom(-,E_0)$ defines an
equivalence of categories between supersingular elliptic curves with
isogenies and invertible left $\O_0$-modules with non-zero
homomorphisms.  %

\begin{theorem}
  \label{th:deuring-corr}
  Let $E_0$ be a supersingular curve and let $\rho:E_0\to E$ be an
  isogeny. %
  Define the ideal
  \begin{equation}
    \label{eq:deuring-corr}
    I_{\rho} := I_{\rho,E} := \{ \psi \circ \rho : \psi \in \Hom(E,E_0) \}  \subset \End(E_0) \ ,
  \end{equation}
  then $\norm(I_\rho) = \deg(\rho)$.  %
  Two isogenies $\rho,\rho'$ define the same ideal if and only if
  $\rho'=\iota\circ\rho$ for some isomorphism $\iota$. %
  Any integral ideal of $\End(E_0)$ arises this way.
\end{theorem}

In what follows, we shall write $I_\rho\action E_0$ for the
isomorphism class of $E$, and sometimes abuse notation and simply
write $E = I_\rho\action E_0$. %
Two ideals $I,J$ yield the same class $I\action E_0 = J\action E_0$ if
and only if they are in the same left ideal class, i.e., if and only
if $I = J\alpha$ for some $\alpha \in \Bpinf$. %
Equivalently, we may write $[I]$ for the ideal class of $I$ and
$[I]\action E_0$ for the unique isomorphism class of $I\action E_0$.

\subsubsection{An algorithmic tool.}
Deuring's correspondence also has an algorithmic content that
isogeny-based cryptography has helped unearth. %
It can be summarized by the mantra: \emph{some problems are hard for
  curves, every problem becomes easy in the quaternions.}

Concretely, the foundations of isogeny-based cryptography rest on the
famous \emph{isogeny problem}:

\begin{problem}[Supersingular Isogeny Problem]
  \label{prob:isogeny}
  Given supersingular elliptic curves $E,E'/\F_{p^2}$, compute an
  isogeny between them.
\end{problem}

This problem is assumed to be exponentially hard, even for quantum
computers, when at least one of $E$ or $E'$ is chosen at random. %
But the analogous problem for quaternions, that is to compute a
\emph{connecting ideal} between two given maximal orders of $\Bpinf$,
is solved by basic integer linear algebra. %
In fact, Kohel, Lauter, Petit and Tignol~\cite{klpt} showed that even
a more restricted variant of the problem, where the connecting ideal
is requested to have prime-power norm, can be solved in polynomial
time. %
Their algorithm has become a staple of isogeny-based cryptography,
under the name of KLPT. %

This disparity between curves and quaternions suggests a strategy to
solve the isogeny problem: because curves are in (almost) one-to-one
correspondence with maximal order types, if we could effectively
evaluate this correspondence, we could translate isogeny problems to
quaternion problems, solve them and bring the solution back to the
curves. %
The intuition checks out and leads to the inevitable conclusion that
computing endomorphism rings of supersingular elliptic curves is
generically hard. %
The problem below is indeed proven polynomial-time equivalent to the
isogeny problem for a random input curve, assuming
GRH~\cite{FOCS:Wesolowski21}.

\begin{problem}[Supersingular Endomorphism Ring Problem]
  \label{prob:endring}
  Given a supersingular elliptic curve $E$ defined over $\F_{p^2}$,
  find a basis for $\End(E)$.
\end{problem}

On the other hand, the ``converse'' problem is now known to be easy:
given a maximal order $\O$ of $\Bpinf$, a supersingular curve such that
$\End(E) \simeq \O$ may be computed as follows:
\begin{enumerate}
\item Start with a pair $(E_0,\O_0)$ such that
  $\End(E_0) \simeq \O_0$;\footnote{For any $p$, there exists a
    polynomial-time algorithm to produce a supersingular curve
    together with its endomorphism ring.  Note this does not
    contradict the fact that computing endomorphism rings of random
    curves is hard.}
\item Use KLPT to compute a connecting ideal $I$ between $\O_0$ and
  $\O$ of norm $\ell^n$ for some small prime $\ell$;
\item Use an \emph{ideal-to-isogeny} algorithm~\cite{EC:EHLMP18} to
  compute $E = I\action E_0$ by evaluating an isogeny of degree
  $\ell^n$. Then $\End(E)\simeq \O$.
\end{enumerate}

Or at least this is how one would have solved the problem before
2022. %
The discovery of efficient attacks on
SIKE~\cite{EC:CasDec23,EC:MMPPW23,EC:Robert23} brought several
innovations in isogeny-based cryptography, including more efficient
ideal-to-isogeny
algorithms~\cite{EPRINT:PagRob23,AC:BDDLMP24,AC:NOCCIL24,qlapoti}. %
These leverage \emph{higher dimensional isogenies} to bypass the
constraint of small-prime-power isogenies and thus dispense with KLPT
entirely.  %
There also exist \emph{isogeny-to-ideal} algorithms, though they are
currently more limited in scope (see~\cite{AC:CIIKLP23}). %
We will not discuss them as they are not necessary for our purpose.

\subsubsection{From the correspondence to cryptography.}
Summarizing, it is easy to construct supersingular curves knowing
their endomorphism rings, but once we throw away this information it
is hard to recover it.  %
This defines at the very least a one-way function, but hopefully
something more. %
The first to observe that one could use the knowledge of the
endomorphism ring as a ``trapdoor'' in a signature scheme were
Galbraith, Petit and Silva (GPS)~\cite{JC:GalPetSil20}, but it is only
with SQIsign~\cite{AC:DKLPW20} that the Deuring correspondence truly
showed its full potential. %

Both protocols define public parameters $(E_0,\O_0)$ with
$\O_0\simeq\End(E_0)$. %
Both generate the secret key as a random ideal $I_\sk\subset\O_0$, and
define the public key as $\publicCurve := I_\sk\action E_0$. %
Both are based on sigma-protocols proving knowledge of
$\End(\publicCurve)$, but while GPS resembles the classic GMW protocol
for graph isomorphism~\cite{FOCS:GolMicWig86}, SQIsign uses a
radically different strategy achieving exponentially low soundness
error. %
Seen from afar, SQIsign's interactive protocol proceeds as follows:
\begin{enumerate}
\item The prover generates a random ideal $I_\cmt \subset \O_0$,
  computes $\commitCurve := I_\cmt\star E_0$ and sends $\commitCurve$
  to the verifier;
\item\label{step:chall} The verifier responds with a challenge
  $P_\chl \in \publicCurve[N]$, a point of fixed (and smooth) order
  $N$ on $\publicCurve$;
\item The prover uses its knowledge of
  $\Hom(\publicCurve,\commitCurve) \simeq \overline{I_\sk}I_\cmt$ to
  produce an isogeny $\respIsogeny:\publicCurve\to\commitCurve$ such
  that $\respIsogeny(P_\chl) = 0$, sends $\respIsogeny$ to the verifier.
\item The verifier checks that $\respIsogeny$ matches all constraints.
\end{enumerate}

\begin{figure}
\centering
  \begin{tikzpicture}
    \node[anchor=south] (E0) at (0,0){$E_0$};
    \node[anchor=west] (com) at (3,-1){$\commitCurve$};
    \node[anchor=east] (pk) at (-3,-1){$\publicCurve$,$\textcolor{blue}{P_\ch}$};

    \draw[red,-latex,dashed]
    (E0) edge node[auto] {$\commitIsogeny$} (com)
    (E0) edge node[auto,swap] {$\secretIsogeny$} (pk);
    \draw[-latex]
    (pk) edge
    node[above] {$\respIsogeny$}
    node[below] {$\respIsogeny( P_\ch ) = 0$}
    (com);
  \end{tikzpicture}
\caption{\SQIsign identification protocol. Secrets dashed
    in red. Challenge in blue.}
\label{fig:triangle_sqi}
\end{figure}
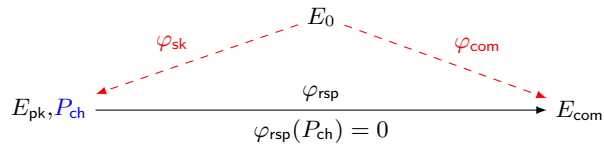

The key difference between GPS and SQIsign lies in
Step~\ref{step:chall}. %
There are $\approx N^2$ choices for $P_\chl$, though linearly
equivalent points are effectively the same challenge, so the number of
distinct challenges is rather $\approx N$.  %
Because the protocol is proven 2-special sound, it has soundness error
$\approx 1/N$, as opposed to $1/2$ for GPS. %
If $N$ is exponentially large, it thus suffices to run a single
iteration of the identification protocol to achieve the desired level
of security.

At first sight, using torsion points as challenges may seem a hack
with little relation to the Deuring correspondence. %
Looking closer we notice that the set
$\Lambda \subset \Hom(\publicCurve,\commitCurve)$ of isogenies
vanishing on $P_\chl$ forms a sublattice.  %
Precisely, letting $\chIsogeny:\publicCurve\to\chCurve$ be the isogeny
such that $\ker\chIsogeny = \langle P_\ch\rangle$, we have
\begin{equation}
  \Lambda = \{ \psi\circ\chIsogeny : \psi\in\Hom(\chCurve,\commitCurve) \}.
\end{equation}
Some readers will find our explanation unnecessarily contrived: we may
as well have said that $\chIsogeny$ is the challenge sent by the
verifier, and that the response is computed from
$\Hom(\chCurve,\commitCurve)$. %
However this point of view, which is the traditional way of presenting
SQIsign, misses two important points: First, if $\chIsogeny$ is the
challenge, we need an isogeny-to-ideal algorithm, somewhat
complicating the protocol. %
This was already observed in~\cite{AC:DupFou24}.  %
Second, there are potentially other choices of sublattices
$\Lambda \subset \Hom(\publicCurve,\commitCurve)$ that lead to
interesting variants of SQIsign. %
This is the topic of the present work.

\subsubsection{Roadmap.}
We start by characterizing which restrictions on $\Hom(E,E')$ lead to
a generalized Deuring correspondence. %
It is tempting to impose that a point $P\in E[N]$ is mapped to some
other point $P\in E'[N]$, but this restriction does not define a
sublattice of $\Hom(E,E')$. %
We will see instead in the next section that the appropriate concept
is that of mapping elliptic curves with \emph{level structures},
extending previous work of Arpin on \emph{Borel level
  structures}~\cite{arpin2024generalized}.

The rest of the paper will be devoted to adapting SQIsign and its
building blocks to the new framework.  %
In \Cref{sec:sqisign_protocol_overview} we present \SQItorsion, a general template for
instantiating signature schemes from our generalized Deuring
correspondence, and analyze its security properties insofar as they
are independent of the choice of level structure and isogeny-to-ideal
algorithm.  %
In \Cref{sec:klpt} we show that the KLPT
algorithm can be generalized to find connecting ideals with
restrictions defined by level structures. %
This lets us instantiate protocols using exclusively isogenies of
elliptic curves (aka, one-dimensional isogenies), in the style of
pre-SIDH-attacks SQIsign. %
Although modern SQIsign has moved away from one-dimensional isogenies
for efficiency, these algorithms are still very much relevant for
SQIsign-like ring signatures~\cite{CiC:BorLaiLer24} and chameleon
hashes~\cite{forensic}. %
Finally, in \Cref{sec:hd} we analyze in detail a modern instantiation
of \SQItorsion, using higher dimensional isogenies, and compare it to
the latest version of the NIST candidate SQIsign.

\section{Level structures and the Deuring Correspondence}
\label{sec:level_all}
We start by studying how isogenies map torsion subgroups of elliptic
curve, then enlarge the perspective to isogenies mapping level
structures, and finally generalize the Deuring correspondence to them.

\subsection{From ideals to matrices via isogenies}
\label{sec:id-isog-mat}

Fix a positive integer $\level$, let $\Zmod := \Z/N\Z$, with group of
units ${(\Zmod)}^*$.  For $E$ an elliptic curve defined over a field
of characteristic not dividing $\level$, a basis of the $N$-torsion
group $E[N]$ is any ordered pair of points $\B = (P, Q)$ that generate
$E[N]$.  Then the ring $\Mattwo$ of $2\times 2$ matrices modulo $N$
acts on pairs of elements of $E[\level]$ as follows
\begin{equation}
	\label{eq:GL_2_action}
	\smt abcd \cdot \vv{P}{Q} = \vv{a P + bQ}{ cP + d Q} \ .
\end{equation}

If for every elliptic curve $E$ we make an arbitrary choice of a basis
of $E[N]$, we can associate to any isogeny a matrix as follows.

\begin{definition}\label{def:matrix_representation}
Given a morphism of elliptic curves $\varphi \colon E \to E'$  and bases $(P,Q)$ 
and $(P',Q')$ of $E[\level]$ and $E'[\level]$,
the matrix representation of $\varphi$ on the $\level$-torsion 
is the matrix  $\tomatrix{\varphi}{\level}\in \Mattwo$  such that\footnote{To simplify notation we omit the
	dependence on $\level$ and the bases $(P,Q)$ and $(P',Q')$.} 
\begin{equation}
    \label{eq:matrix_representationnn}
    \vv{\varphi(P)}{\varphi(Q)} = \tomatrix{\varphi}{\level} \cdot \vv{P'}{Q'} \ 
    .
\end{equation}
\end{definition}

The map $\varphi\mapsto\tomatrix{\varphi}{\level}$ defines a
homomorphism of abelian groups, and in the case of supersingular elliptic
curves it is surjective.

\begin{lemma}
    \label{lem:surjectivity_torsion}
    Let $\level$ be a positive integer, let 
    $E, E'$ be supersingular elliptic curves in characteristic $p\nmid N$.
    Then the natural map 
    \begin{equation} \label{eq:hom_reductionn}
         \Hom(E,E') \lto \Hom(E[\level],E'[\level]) 
    \end{equation}
    is surjective. In other words if we choose bases $(P,Q)$, $(P',Q')$ 
     of $E[\level]$, $E'[\level]$ 
    the map $\varphi \mapsto \tomatrix{\varphi}{\level}$ is surjective.
\end{lemma}
  
\begin{proof}
	The kernel of the map~\eqref{eq:hom_reductionn} is the set of morphisms 
	$\varphi \colon E \to E'$ such that $E[N] \subset \ker\varphi$, i.e. the maps $\varphi \in N  \Hom(E,E')$ by~\cite[Chapter III, 
	Corollary 4.11]{silverman2009arithmetic}. After 
	quotienting by the kernel, the  map~\eqref{eq:hom_reductionn} becomes an 
	injection
	\[
	\qquad \frac{\Hom(E,E')}{N \Hom(E,E')} \hookrightarrow \Hom(E[N],E'[N]) 
	\cong \Mattwo  \,.
	\]
	But for supersingular curves $\Hom(E,E')$ is a rank~$4$ lattice, hence all the above groups have order equal to $N^4$ and the map is surjective.
        The last part of the statement is simply an explicit choice for the map following from having chosen bases of $E[\level]$ and $E'[\level]$.
        \qed
\end{proof}

This homomorphism enjoys several other properties:
\begin{itemize}
\item It is contravariant with respect to composition: given
  $\varphi \colon E \to E'$ and $\psi \colon E' \to E''$ we have
  $\tomatrix{\psi \circ \varphi}{\level} = \tomatrix{\varphi}{\level}
  \cdot \tomatrix{\psi}{\level}$.
\item If the bases have the same Weil pairing $e(P,Q) = e(P',Q')$,
  then $\det\tomatrix{\varphi}{\level} = \deg\varphi \pmod N$; in
  particular this is the case if $\varphi$ is an endomorphism.
\item If $\dual\varphi$ is the dual of $\varphi$, then
  $\tomatrix{\dual{\varphi}}{\level} \cdot \tomatrix{\varphi}{\level}
  = \deg(\varphi) \cdot \Id$; in particular, if
  $\deg\varphi = \det\tomatrix{\varphi}{\level} \pmod N$, then
  $\tomatrix{\dual\varphi}{\level} = \tomatrix{\varphi}{\level}^\adjoint$, the adjoint of
  $\tomatrix{\varphi}{\level}$.
\end{itemize}

We would like to extend this mapping to the ideals defined in
\Cref{th:deuring-corr}, however we must be careful because ideals only
correspond to isogenies up to isomorphism, thus even after fixing
$(P,Q)$ and $(P',Q')$ the action of an ideal is only defined up to
automorphisms of $E'$. Keeping this ambiguity in mind, we associate to an integral ideal one among the possible corresponding isogenies. Having made this choice for one ideal is enough to fix the choice for all equivalent ideals. Indeed looking at \Cref{eq:deuring-corr}, if $I = I_\varphi$ is associated to $\varphi \colon E \to E'$ then we have maps 
\begin{equation}
	\begin{array}{ccccc}
	\Hom(E,E') &\overset{\xi_I}{\longleftrightarrow} & I = I_\varphi & \overset{\chi_I}\lto &\{\text{ideals equivalent to }I\}
	\\
\xi_I(\alpha) = \frac{1}{\deg \varphi}  \varphi \circ \dual {\alpha}  & \longleftrightarrow & \alpha & \longmapsto & \chi_I(\alpha) = \frac{I\conjugate{\alpha}}{\norm(I)} \ .
	\end{array}
	\label{eq:Xi}
\end{equation}
It is easy to check that $\xi_I(\alpha)$ is one of the (not so many) isogenies associated to $\chi_I(\alpha)$. 
Then, with a small abuse of notation, we extend the correspondence to matrices: given an integral ideal $J = \chi_I(\alpha)$ equivalent to $I$, we associate to it
\begin{equation}\label{eq:Chi_change_matrix}
\tomatrix{J}{\level} := 
\tomatrix{\xi_I(\alpha)}{\level} = 
\frac{\tomatrix{\alpha}{\level}^{\adjoint}
	\tomatrix{\varphi}{\level}}{\deg\varphi}
\ .%
\end{equation}
Notice moreover that the map $\chi_I$ is defined also when $I$ is any fractional ideal and $\alpha$ is any quaternion. 
We summarize this in \Cref{alg:compute_solutions}, which solves the
special case where we look for an equivalent ideal mapping $(P,Q)$ to
$(P',Q')$.

\begin{algorithm}
    \caption{$\ComputeSolutions$}
    \label{alg:compute_solutions}
    \begin{algorithmic}[1]
      \Require Supersingular curves $E,E'$ with bases $\langle P,Q\rangle = E[N]$ and $\langle P',Q'\rangle = E'[N]$;
      \Require A basis $\langle \alpha_1, \ldots, \alpha_4 \rangle$ of $\OrderEnd(E)$ and the matrices $\tomatrix{\alpha_1}{\level}, \ldots, \tomatrix{\alpha_4}{\level}$ wrt $(P,Q)$;
      \Require An ideal $\generalIdeal$ associated to an isogeny $\zeta:E\to E'$ of degree coprime to $N$;
      \Require The matrix $\tomatrix{\zeta}{\level}$ wrt $(P,Q)$ and $(P',Q')$;
      \Ensure $I_\rho\sim \generalIdeal$ such that $M_{I_\rho} = \Id$.
      \State Let $\langle \beta_1, \ldots, \beta_4\rangle$ be a basis of $I_\zeta$,
      compute $\tomatrix{\beta_1}{\level}, \ldots, \tomatrix{\beta_4}{\level}$ by linearity;
      \State
      Find by linear algebra $n_1,\ldots n_4 \in \Zmod$ s.t.
      $n_1\tomatrix{\beta_1}{\level} + \cdots + n_4
      \tomatrix{\beta_4}{\level}
      = \frac{\norm(\generalIdeal)}{\det M_\generalIsogeny}
      \tomatrix{\generalIsogeny}{\level}$;
        \label{line:linear_system}
        \State Set $\alpha = n_1 \beta_1 + \cdots + n_4 \beta_4$; 
        \label{line:i0_def}
        \State\Return  $I_\rho \gets  \chi_{\generalIdeal}(\alpha)$. \label{line:ideal} 
    \end{algorithmic}
\end{algorithm}

\begin{remark}
  \label{rmk:basis_action}
  \Cref{alg:compute_solutions} requires as input a basis of
  $\langle\alpha_1,\dots,\alpha_4\rangle$ for which we know the
  corresponding $\tomatrix{\alpha_i}{\level}$. %
  If $\level$ is smooth, then the $\alpha_i$ can be arbitrary, as long
  as they can be evaluated: the matrices are then computed through a
  generalized discrete logarithm.
  
  When $\level$ contains a large prime factor, instead, we 
  construct the basis along with $(P,Q)$.
  For special orders, such as $\OrderEnd(E_0)$,
  we look for two endomorphisms $\alpha, \beta$ such 
  that
  \begin{enumerate}
  \item $1,\alpha,\beta, \beta \alpha$ generate 
    $\End(E_0)/\level\End(E_0)$ as a $\Zmod$-module,
  \item we know $P \in E_0[\level]$, an eigenvector of $\alpha$ with
    non-zero eigenvalue.
  \item $\langle P, \beta(P) \rangle$ form a basis of $E_0[N]$.
  \end{enumerate}
  Then the matrices
  $\tomatrix{\alpha}{\level}, \tomatrix{\beta}{\level},
  \tomatrix{\beta\alpha}{\level}$ on the basis $(P, \beta(P))$ follow by direct computation.  
  When $E_0 : y^2 = x^3 + x$, we can take
  $\beta = \iota = \sqrt{-1}$ and $\alpha = \pi + n \iota$, with $\pi$
  the Frobenius endomorphism and $n$ integer such that $(-p - n^2)$ is
  a quadratic residue modulo $\level$, so that
  $\sqrt{-p - n^2} \pmod{N}$ is an eigenvalue of $\alpha$.

  For other orders, the information can be transported via an isogeny
  from $E_0$, as detailed in \Cref{rmk:computation_torsion_matrices}.
\end{remark}

\subsection{Level Structures} \label{sec:level_structures}
\label{sec:homsets_and_lvl_structures_love_story}

Restricting the action of \Cref{eq:GL_2_action} to the group $\GLtwo$ of invertible matrices, we get an action of $\GLtwo$ on the set of bases of $E[\level]$.
Given a subgroup $\Gamma$  of $\GLtwo$  a \textit{$\Gamma$-level 
structure}, or simply $\Gamma$-structure on an elliptic curve $E$ is a basis for 
$E[\level]$ up to change of basis by matrices in $\Gamma$: given a basis $(P,Q)$ of 
$E[\level]$, the $\Gamma$-level structure containing it is 
\begin{equation}    
    \label{eq:level_structure_orbit}
    \Gamma \cdot \vv{P}{Q} = \left\{ \Mat \cdot \vv{P}{Q}  \mid  \Mat \in \Gamma \right\} \ .
\end{equation}
The integer $\level$ is referred to as the \emph{level}.  

For example, the bases of the $E[\level]$ are $\{ \Id \}$-structures,
we call these \emph{full-level structures}. In this work we focus on
the following level structures:
\begin{itemize}
    \item \textit{Borel level structures} are the level structures relative to the Borel subgroup
    \begin{equation}
    	\Gamma_0 = \left\{ 
            \smt ab0d
    	\big| a\cdot d \in {(\Zmod)}^*
    	, b \in \Zmod \right\} \ .
    \end{equation}
    
   Borel structures of level $\level$ on $E$ are in bijection with cyclic subgroups 
   $C \subset E$ of order $N$: a Borel structure $\Gamma_0  \cdot \vv{P}{Q}$ corresponds to $C = \langle Q \rangle$ since 
    $\Gamma_0  \cdot \vv{P}{Q}$ is exactly the set of bases $(P',Q')$ 
   with $Q' \in C$.
\item \textit{Split Cartan level structures}, \ie diagonal level structures, relative to
\[
\Gamma = \left\{ \smt \lambda 0 0  \mu  \mid \lambda,\mu \in {(\Zmod)}^* \right\} \ .
\]
They correspond to ordered pairs of cyclic subgroups whose direct sum is $E[\level]$;
indeed $\Gamma  \cdot \vv{P}{Q}$ is the set of bases $(P',Q')$ 
    with $Q' \in \langle Q \rangle$ and $P'\in \langle P \rangle$.
\item \textit{Scalar level structures}, \ie $Z$-level structures,
    with 
    $$
    	Z = \left\{ \smt \lambda00\lambda \mid \lambda \in {(\Zmod)}^* \right\} \ .
    $$
\end{itemize}

If $\alpha\colon E \to E'$ is an isogeny of degree coprime to $\level$
and $S = \Gamma \cdot \vv PQ$ is a $\Gamma$-structure, we define
\begin{equation}\label{eq:level_structures_and_isogenies} \alpha(S) =  \Gamma
\cdot \vv{\alpha(P)}{\alpha(Q)} \;, \end{equation} which is a well defined
$\Gamma$-structure, in particular the above definition does not depend on the
choice of $\vv PQ \in S$.

The assumption $\gcd(N,\deg\alpha)=1$ is necessary and sufficient to
ensure that $({\alpha(P)},{\alpha(Q)})$ is a basis of $E'[\level]$,
however it will be convenient to extend definitions to isogenies of
degree not coprime to $N$.  We let $\spanned{\Gamma} \subset \Mattwo$
be the set of all $\Zmod$-linear combinations of elements in $\Gamma$,
and we define the set of all linear combinations of pairs of points in
$S = \Gamma \cdot \vv PQ$ as
\begin{equation}
    \label{eq:span_S}
    \spanned{S} =  \spanned{\Gamma}  \cdot \vv PQ  =  \{ \gamma \cdot \vv{P}{Q}  \mid \gamma \in \spanned{\Gamma}
    \} 
     \subset 
    E[\level] \times E[\level] \ ,
\end{equation}
which, again, does not depend on the choice of $\vv PQ \in S$.

\begin{definition}%
    \label{def:supersingular_curve_level_structure_set_2}
    \label{def:hom_level}
    \label{def:order_level_structure}
    Let $\level$ be a positive integer and $\Gamma$ be a subgroup of $\GLtwo$.
    \begin{enumerate}
        \item A (supersingular) elliptic curve with $\Gamma$-level structure is 
        a pair $\curvelevel = (E,S)$ where $E$ is a (supersingular) elliptic curve defined 
        in characteristic not dividing $\level$ and $S$ is a $\Gamma$-level 
        structure on $E$.

        \item Given two elliptic curves with $\Gamma$-level structures $(E,S)$
        and $(E',S')$, the set of homomorphisms between them is  
        \begin{equation}
            \label{eq:Hom_level_structure}
            \Hom((E,S), (E', S')) = \{ \alpha \colon E \to E' : \alpha(S) 
            \subseteq \spanned{S'} \} \ .
        \end{equation}
         We use the notation $\alpha \colon (E,S) \to (E', S')$ for such homomorphisms. Moreover $\alpha$ is an isomorphism between 
        $(E,S)$ and $(E', S')$ if it has degree 1.
        \item The endomorphism ring of an elliptic curve with $\Gamma$-level structure $(E,S)$ is 
            \begin{equation*}
                \OrderEnd(E,S) = \End(E,S) := \Hom((E,S),(E,S)) =  \{
                    \alpha \in \End(E) \mid \alpha(S) \subseteq \spanned{S}
                \} \,.
            \end{equation*}
    \end{enumerate}
\end{definition}

Connecting level structures to the matrix representation of
\Cref{def:matrix_representation}, $\varphi \in \Hom((E,S), (E', S'))$
if and only if $\tomatrix{\varphi}{\level} \in \spanned{\Gamma}$ for
an arbitrary choice of representatives of $S,S'$, and
$\varphi(S) = S'$ if and only if
$\tomatrix{\varphi}{\level} \in \Gamma$.

\begin{remark}
  \label{rmk:HomAndSpanGamma}
  From the characterization in terms of representation matrices, we
  notice that $\Hom((E,S), (E', S'))$ only depends on
  $\spanned{\Gamma}$, in particular $\Gamma$ and
  $\spanned{\Gamma} \cap \GLtwo$ give the same $\Hom$. This justifies
  the following definition.

  \begin{definition}
    \label{def:add-stable}
    $\Gamma \subset \GLtwo$ is \emph{additively stable} if
    $\Gamma = \GLtwo \cap \spanned{\Gamma}$.
  \end{definition}

  The scalar, Borel and Cartan subgroups are all additively stable, but
  the trivial group or $\Gamma_1$ (see~\cite[Section
  5.4]{EC:DeFFouPan24} for a definition) are not. %
  Moreover if $\Gamma$ is additively stable, an isogeny
  $\varphi \in \Hom((E,S), (E', S'))$ has degree coprime to $N$ if and
  only if it satisfies $\varphi(S)= S'$, which is stronger than
  $\varphi(S)\subset \spanned{S'}$.
\end{remark}

\Cref{lem:surjectivity_torsion} implies that, given $\bbE = (E',S')$,
$\bbE' = (E',S')$ supersingular elliptic curves with
$\Gamma$-structure, $\Hom(\bbE, \bbE')$ is not empty.  It is, in fact,
a sublattice inheriting the usual scalar product, norm and length from
$\Hom(E,E')$, namely
\begin{equation}\label{eq:scalar_product_norm_length}
	\langle \alpha,\beta \rangle = \tfrac 12 \tr(\hat \alpha\circ \beta) \,, \quad 
	\Norm(\alpha) = \deg(\alpha)\,, \quad |\alpha| = \sqrt{\deg\alpha} \;. 
\end{equation}

\begin{corollary}
    \label{cor:HomLarge}
    \label{lem:level_structure_isogeny_dual}
            \label{cor:end_sub_order}
    Let $p$ be a prime, $\level$ be a positive integer not divisible by $p$ and 
    $\Gamma$ a subgroup of $\GLtwo$.
    Let $c_\Gamma = [M_2(\Zmod) : \spanned{\Gamma}]$.
    Let $\curvelevel$ and $\curvelevel'$ be supersingular elliptic curves with 
    $\Gamma$-structure in characteristic $p$.
    \begin{enumerate}
        \item \label{item:Hom_index}
         $\Hom(\curvelevel, \curvelevel') $ is a
            (non-empty) $4$-dimensional sublattice of $\Hom(E,E')$ of index 
            equal to $ c_\Gamma $ and covolume $ \frac{p}{4} \cdot 
            c_\Gamma  $
        \item If $\varphi \in \Hom(\curvelevel,\curvelevel')$ then $\hat{\varphi} \in 
            \Hom(\curvelevel',\curvelevel)$. Also, if $\Gamma$ contains the group of scalar 
            matrices, $\varphi(S) = S'$ if and only if $\hat \varphi(S') = S$.
            \label{item:dual_in_hom}
        \item 
			For each positive 
			integer $M$, the lattice $\Hom(\curvelevel, \curvelevel') $ contains an element of degree prime to $M$.
			\label{item:Hom_coprime}
        \item The set $\OrderEnd(\curvelevel)$ is a subring 
            of $\End(E)$ of finite index $c_\Gamma$ closed under dual; in particular it is an order.
            \label{item:end_order}
        \item There is always an isogeny 
            $\varphi \in \Hom(\curvelevel, \curvelevel')$ of degree less than or equal to
            $2\sqrt{2 p c_\Gamma }/\pi$.
            \label{item:small_isogeny}
    \end{enumerate}
\end{corollary}
\begin{proof}
  The lattice $\Hom(\bbE, \bbE')$ is the kernel of the surjective map
  $\Hom(E,E') \to \Mattwo/\spanned{\Gamma}$, thus it is a sublattice
  of $\Hom(E,E')$ of finite index $c_\Gamma$.
  Statement~\ref{item:Hom_index} follows from $\Hom(E,E')$ having
  covolume $\tfrac p4$.  Statements~\ref{item:dual_in_hom}
  and~\ref{item:Hom_coprime} follow from
  \Cref{lem:surjectivity_torsion} and \Cref{rmk:HomAndSpanGamma}.
  Statement~\ref{item:end_order} is then immediate.  The last
  statement follows from Minkowski's theorem.
  \qed
\end{proof}

\subsection{Generalising the Deuring correspondence}
\label{sec:order_and_lvl_structures_love_story}

We now extend the Deuring Correspondence to the context of
supersingular elliptic curves with level structures.
This also generalizes work of Sarah Arpin done 
in~\cite{arpin-sarah-borel-lvl-structures} for Borel level structures.

Given a supersingular elliptic curve with $\Gamma$-structure $\curvelevel$ the natural 
analogue of the  maximal order $\calO = \End(E)$ is the subring 
$\OrderEnd(\curvelevel) 
= \End (\curvelevel) $ in \Cref{def:hom_level}, 
that is again an order in $\Bpinf$ by 
\Cref{cor:end_sub_order}. 

\begin{remark}
	\label{rmk:order_not_unique}
	In general, when considering non-trivial level structures, there is not a 
	bijection between supersingular elliptic curves with $\Gamma$-level 
	structure and orders in $\Bpinf$. In fact, if there exists a matrix $\Mat\notin \Gamma$ in the normalizer of 
	$\Gamma$ in $\GLtwo$, i.e.
	$\Mat^{-1}\Gamma \Mat = \Gamma$, then it is immediate to check that 
	$\OrderEnd(E,S) = \OrderEnd(E, \Mat S)$ 
 	 even if usually\footnote{e.g. if $j(E)\neq 0,1728$ and $M \neq -\mathrm{Id}$} $(E,S)$ and $(E, \Mat S)$ 
	are not isomorphic.
\end{remark}

The relevant ideals, in the context of supersingular elliptic curves with level 
structure are the following.
\begin{definition}	\label{def:ideal_level_structure}
    Let $\bbE, \bbE'$ be as in  \Cref{lem:level_structure_isogeny_dual}. 
    Given  $0 \neq \rho \in \Hom(\bbE,\bbE')$, let 
    \begin{equation}
      \label{eq:PGLideals}
        I_{\rho} := I_{\rho,\bbE'} := \{ \psi \circ \rho : \psi \in \Hom(\curvelevel', 
        \curvelevel) \}  \subset \OrderEnd(\curvelevel) \ .
    \end{equation}
\end{definition}

Since $\Hom(\curvelevel',\curvelevel)$ is a left $\OrderEnd(\curvelevel)$-module by postcomposition, 
then $I_{\rho,\curvelevel'}$ is a left ideal of $\OrderEnd(\curvelevel)$. 
Not all $\OrderEnd(\bbE)$-ideals can be obtained as in \Cref{def:ideal_level_structure}, only locally principal 
ones, i.e. ideals $I \subset \OrderEnd(\bbE)$ such that, for all primes $\ell \in \Z$ the ideal $I\otimes \Z_\ell$ of $\OrderEnd(\bbE) \otimes \Z_\ell$ is generated by one element (a heuristic explanation is that $\OrderEnd(\bbE) \otimes \Z_\ell$ and $\OrderEnd(\bbE') \otimes \Z_\ell$ are isomorphic, and we can interpret $\rho$ as a local generator, for the rigorous proof see \Cref{thm:Deuring_with_level_structure}). 
For example if $\Gamma$ is the group of scalar matrices, then $N \End(E)$ is non-principal as $\OrderEnd(\bbE)$-ideal for $N>1$, and it is not of the form \eqref{eq:PGLideals}.

If we change $\rho$ with $\rho'\colon \curvelevel \to \curvelevel'$, the 
ideal changes to an equivalent one: 
\begin{equation}\label{eq:equivalent_ideals}
	I_{\rho',\curvelevel'}  = I_{\rho,\curvelevel'} \cdot \tfrac{\hat\rho \circ 
		\rho'}{\deg\rho} \,.
\end{equation}  
The converse is not true in general: there may be equivalent ideals
corresponding to different codomains $\bbE'$ (see
\Cref{rmk:HomAndSpanGamma}), unless
$\Gamma$ is additively stable, in which case we write
$[I_{\curvelevel'}]$ for the ideal class of $I_{\rho,\curvelevel'}$.

Up to these two observations, the Deuring correspondence generalises, as summarized in \Cref{tab:deuring}. 

\begin{theorem}
	\label{thm:Deuring_with_level_structure}
	Let $p,\level,\Gamma, \bbE, \bbE'$ be as in \Cref{lem:level_structure_isogeny_dual}.
		For each $\rho\colon \curvelevel \to \curvelevel'$ the ideal $I_{\rho,\curvelevel'}$ is 
		locally principal, with norm  
		\[
		\Norm(I_{\rho,\curvelevel'}) :=
		\gcd \left( \{ \deg(\alpha) \mid \alpha \in 
		I_{\rho,\curvelevel'}\} \right)
		= \sqrt {[\OrderEnd(\curvelevel) : 
			I_{\rho,\curvelevel'}]}   = \deg(\rho)  \ .
		\]
Moreover, 
if we denote by $\overline{\cdot}$ the conjugation in $\Bpinf$ and we identify $\OrderEnd(\bbE')$ with
$$\tfrac1{\deg \rho}\dual\rho \circ
\, \OrderEnd(\bbE')  \circ\rho \quad\subset\quad \OrderEnd(\bbE) \otimes \Q \quad\cong\quad \Bpinf \ ,$$
then $I_{\dual\rho, \bbE} = \overline{\ol{I_{\rho,\bbE}}}$.
Under the same identification, if $\tau \colon \bbE' \to \bbE''$ is another non-zero morphism, then 
$I_{\tau \circ \rho, \bbE''} = I_{\rho, \bbE'} \cdot I_{\tau, \bbE''}$.
\end{theorem}

\begin{proof}  
	Local principality is invariant if we replace $I_{\rho, \bbE'}$ with an 
	equivalent ideal. In particular, using \eqref{eq:equivalent_ideals} we can 
	change $\rho$ to another element of $\Hom(\bbE,\bbE')$ and, using 
	\Cref{cor:HomLarge}, suppose that $\deg(\rho)$ is coprime to $N$. In 
	particular if $\ell$ divides $\deg \rho$, then it does not divide the index 
	$[\End(E): \OrderEnd(\bbE)]$, nor $[\Hom(E,E'):\Hom(\bbE,\bbE')]$, hence 
	$\OrderEnd(\bbE) \otimes \Z_\ell$ is equal to the maximal order 
	$\End(E)\otimes \Z_\ell$ and  
	\begin{equation}
		\label{eq:ideal_local_classical}
		I_{\rho,\bbE'} \otimes \Z_\ell =  \left(\widehat{\Hom(\bbE,\bbE') } 
		\circ \rho \right) \otimes \Z_\ell
		= \left(\widehat{\Hom(E,E') } \circ \rho \right) \otimes \Z_\ell
	\end{equation}
	which is principal since it is the local version of the ideal  
	$\widehat{\Hom(E,E') } \circ \rho$ given by the classical Deuring 
	correspondence. For the primes $\ell$ that do not divide $\deg \rho$, by 
	the last part of \Cref{cor:HomLarge}
	we can choose  an element $\alpha$ of $\Hom(\bbE,\bbE')$ of degree coprime 
	to $\ell$, so that $\hat \alpha \rho \in I_{\rho,\bbE'}$ also has degree 
	prime to $\ell$, implying that $I_{\rho,\bbE'} \otimes \Z_\ell = 
	\OrderEnd(\bbE)\otimes\Z_\ell$, which is principal. 
	
	By definition $\Norm(I_{\rho. \bbE'})$ is the $\gcd$ of  
	$\deg(\rho)\deg(\alpha)$ for all $\alpha$ in 
	$\Hom(\bbE,\bbE')$, which is equal to $ \deg(\rho)$ by
	 \Cref{cor:HomLarge}, \Cref{item:Hom_coprime}.
	
	To compute the index of $I_{\rho,\bbE'}$ we can change the ideal to an 
	equivalent one using \eqref{eq:equivalent_ideals} and suppose that $\deg \rho$ 
	is coprime to $N$ (to see how the index changes, recall that for $\alpha \in \OrderEnd(\bbE)$, multiplication by $\alpha$ has 
	determinant $\deg(\alpha)^2$, hence $\alpha I$ 
	has index $\deg(\alpha)^2$ inside $I$). Then by the same arguments used to prove local 
	principality we see that all the primes $\ell$ not dividing $\deg \rho$ do not 
	divide the index $[\OrderEnd(\bbE):I_{\rho,\bbE'}]$, while the primes $\ell$ dividing $\deg \rho$ divide the index with the right multiplicity, by
	\eqref{eq:ideal_local_classical} and the classical case.

    The characterization of $I_{\dual\rho}$ is a direct application of 
    \Cref{def:ideal_level_structure}, using that taking the dual in $\End(E)$ 
    corresponds to conjugation in $\Bpinf \cong \End(\bbE) \otimes \Q$. For the 
    last statement, the  inclusion $I_{\rho, \bbE'} \cdot I_{\tau, \bbE''} 
    \subset I_{\tau \circ \rho, \bbE''}$ can be checked directly, and then it is 
    enough to prove that both sides have the same index in $\End(\bbE)$. This is 
    true since $I_{\rho, \bbE'}, I_{\tau, \bbE''}$ are locally principal, hence 
    arguing locally, in our situation the index is multiplicative.
        \qed
    \end{proof}

%

%

%
%
%
%
%
%
%
%
%
%
%
%
%
%
%
%
%
%
%

%
%
%
%
%
%
%
%
%
%
%
%
%
%
%
%
%
%
%
%
%
%
%
%
%
%
%
%
%
%
%
%
%
%
%
%
%
%
%
%
%
%
%
%
%
%
%
%
%
%
%
%
%
%
%
%
%
%
%

A useful consequence of \Cref{thm:Deuring_with_level_structure} is that, when $\rho \colon \bbE \to \bbE'$ is an isogeny of degree coprime to the level $N$, then we have
\begin{equation}\label{eq:IcapO}
	I_{\rho,\bbE'} = I_\rho \cap \OrderEnd(\bbE) \ ,
\end{equation}
where $I_{\rho}$ is the ideal associated to $\rho$ without considering level structures, as in \Cref{eq:deuring-corr}. Indeed one has
$I_{\rho,\curvelevel} \subset I_\rho \cap \OrderEnd(\curvelevel')$ by definition. To get the equality we compare indices: the index $[\OrderEnd(E):\OrderEnd(\bbE)] = [\Mattwo: \spanned{\Gamma}]$ is a divisor of $N^4$, hence it is coprime to $\deg\rho$ and in particular by  \Cref{thm:Deuring_with_level_structure} both sides in \eqref{eq:IcapO} have index $[\Mattwo: \spanned{\Gamma}] \cdot \deg \rho$ in  $\OrderEnd(E)$.

In the following theorem we state the other direction of the Deuring correspondence, \ie from ideals to elliptic curves with level structure.  We give the proof in \Cref{sec:proofs_level_structures}, since the statement is not needed in our applications.

\begin{restatable}{theorem}{DeuringWithLevelStructure}
	\label{thm:Deuring_with_level_structure2}
	Let $p,\level,\Gamma, \bbE, \bbE'$ be as in \Cref{lem:level_structure_isogeny_dual}. Then each locally-principal ideal of $\OrderEnd(\curvelevel)$ is 
	of the form $I_{\rho',\curvelevel'}$ for some $\curvelevel'$ and $\rho'$. 
		
	Moreover, if we suppose that $\Gamma$ is additively stable, then two ideals  $ I_{\rho',\curvelevel'}$, 
		$I_{\rho'',\curvelevel''}$ are equivalent if and only if $\curvelevel'$, $\curvelevel''$ are 
		isomorphic.
\end{restatable}

%
%
%
%
%
%
%
%
%
%
%
%
%
%
%
%
%
%
%
%
%
%
%
%
%
%
%
%
%
%
%
%
%
%
%
%
%
%
%
%
%
%
%
%
%
%
%
%
%
%
%
%
%
%
%
%
%
%
%
%
%
%
%
%
%
%
%
%
%
%
%
%
%
%
%
%
%
%
%
%
%
%
%
%
%
%
%
%
%
%
%
%
%
%
%
%
%

\begin{table}
	\centering
	\begin{tabular}{p{5cm} p{4.5 cm} p{2.4cm}}
		\toprule
		Curves side &  Quaternion side & Reference
		\\ \toprule
		$\varphi : \curvelevel \rightarrow \curvelevel'$ &
		left locally-principal ideal 
		\par\quad $I_{\varphi,\curvelevel'}$ of $\OrderEnd(\curvelevel)$  
		& 
		Theorems \ref{thm:Deuring_with_level_structure}, \ref{thm:Deuring_with_level_structure2}
		\\ \hline
		$\varphi : \curvelevel \rightarrow \curvelevel', \psi : \curvelevel \rightarrow \curvelevel'$ &
		Equivalent ideals $I_{\varphi,\curvelevel'} \sim I_{\psi,\curvelevel'}$  &
		\Cref{eq:equivalent_ideals}
		\\  \hline
		$\curvelevel'$ supersingular elliptic curves with level structure up to isomorphism     &
		$\mathrm{Cl}(\OrderEnd(\bbE)) =  $
		\par\quad $\{\text{locally principal ideal}\}/\text{equivalence}$
		&  \Cref{thm:Deuring_with_level_structure2}
		\\ \hline 
		$ \theta \in \OrderEnd(\curvelevel)$, up to units on the left &
		Principal ideal $\OrderEnd(\curvelevel)\cdot \theta$ 
		& 	\Cref{def:ideal_level_structure}
		\\ \hline
		$\dual{\varphi}$ &
		$\conjugate{I_\varphi}$ 
		&  \Cref{thm:Deuring_with_level_structure}		
		\\ \hline
		$\deg{\varphi}$ &
		$\Norm (I_{\varphi,\curvelevel'})$ 
		&  \Cref{thm:Deuring_with_level_structure}
		\\ \hline
		$\tau \circ \rho : \curvelevel \rightarrow \curvelevel' \rightarrow \curvelevel''$ &
		$I_{\tau \circ \rho} = I_\rho \cdot I_\tau$ 
		&  \Cref{thm:Deuring_with_level_structure}
		\\ 		\hline
		\rowcolor{gray!25}
		Supersingular $j$-invariants $j(E)$&
		Maximal orders in $B_{p,\infty}$ & 
		\\
		\rowcolor{gray!25}
		in $\mathbb{F}_{p^2}$ up to Galois conjugacy &
		$\O \cong \End(E)$ up to isomorphism
		&  
		\\
		\hline
	\end{tabular}
	
	\smallskip
	\caption{Generalization of the Deuring correspondence (see \cite{EC:DLLW23}) for supersingular elliptic curves with additively stable $\Gamma$-structure. The last line, in grey, is stated without level structures since it does not fully generalize}
	\label{tab:deuring}
\end{table}

\subsection{$\OrderEnd(E,S)$ for concrete choices of level structures}
\label{sec:orders_level_structures}

We describe the rings of endomorphisms $\OrderEnd(E,S)$ and their properties for
 Borel, scalar and split Cartan level structures. A summary is given  in \Cref{tab:orders_level_structures}.
Let $\OrderEnd_E = \End(E)$ be the endomorphism ring of the supersingular
elliptic curve $E$.

\paragraph{Borel Level Structures.}
Thanks to~\cite{arpin-sarah-borel-lvl-structures} we know that for Borel level 
structures the set $\OrderEnd(E,S)$ is an Eichler order in the quaternion 
algebra $\Bpinf$ .
As proven in~\cite{AC:DKLPW20}, if $I_S$ is the ideal associated to the 
isogeny with kernel given by $S$, then $\OrderEnd(E,S) = \Z + I_S$.  
For Borel level structures, if $N>2$,
 $\spanned{\Gamma}$ is equal to the 
set of upper triangular matrices, hence $ [M_2(\Zmod):\spanned{\Gamma}] = 
N$ and the covolume of $\Hom((E,S), (E', S'))$ is $\tfrac 14 pN$.

\paragraph{Scalar level structures.}
Given $E,S$ with $S \ni (P_S,Q_S) $ we have 
    $\OrderEnd(E,S) = \Z + \level \OrderEnd_E$,
where $\OrderEnd_E$ is the endomorphism ring of the curve $E$. In fact, whenever 
$\alpha \in \OrderEnd(E,S)$ we have that $\alpha(P_S) = \lambda P_S$ and 
$\alpha(Q_S) = \lambda Q_S$ for some $\lambda \in \Zmod$, that implies that 
$\alpha - [\lambda]$ maps both $P_S,Q_S$ to $0_E$, so $\alpha - [\lambda] \in 
\level\OrderEnd_E$.
We used that for scalar structures $\spanned{\Gamma}$ is the set of all scalar 
matrices. In particular $ [M_2(\Zmod):\spanned{\Gamma}] = N^3$ and the 
covolume of $\Hom((E,S), (E', S'))$ is $\tfrac 14 pN^3$.

\paragraph{Split Cartan level structures.}
Let  $S = \Gamma \cdot (P_S,Q_S)$ be a split Cartan structure on $E$ and
and suppose $N>2$ (otherwise split Cartan is the same as scalar).
Then, for any choice of an endomorphism $\gamma \in \OrderEnd_E$ such 
that\footnote{In other words we want $\gamma|_{E[N]}$ to diagonalize in basis 
$(P_S,Q_S)$.} $\gamma(P_S) = \lambda P_S$ and $\gamma(Q_S) = \mu Q_S$ with $\lambda 
\neq  \mu $ in $(\Zmod)^*$, we have $ \OrderEnd(E,S) = \Z[\gamma] + \level 
\OrderEnd_E$.
Indeed $\tomatrix{\gamma}{(N)} = \smt\lambda00\mu$, hence 
$\Zmod[\tomatrix{\gamma}{(N)}]$ is equal to $\spanned{\Gamma} = \{ \smt 
\ast00\ast\}$. In particular we have $ [M_2(\Zmod):\spanned{\Gamma}] = N^2$ and 
the covolume of $\Hom((E,S), (E', S'))$ is $\tfrac 14 pN^2$.

\begin{table}[h]
    \centering
    \renewcommand{\arraystretch}{1.2}
    \setlength{\tabcolsep}{6pt}
    \begin{tabular}[h]{lcccc}
        Level Structure     & Trivial              & Borel                & 
        Split Cartan         & Scalar \\ 
        \toprule
        $\Gamma$     & $\GLtwo$                & $\smt **0*  $         
            & $\left[ \begin{smallmatrix} * & 0 \\
            0 & * \end{smallmatrix}\right]$
            & $\left[ \begin{smallmatrix} \lambda & 0 \\
            0 & \lambda \end{smallmatrix}\right]$ \\
        \hline
        $\OrderEnd(E,S)$ &
        $\OrderEnd_E$ & $\Z + I_{S}$ 
            & $\Z[\gamma] + \level \OrderEnd_E$ 
            & $\Z + \level \OrderEnd_E$ \\

        \hline
        Covolume &
        $\frac{p}{4}$ & $\frac{p}{4} \cdot \level$ 
            & $\frac{p}{4} \cdot \level^2$ 
            & $\frac{p}{4} \cdot \level^3$ \\
        \bottomrule
    \end{tabular}

    \caption{
        Structure and covolume of the order $\OrderEnd(E,S)$
        for different level structures.
        }
    \label{tab:orders_level_structures}
\end{table}

\section{\SQInstructor's identification protocols}
\label{sec:sqisign_protocol_overview}

We now describe a paradigm for creating \SQIsign-like interactive
identification protocols. %
The general idea is simple and resembles the description of \SQIsign
we gave in \Cref{sec:preliminaries}, however it has several moving
pieces and instantiations may differ a lot. %
In the next sections we will take a closer look at two instantiations:
one using KLPT and 1-dimensional isogenies, the other using
2-dimensional isogenies.

Regardless of the details, the security of all instantiations is
comparable to that of \SQIsign. %
We analyze in this section soundness, which is proven similarly for
all instantiations, while we defer zero-knowledge of each variant to
the respective section.

\subsection{The Protocol} \label{sec:sqisign_protocol_new}

As usual, we choose a prime $p$ and a supersingular elliptic curve
$E_0/\F_{p^2}$ of known endomorphism ring. %
In practice we may take $p = 3 \mod 4$ and $E_0 : y^2 = x^3 + x$. %
Additionally we fix a level $\torsion$ coprime to $p$ and an
additively stable subgroup $\Gamma$ of $\GLtwo$.

The protocol being a sigma-protocol, we describe its phases one by
one.

\paragraph{Key generation.}
Like in \SQIsign, the secret key is a left ideal
$\secretIdeal \subset \OrderEnd(E_0)$, corresponding to an isogeny
$\secretIsogeny : E_0 \to \publicCurve$. %
Then an arbitrary $\Gamma$-structure $S_\pk$ is chosen on
$\publicCurve$. %
The public key is $\curvelevel_\pk := (\publicCurve,S_\pk)$.

The exact way in which $\secretIdeal$ is sampled is not crucial, as
long as recovering it from $\publicCurve$ is hard. %
Using modern techniques~\cite{NISTPQC-ADD-R2:SQIsign24} it can be
taken uniformly at random among the ideals of a fixed prime degree
$> p^2$, ensuring $\publicCurve$ is statistically close to being
uniformly distributed. %
For convenience, we assume $\norm(I_\sk)$ is coprime to $\level$.

Depending on the instantiation, $S_\pk$ may be chosen at key
generation time, or it may be deterministically computed from
$\publicCurve$ using a public algorithm and thus take no space in the
public key.

\paragraph{Commitment} is similar to key generation: a secret ideal
$\commitIdeal$ is sampled and the associated isogeny
$\commitIsogeny : E_0 \to \commitCurve$ evaluated. %
The curve $\commitCurve$ is sent to the verifier as commitment. %
The same sampling algorithm used for $\secretIdeal$ may be used here, in particular we assume $\norm(\commitIdeal)$ is coprime to $\level$.

\paragraph{Challenge.}
The verifier picks a uniformly random $\Gamma$-structure $S_\cmt$ on
$\commitCurve$ and sends it to the prover. %
Let $\curvelevel_\cmt = (\commitCurve, S_\cmt)$.

\paragraph{The response} is an isogeny
$\respIsogeny\in\Hom(\curvelevel_\pk, \curvelevel_\cmt)$ such that
$\respIsogeny(S_\pk)=S_\cmt$.

This is where we encounter the biggest variability: we may opt to
sample responses of smooth degree, so that they can be encoded as
chains of 1-dimensional isogenies like in earlier versions of
\SQIsign~\cite{NISTPQC-ADD-R1:SQIsign23}, or we may let them have
arbitrary degree and encode them as higher-dimensional isogeny
chains~\cite{NISTPQC-ADD-R2:SQIsign24}. %
The latter choice leads to more performant implementations, on par
with the version of \SQIsign currently submitted to NIST, as we will
show in \Cref{sec:hd}. %
But the former is still of interest as it offers the only known way to
instantiate advanced primitives such as ring
signatures~\cite{CiC:BorLaiLer24} or chameleon
hashes~\cite{forensic}. %
We analyze this alternative in \Cref{sec:klpt}.

Irrespective of this choice, the first step is to compute
$\Hom(\curvelevel_\pk, \curvelevel_\cmt)$. %
We will detail this step in \Cref{sec:compute-hom}.

\paragraph{Verification.}
The verifier checks that $\respIsogeny$ is indeed a valid isogeny
$\publicCurve \to \commitCurve$ and that
$\respIsogeny(S_\pk)=S_\cmt$. %
The protocol is summarized in \Cref{fig:triangle}.

\begin{figure}
  \centering
  \begin{tikzpicture}
    \node[anchor=south] (E0) at (0,0){$E_0$};
    \node[anchor=west] (com) at (3,-1){$\commitCurve$,$\textcolor{blue}{\commitLevelStructure}$};
    \node[anchor=east] (pk) at (-3,-1){$\publicCurve$,$\chLevelStructure$};

    \draw[red,-latex,dashed]
    (E0) edge node[auto] {$\commitIsogeny$} (com)
    (E0) edge node[auto,swap] {$\secretIsogeny$} (pk);
    \draw[-latex]
    (pk) edge
    node[above] {$\respIsogeny$}
    node[below] {$\respIsogeny( \chLevelStructure ) = \commitLevelStructure$}
    (com);
  \end{tikzpicture}
  \caption{\SQInstructor identification protocol. Secrets dashed
    in red. Challenge in blue.}
  \label{fig:triangle}
\end{figure}
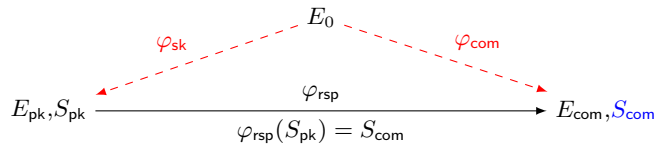

\subsection{Encoding and sampling level structures}
The group $\Gamma$ being a fixed parameter, a level structure $S$ can
simply be encoded by taking an arbitrary representative basis
$(P_S,Q_S)$. %
A random level structure is then drawn by sampling a random basis of
$E[\level]$. %
Some level structures have more compact representations: for example a
Borel level structure is equivalent to a cyclic group of order
$\level$, it may thus be encoded by a single point.

If $(P,Q)$ is some basis of $E[N]$ known to both prover and verifier,
for example because it is computed deterministically from $E$ using a
public algorithm, such
as~\cite[Algorithm~2.1]{NISTPQC-ADD-R2:SQIsign24}, then an alternative
is to represent $(P_S,Q_S)$ by the matrix $M_S\in\GLtwo$ such that
$(P_S,Q_S) = M_S\cdot(P,Q)$. %
	While transmitting $(P_S,Q_S)$ requires $4\log(q)$ bits for a suitable $q = p^r$, the matrix $M_S$ only
	takes $4\log(N)$ bits, which is less since $N^2 = O(q)$ by Hasse inequality and, later, by our choice of parameters. %
Because we are free to act by $\Gamma$, we may even reduce $M_S$ to
some standard form requiring only $\approx 4\log(N) - \log(|\Gamma|)$
bits. %
Sampling a random level structure is then equivalent to sampling a
random element in $\Gamma \backslash \GLtwo$. %

\subsection{Computing $\Hom(\curvelevel_\pk, \curvelevel_\cmt)$}
\label{sec:compute-hom}

The first step in computing the response is computing the module
$\Hom(\curvelevel_\pk, \curvelevel_\cmt)$. %
Thanks to \Cref{thm:Deuring_with_level_structure}, this is isomorphic
to an ideal
$I_{\rho,\curvelevel_\cmt}\subset \OrderEnd(\curvelevel_\pk)$ for some
$\rho:\curvelevel_\pk\to\curvelevel_\cmt$. %
We already know an isogeny
$\commitIsogeny\circ\dual\secretIsogeny:\publicCurve\to\commitCurve$,
we may thus use Algorithm $\ComputeSolutions$ to find
$I_{\rho,\curvelevel_\cmt}$.

\begin{algorithm}
  \caption{\ResponseIdeal}
  \label{alg:response_ideal}
  \begin{algorithmic}[1]
    \Require Curves $E_0,\publicCurve,\commitCurve$, with $N$-torsion bases
    $(P_0,Q_0),(P_\pk,Q_\pk),(P_\cmt,Q_\cmt)$;
    \Require A basis $\langle \omega_1, \ldots, \omega_4 \rangle$ of
    $\OrderEnd(E_0)$ and the matrices
    $\tomatrix{\omega_1}{\level}, \ldots, \tomatrix{\omega_4}{\level}$
    wrt $(P_0,Q_0)$;
    \Require Ideals $\secretIdeal$ and
    $\commitIdeal$ of norm coprime to $N$ s.t.
    $\publicCurve = \secretIdeal\star E_0$ and $\commitCurve = \commitIdeal\star E_0$;
    \Ensure Ideal $I_\rho\subset\OrderEnd(\publicCurve)$ such that
    $\rho(P_\pk) = P_\cmt$ and $\rho(Q_\pk) = Q_\cmt$.
    \State Let $M_\sk$ be the matrix of $\secretIdeal$ in $(P_0,Q_0)$ and $(P_\pk, Q_\pk)$;
    \State Let $M_\cmt$ be the matrix of $\commitIdeal$ in $(P_0,Q_0)$ and $(P_\cmt, Q_\cmt)$;
    \label{alg:ri:m_com}
    \State Let $\langle \beta_1,\ldots,\beta_4\rangle$ be a basis of the right order of $I_\sk$;
    \For {$1 \le i \le 4$}
    \State $\alpha_i \gets \secretIsogeny\beta_i\dual\secretIsogeny/\deg\secretIsogeny$,
    \label{alg:ri:order_end}
    \State $M_{\alpha_i} \gets M_\sk^{-1}(c_1M_{\omega_i} + \cdots + c_4M_{\omega_4})M_\sk$,
    where $\beta_i = c_1\omega_1 + \cdots + c_4\omega_4$;
    \label{alg:ri:order_mat}
    \EndFor
    \State $I_\zeta \gets \conjugate{\secretIdeal}\cdot \commitIdeal$;
    \label{alg:ri:ideal}
    \State $M_\zeta \gets \deg\secretIsogeny\cdot M_\sk^{-1}\cdot M_\cmt$;
    \label{alg:ri:isog_mat}
    \State\Return $I_\rho \gets \ComputeSolutions\bigl((\publicCurve,\commitCurve), (\alpha_i)_i, (M_{\alpha_i})_i, I_\zeta, M_\zeta\bigr)$.
  \end{algorithmic}
\end{algorithm}

With the notation above, and letting $(P_\pk, Q_\pk)$ and
$(P_\cmt, Q_\cmt)$ be arbitrary representatives of $S_\pk$ and
$S_\cmt$, the procedure is summarized in \Cref{alg:response_ideal}. %
To understand why it produces the desired output, observe that
$\OrderEnd(\publicCurve)$ is identified to the right order of
$\secretIdeal$ via
\begin{equation*}
  \begin{aligned}
    \iota_\pk : \OrderEnd(\publicCurve) &\to  \OrderEnd(E_0) \otimes \Q \\
    \alpha &\mapsto \tfrac{1}{\deg\secretIsogeny}\dual\secretIsogeny\alpha\secretIsogeny
             \ .
  \end{aligned}
\end{equation*}
Thus if $\beta_i = \iota_\pk(\alpha_i)$, then
Lines~\ref{alg:ri:order_end}--\ref{alg:ri:order_mat} recover
$\alpha_i$ and its matrix $M_{\alpha_i}$ in the basis
$(P_\pk,Q_\pk)$. %
Note that $M_\sk$ and the denominators of the $c_i$'s are invertible
modulo $\level$ thanks to the assumption
$\gcd(\level,\deg\secretIsogeny)=1$.

Then, at Lines~\ref{alg:ri:ideal}--\ref{alg:ri:isog_mat}, $I_\zeta$ is
the ideal corresponding to $\commitIsogeny\circ\dual\secretIsogeny$
and $M_\zeta$ its associated matrix in $(P_\pk,Q_\pk)$ and
$(P_\cmt,Q_\cmt)$. %
Hence $\ComputeSolutions$ returns an ideal $I_\rho \sim I_\zeta$
mapping $(P_\pk, Q_\pk)$ to $(P_\cmt, Q_\cmt)$ (and thus $S_\pk$ to
$S_\cmt$).

Finally, to represent $\Hom(\curvelevel_\pk, \curvelevel_\cmt)$, we
first compute $\OrderEnd(\curvelevel_\pk)$ using
\Cref{tab:orders_level_structures}, and then we obtain
$I_{\rho,\curvelevel_\cmt} = I_\rho \cap \OrderEnd(\curvelevel_\pk)$
thanks to \Cref{eq:IcapO}.

Pay attention to the fact that the isogeny $\rho$ associated to
$I_\rho$ is not the $\respIsogeny$ we are looking for, but only one in
the same $\Hom(\curvelevel_\pk, \curvelevel_\cmt)$. %
Indeed, having been computed by linear algebra, we cannot argue that
$\rho$ hides $\dual\secretIsogeny\circ\commitIsogeny$, nor do we
control its degree. %
Picking a suitable equivalent ideal to $I_{\rho,\curvelevel_\cmt}$ is
the topic of \Cref{sec:klpt,sec:hd}.

\begin{remark}
    \label{rmk:computation_torsion_matrices}
  In the discussion above it was crucial that we could compute
  $M_\cmt$ with respect to $(P_0,Q_0)$ and $ (P_\cmt, Q_\cmt)$ (and similarly for
  $M_\sk$). %
  When discrete logarithms modulo $\level$ can be efficiently
  computed, this can be done simply by evaluating $\commitIsogeny$ on
  $(P_0,Q_0)$ and then computing the discrete logarithm of
  $(\commitIsogeny(P_0),\commitIsogeny(Q_0))$ in base (a
  representative of) $S_\cmt$.

  When discrete logarithms cannot be computed, our options are more
  limited. %
  A workaround is for the prover to compute
  $(\commitIsogeny(P_0),\commitIsogeny(Q_0))$, mask it as
  \[(P_\cmt',Q_\cmt') :=
    M\cdot(\commitIsogeny(P_0),\commitIsogeny(Q_0))\] %
  via a random change of basis $M$, and commit to
  $\commitCurve,(P_\cmt',Q_\cmt')$. %
  Then the challenge $S_\cmt$ can be submitted as a matrix 
  $M' \in \Gamma \backslash \GLtwo$
  , so that $S_\cmt = M' \cdot (P_\cmt',Q_\cmt')$, and the prover is able to
  compute $M_\cmt$.
  This slightly increases the protocol's bandwidth, since the curve needs to be 
  sent together with a basis of its $\level$-torsion,
  but has a negligible impact on performance. %
\end{remark}

\subsection{Security}
\label{sec:soundness_general}

We conclude this section with an analysis of the security of the
identification protocol, insofar as the arguments apply to any
instantiation.

\paragraph{Special soundness}
is proven analogously to \SQIsign: we prove that two responses with
distinct challenges for the same commitment compose to a non-scalar 
endomorphism.
Actually, we can prove the endomorphism provides a 
witness for the stronger relation
\begin{equation}
    \label{eq:rel-oneend}
    \ROneEnd := \{(\bbE,\omega) \;|\; \omega\in\End(E)\setminus\OrderEnd (\bbE) 
    \}\ ,
\end{equation}
for $\bbE$ a supersingular elliptic curve with $\Gamma$-level structure.
Since we assume $\Gamma$ contains the scalar matrices, $\omega$ is not a scalar 
endomorphism. 
So, by~\cite{EC:PagWes24}, $\ROneEnd$ is a hard
relation equivalent to the endomorphism ring problem.

\begin{theorem}
    \label{thm:soundness}
    The protocol in \Cref{sec:sqisign_protocol_new} is $2$-special
    sound for $\ROneEnd$. The knowledge error is
    \begin{equation} \label{eq:soundness_error}
      \varepsilon = [\GLtwo : \Gamma]^{-1} \ .
    \end{equation}
    If an adversary breaks soundness with probability $w$ and expected
    running time $r$ for a uniform public key $\publicCurve$, then
    there is an algorithm computing a witness for $\ROneEnd$ with
    expected running time $O(r/(w - \varepsilon))$.
\end{theorem}
\begin{proof}
    Consider two valid transcripts $(\commitCurve,S_\cmt,\respIsogeny)$ and 
    $(\commitCurve,S_\cmt',\respIsogeny')$ with $S_\cmt \ne S_\cmt'$.
    By definition $\respIsogeny(S_\pk) = \commitLevelStructure$ and 
    $\respIsogeny'(S_\pk) = \commitLevelStructure'$.
    Let $\beta = \dual\respIsogeny \circ \respIsogeny'$ and 
    assume $\beta \in \OrderEnd(\curvelevel_\pk)$, then
    \begin{equation*}
         S_\pk =  \beta(S_\pk) = \dual\respIsogeny \circ \respIsogeny'(S_\pk) =
        \dual\respIsogeny (S_\com ') \ .
    \end{equation*}
    Since $\Gamma$ is additively stable, it is closed under scalar 
    multiplication, thus by \Cref{item:dual_in_hom} from 
    \Cref{cor:end_sub_order} we have $\respIsogeny(S_\pk) = S_\com ' \neq 
    S_\com$, a contradiction.
    So $\beta$ is a witness for $\ROneEnd(\publicCurve)$.
    We conclude that the protocol is $2$-special sound.  
    To compute the knowledge error we
    just need to count the number of cosets of $\Gamma$ in $\GLtwo$,
    that is the index
    $[\GLtwo : \Gamma] = \# \GLtwo/\# \Gamma$.
    
    The second part of the theorem follows by standard techniques, see for 
    example~\cite[Theorem~1 and Definition~2]{sigma_protocols}.
    \qed
\end{proof}

\paragraph{Zero Knowledge.}
To prove zero-knowledge we would need to exhibit a simulator that on
input $E_\pk$ produces a random isogeny
$\respIsogeny':\publicCurve \to \commitCurve'$ and computes
$S_\cmt' := \respIsogeny'(S_\pk)$, so that
$(\commitCurve',S_\cmt',\respIsogeny')$ is distributed identically to a valid transcript.

However the existence of such a simulator depends on how the response
is computed. %
Like in \SQIsign, if we use a KLPT-style algorithm to compute a
1-dimensional response, then the distribution of $\respIsogeny$ is
very poorly understood, and the only thing we can do is argue
computational zero-knowledge. %
We briefly discuss this in \Cref{sec:klpt_rerandomization}. %

If instead we use a higher-dimensional method to produce the response,
we are faced with the obstacle that no polynomial time algorithm is
known to compute a large arbitrary degree isogeny from $\publicCurve$
without knowing its endomorphism ring. %
The workaround used in \SQIsign is to prove security in a model where
the simulator has access to an oracle producing such
isogenies~\cite{NISTPQC-ADD-R2:SQIsign24}, or at least where it is
provided non-interactively with a polynomially sized \emph{hint}
string~\cite{C:ABDPW25}. %
We adapt these proof techniques to \SQInstructor in
\Cref{sec:hd_zk,sec:proofs_sqisign_protocol_new}.

\smallskip

Whatever the case is, if we take Zero-Knowledge for granted, then the
we can apply the Fiat-Shamir transform to the protocol to obtain a
signature scheme as secure as \SQIsign: its \EUFCMA security reduces
to the hardness of $\ROneEnd$, and thus to \Cref{prob:endring}
via~\cite{EC:PagWes24}.

\paragraph{Parameter sizes}
The best classical attacks against \Cref{prob:endring} having
complexity $\tilde{O}(\sqrt{p})$, the bit-size of $p$ should be taken
to be twice the security parameter $\secpar$. %
The primes defined in \SQIsign v1~\cite{NISTPQC-ADD-R1:SQIsign23} or
v2~\cite{NISTPQC-ADD-R2:SQIsign24} are valid choices for a
1-dimensional or a higher-dimensional implementation, respectively.

The level $\level$ may then be taken according to the available
torsion in $E_0(\F_{p^2})$, i.e., according to the prime factors of
$p\pm 1$. %
For a 1-dimensional implementation this may be a product of some
mid-sized factors of $p^2-1$, or a power of a small prime not dividing
the degree of the response. %
For a higher-dimensional implementation it appears to be most
efficient to take a power of $2$.

Finally the size of $\level$ will be determined by the size of the
challenge space: which we just saw it is $[\GLtwo : \Gamma]$.  %
We want this quantity to be $\approx 2^\secpar$. %
Recall that
$\# \GLtwo = \torsion^4 \prod_{p \mid \torsion} (1 - \frac{1}{p^2} )(1
- \frac{1}{p})$, and that
$\#\Gamma = \level \phi(\level)^2,\, \phi(\level)^2,\, \phi(\torsion)$
for Borel, split-Cartan and scalar level structures respectively.
Thus we need
$\log(\torsion) \approx \secpar, \frac{\secpar}{2}, \frac{\secpar}{3}$
respectively.

\section{KLPT for level structures} \label{sec:klpt}

We now come to the problem of finding an element of
$\Hom(\curvelevel,\curvelevel')$ of smooth norm. %
Here we focus solely on extending to the \SQInstructor framework the
KLPT algorithm~\cite{klpt}, which computes a connecting ideal of
smooth norm given $\OrderEnd(E)$ and $\OrderEnd(E')$. %
Then, standard techniques~\cite{NISTPQC-ADD-R1:SQIsign23} let us
translate this ideal into a 1-dimensional isogeny, to be used, e.g.,
as response $\respIsogeny$ in the identification protocol.

These algorithms lead to larger degree isogenies and less efficient
implementations than what is achievable through higher-dimensional
techniques, but, unlike those, produce outputs that can be efficiently
simulated (under an \emph{ad hoc} computational assumption). %
This feature makes 1-dimensional techniques unavoidable in the context
of advanced protocols such as ring signatures~\cite{CiC:BorLaiLer24}
and chameleon hashes~\cite{forensic}.

\paragraph{Notation.} We call integral ideals of $\Bpinf$
ideals for short.
For $I$ a fractional ideal, $\O_L(I)$ is its left order and $\O_R(I)$ its right one.
Two ideals $I,J$ are compatible if $\O_L(J)=\O_R(I)$. 
For compatible ideals $I,J$, the pullback of $J$ through $I$ is $[I]^*J=IJ+\normklpt{J}\O_L(I)$.
For $I,J$ left ideals of a maximal order $\O$, the pushforward of $J$ through $I$ is $[I]_*J=\frac{1}{\normklpt{I}}\bar I (I\cap J)$.
Similar to \cite[Section 2.4]{AC:DKLPW20}, $\O_0=\O(E_0)$ is a special maximal order.

\paragraph{Solving norm equations with level structure.}
In the following, we are given a left $\OrderEnd(\curvelevel)$-ideal
$\LatticeOfSolutions$ defined as in \Cref{eq:PGLideals}, and we want to compute an
equivalent ideal of smooth norm. %
Recall that, in quaternion terms, $\LatticeOfSolutions$ is the intersection of an 
$\OrderEnd(E)$-ideal with the suborder $\OrderEnd(\curvelevel)$, as summarized in 
\Cref{tab:orders_level_structures}.
Thus, for $\O$ a maximal order and  $\O_{\torsion}$ a suborder associated to a level structure, we 
say that an $\O$-ideal $J$ \textit{is equivalent in $\O_{\torsion}$} or 
\textit{is equivalent respecting the level structure associated to 
  $\O_{\torsion}$}
to an input $\O$-ideal $I$ if $J$ is such that 
$J=I\frac{\bar\alpha}{\normklpt{I}}$ with $\alpha$ in $I\cap \O_{\torsion}$.
This also implies that $J\cap\O_\level$ and $I\cap\O_\level$ are
equivalent $\O_\level$-ideals.
We therefore need to solve norm equations \textit{with 
level structure} or \textit{in $\O_{\torsion}$}, which means finding an ideal 
equivalent to the input in $\O_{\torsion}$ and of smooth norm.

\paragraph{Prior work.}
The papers on 1-dimensional \SQIsign~\cite{AC:DKLPW20,EC:DLLW23} propose the algorithms 
\GenericKLPT and \SigningKLPT, which transform a quaternion ideal of 
a maximal order into an equivalent ideal of smooth or even power-of-two norm.
The difference between the two algorithms is that \SigningKLPT outputs an ideal 
following a distribution which is assumed to be computationally 
indistinguishable from uniformly random ideals of the same norm and equivalence 
class.
\GenericKLPT is a simpler version than \SigningKLPT, which leaves out some 
rerandomization procedures to focus solely on smooth-norm ideals in the correct 
equivalence class. It is therefore suitable to first adapt \GenericKLPT, then 
the rerandomization.

In his PhD thesis~\cite{Leroux}, Leroux analyses these KLPT versions and creates 
more variants. \GenericKLPT consists in some equivalence-taking to get prime 
norm inputs, a pullback to $\O_0$ of the input ideal, and an algorithm for 
resolution in a special order $\O_0$.
This last algorithm is called \IdealEichlerNorm~\cite[Algorithm 10]{Leroux} and 
solves (in $\O_0$) a norm equation in the intersection of an ideal $I$ and an 
Eichler suborder $\Z+I_c$, where $I_c$ is a primitive ideal of a maximal order.
While Leroux often assumes $I_c$ to be of prime norm, 
he insists this is not a requirement, but a low number of distinct prime factors 
is better for performance.
Leroux also describes a similar algorithm solving (in $\O_0$) a norm equation in 
the intersection of an ideal $I$ and a suborder $\Z+\torsion I_c$ for some 
integer $\torsion$, called  \IdealSuborderEichlerNorm~\cite[Algorithm 
14]{Leroux}.
A copy of the pseudocode for both of these algorithms is given in 
\Cref{sec:pseudocode}.
In the remainder of this section, we use this \IdealSuborderEichlerNorm and an 
input treatment derived from \SigningKLPT to create an algorithm for smooth-norm 
equivalence with level structure.

\paragraph{Plan of the section.} We focus on the scalar level structure 
because solving this case allows us to solve the problem also for other level 
structures (see \Cref{sec:klpt_others}). 
For simplicity, we restrict the discussion here to input ideals $I$ such that
$\normklpt{I}$ is a square modulo $\torsion$. 
An explanation of this choice and a generalization to the non-square case can 
be found in \Cref{sec:squares}. 

In the following, we show how 
\IdealSuborderEichlerNorm~\cite[Algorithm 14]{Leroux} can be used to solve norm 
equations associated to scalar level structures. 
Then, we adapt the initialization steps
of \GenericKLPT to the scalar case, and evaluate the output length of our 
KLPT variant.
Next, we show how to adapt the rerandomization step of \SigningKLPT to take 
equivalences respecting a level structure.
Pseudocode for the resulting \ScalarSigningKLPT algorithm can be found 
in~\Cref{alg:scalar_klpt}.
Finally, we explain how to solve the non-scalar case, including a more efficient 
implementation for the Borel case using \IdealEichlerNorm.
We also present our proof-of-concept implementation.
For the sake of conciseness we refer to \Cref{sec:proofs_klpt} for the proofs
of the lemmas and propositions stated in this section.

\subsection{The use of \IdealSuborderEichlerNorm in \ScalarSigningKLPT}

We search a $\beta$ in $\LatticeOfSolutions = I \cap (\Z+\torsion \O)$ of the correct norm, then 
return $J=I\frac{\bar\beta}{\normklpt{I}}$.
However, all KLPT-related algorithms solve norm equations only in special orders
like $\O_0$ or suborders of those.
Following prior work such as~\cite{AC:DKLPW20}, we therefore search for $\beta$ 
in the intersection $I\cap (\Z+\torsion \O) \cap \O_0$.

Let us show that \IdealSuborderEichlerNorm can be applied to this problem.
Denote $I_c$ a primitive ideal connecting $\O_0$ to $\O$. Then the Eichler order 
$\O\cap\O_0$ equals $\Z+I_c$. 
Thanks to \Cref{lemma:scalarKLPT}, this means that we search in $I$ intersected with  
$\Z+\torsion I_c$. 
Furthermore, since $\Z+\torsion I_c$ is an order is contained in $\Z+ I_c$, 
the intersection of  $\Z+\torsion I_c$ with $I$ is equal to its intersection with the pullback $[I_c]^*I$.
Since $[I_c]^*I$ is a left $\O_0$-ideal, searching in its intersection with $\Z+\torsion I_c$ 
is exactly the problem solved by \IdealSuborderEichlerNorm.

\begin{restatable}{lemma}{lemmaScalarKLPT}
    \label{lemma:scalarKLPT}
    For $I$ a primitive right ideal of a maximal order $\O$ in \Bpinf.
    Then $(\Z+I)\cap(\Z+N\O)=(\Z+NI)$ for $N\in\Z$ coprime to $\normklpt{I}$.
\end{restatable}

\subsubsection{Creating inputs to \IdealSuborderEichlerNorm}
\label{sec:klpt_inputs}

We now show how to transform the input ideal $I$  (left ideal of a maximal order $\O$) 
into a suitable input for \IdealSuborderEichlerNorm. 
We leave the rerandomization step to \Cref{sec:klpt_rerandomization}.
To ensure efficiency, we need to transform the connecting ideal $I_c$ from $\O$ to $\O_0$
so that it is of prime norm. 
Also, the input ideal $I$ must be pulled back to $\O_0$ through the transformed 
$I_c$ (so that we work at $\O_0$) and replaced by a prime-norm equivalent (again for efficiency).
Finally, none of these steps should change the equivalence class of $I$ in the order 
$\Z+\torsion \O$.

For the connecting ideal $I_c$, we use \EquivalentPrimeIdeal (as 
in~\cite{AC:DKLPW20}) to find an equivalent 
$I_c'=I_c\frac{\alpha_c}{\normklpt{I_c}}$ of prime norm $N_c$.
Conjugating $I$ and $\O$ with $\frac{\bar\alpha_c}{\normklpt{I_c}}$ is an 
isomorphism, and therefore does not modify any problem input. 

Then, since $N_c$ is a large prime, $I$ is (with overwhelming probability)
of norm coprime to $N_c$ and can be pulled back through $I_c'$. 
This results in a left $\O_0$-ideal $I_1$ of norm 
equal to $\normklpt{I}$. To get a prime-norm equivalent without modifying its equivalence 
class with respect to  $\Z+\torsion \O_0$, one uses an element 
$\delta$ in $I_1\cap(\Z+\torsion \O_0)$ with $\frac{\normklpt{\delta}}{\normklpt{I_1}}$ prime.
Therefore $I_0=I_1\frac{\bar\delta}{\normklpt{I_1}}$ is a prime-norm 
input ideal which is in the same class with respect to $\Z+\torsion \O_0$ as $I$ 
is with respect to $\Z+\torsion \O$, as explained in the preceding paragraph.
This results in \Cref{alg:scalar_klpt}.
We set the input exponent $e$ according to the following proposition (proven in 
\Cref{sec:proofs_klpt}).
\begin{restatable}{proposition}{propositionScalarKLPT}
    \label{prop:scalarKLPT}
    There is a variant of \ScalarSigningKLPT which produces output for 
    $\ell^e>p^{9/2} \torsion^{10}\epsilon$ with $\epsilon$ polylogarithmic in 
    $(p\torsion)$, under plausible assumptions.
\end{restatable}

\begin{algorithm}
    \caption{$\text{\ScalarSigningKLPT}_{\ell^e}(I,\torsion)$}
    \label{alg:scalar_klpt}
    \begin{algorithmic}[1]
        \Require $\torsion$ prime or power of prime of known factorization, $I$ 
        left ideal of a maximal order $O$ in $B_{p,\infty}$ such that $
        \normklpt{I}\mod\torsion$ is a square, $\ell$ small prime and
        $e\in\Z$ even such that $\ell^e$ large enough
        (See \Cref{prop:scalarKLPT}).
        \Ensure $J\sim I$ such that $J=\chi_I(\beta)$ with $\beta\in(\Z+\torsion 
        O)\cap I$ and $\normklpt{J}\mid \ell^e$
        
        \State $\alpha,I_c' \gets \EquivalentPrimeIdeal(I_c)$; $I_1 \gets 
        \frac{1}{\alpha} I\alpha$;
        \mycomment{$I_c' = \chi_{I_c}(\alpha)$}
        
        \State 
        $\alpha_I,I_2\gets$\InitialRerandomization$(I_1,\Z+\torsion(\frac{1}{\alpha} 
        O\alpha),I_c')$ \mycomment{See \Cref{sec:klpt_rerandomization}}
        \State Let $\delta\in[I_c']^*I_2\cap (\Z+\torsion \O_0)$ random s.t. 
        $I_0 = \chi_{[I_c']^*I_2}(\delta)$ is of prime norm $N$
        \State $\beta\gets$\IdealSuborderEichlerNorm$_{\ell^e}(\torsion, I_c', 
        I_0)$
        \State $J_0 \gets \chi_{I_0}(\beta)$ and $J_a\gets[I_c']_*J_0$ and 
        $J\gets\alpha J_a \frac{1}{\alpha}$
    \State\Return $J$ 
\end{algorithmic}

\end{algorithm}

\subsection{Rerandomization to improve the output distribution}
\label{sec:klpt_rerandomization}

The difference between \GenericKLPT and \SigningKLPT is that the latter ensures 
that restricting the search space from $I$ to $I\cap(\Z+I_c)$ 
does not bias the output~\cite{AC:DKLPW20,Leroux}.
This is done by replacing the input ideal $I$ ---~before computing the pullback  
through $I_c$~--- by an equivalent ideal 
which is in a uniformly random equivalence class when taking equivalences in 
$(\Z+I_c)$~\cite[Lemma 5.3.1]{Leroux}. 
In the following, the set of ideal classes for a (non necessarily maximal) 
quaternion order $\O$ is denoted Cl$(\O)$, as in~\cite[Section 2.3]{Leroux}.

For rerandomization in \ScalarSigningKLPT, we adapt~\cite[Algorithm 34]{Leroux}, 
but we need to prove that the 
equivalence for $I$ is taken in the order $\O_{\torsion}$ translating the level structure constraint.
This means in $\Z+\torsion \O$ for the scalar case.  As shown in 
\Cref{lemma:klpt_randomization}, it is possible to modify the rerandomization in 
\SQIsign so that the equivalence is taken in $I\cap \O_{\torsion}$ as done in 
\Cref{alg:scalar_rerandomization}.
With this choice of equivalence,~\cite[Proposition 2.3.12]{Leroux} applies, 
so that the class in Cl($\O_0\cap \O$) of the output of this 
rerandomization only depends on the class of $I$ in Cl($\O_{\torsion}$). 
More details on this can be found in \Cref{sec:rerandomization_details}.

\begin{restatable}{lemma}{lemmaKLPTRandomization}
    \label{lemma:klpt_randomization}
    For $I$ a left ideal of a maximal order $\O$ and $\torsion,N_c,\normklpt{I}$ 
    coprime, there exist $\omega\in \Z+\torsion \O$ 
    such that $N_c$ is inert in $\Z[\omega]$ and $\gamma\in I\cap \Z+\torsion \O$ 
    such that $\gamma$ is not in $N_c\O_R(I)$.
\end{restatable}

\begin{algorithm}[H]
    \caption{$\text{\InitialRerandomization}(I,\O_{\torsion},I_c)$}
    \label{alg:scalar_rerandomization}
    \begin{algorithmic}[1]
        \Require $I$ ideal of a maximal order $\O$ in $B_{p,\infty}$, 
        $\O_{\torsion}$ suborder of $\O=\O_L(I)$ of level $\torsion$ and 
        superorder of $\Z+\torsion \O$, $\torsion$ prime or power of prime of 
        known factorization
        \Ensure $I'\sim I$ such that $I'=\chi_{I}(\beta)$ with $\beta\in 
        \O_{\torsion}\cap I$ and the class of $I'$ uniformly random in 
        Cl$(\Z+I_c)$
        \State Get $\omega\in \O_{\torsion}$ s.t. $N_c$ is inert in 
        $\Z[\omega]$;
        \State Get $\gamma\in I\cap \O_{\torsion}$ s.t. $\gamma$ is not in 
        $N_c\O_R(I)$; \mycomment{We assume $N_c$ coprime to $\torsion\normklpt{I}$ as $N_c$ is a large prime}
        \State Let $D\randgets [0,N_c]$, let $C\gets1$; \textbf{if} $D= N$ \textbf{then} $C,D\gets0,1$;
        \State Let $\beta\gets (C+D\omega)\gamma$;
        \State\Return {$\beta,\chi_I(\beta)$};
    \end{algorithmic}
\end{algorithm}

\paragraph{Output distribution.}
As for \SQIsign~\cite{AC:DKLPW20,EC:DLLW23}, the above arguments should not be 
taken as a proof that the output distribution is indistinguishable from 
random ideals of required norm and equivalence class, 
but better proofs are not common for 1-dimensional SQIsign variants.
\SQIsign's authors mostly assume that the output distribution of 
\IdealEichlerNorm is computationally indistinguishable from random ideals 
satisfying the same constraints as its output\footnote{A small bias was 
    noticed in~\cite{EC:DLLW23}, and fixed with small changes to two 
subroutines, \IdealSuborderEichlerNorm already has these changes.}.
Despite these similarities, the arguments on the distribution of outputs of 
\SigningKLPT do not immediately generalize to \Cref{alg:scalar_klpt}.
This is because the algorithms \IdealSuborderEichlerNorm and \IdealEichlerNorm 
are despite all similarities distinct, so the security of 1-dimensional 
\SQItorsion relies again on an \textit{ad hoc} assumption on the output
distribution of \ScalarSigningKLPT.
We claim however that this new assumption is just as plausible 
as the assumption that \SigningKLPT is undistinguishable from uniformly random 
isogenies.

\subsection{Adapting from the scalar to other level structures}
\label{sec:klpt_others}

The above \ScalarSigningKLPT can be adapted to a norm equation solving 
algorithm with other level structures, as long as the corresponding orders 
contain $\Z+\torsion \O$.
Let $I$ the input ideal, a left ideal of a maximal order $\O$ in $B_{p,\infty}$, 
and $\O_{\torsion}$ order translating the level structure constraint, with 
$\Z+\torsion \O \subset \O_{\torsion} \subset \O$. 
One must take special care so that the rerandomized ideal has square norm 
modulo $\torsion$ (see \Cref{sec:squares}).
Then the only modification compared to \ScalarSigningKLPT is that 
$\InitialRerandomization$ is called with the order $\O_{\torsion}$ instead of 
$\Z+\torsion \O$.
This works because all later steps in \ScalarSigningKLPT take an ideal in the 
same class of Cl($\Z+\torsion \O$) as the output of the rerandomization.
Since $\O_{\torsion}$ contains $\Z+\torsion \O$,  ideals which are equivalent for  
$\Z+\torsion \O$ are also equivalent for $\O_{\torsion}$.
The output distribution of this adaptation only depends on the class of the $I$ 
in Cl($\O_{\torsion}$), again by~\cite[Lemma 2.3.12]{Leroux}, and depends on the 
one of \IdealSuborderEichlerNorm just as for \ScalarSigningKLPT.
The required output size is the same as for \ScalarSigningKLPT at the same level 
$\torsion$, given by \Cref{prop:scalarKLPT}.

\subsubsection{More efficient variant for Borel level structure.}

For Borel level structures,  there is an even more efficient approach, which 
allows shorter output sizes.
This \BorelSigningKLPT consists in exactly the same rerandomization (again with 
$\O_{\torsion}=\Z+I_{\torsion}$, and no constraints on the quadratic residuosity of the output norm)
 and input preparation, except that $I_0$ is just 
taken equal to $[I_c]^*I_1$, since \IdealEichlerNorm as stated in 
\cite[Algorithm 10]{Leroux} does not require a prime-norm input. Then, the call 
to \IdealSuborderEichlerNorm is replaced by a call to \IdealEichlerNorm with 
inputs $I_c'I_{\torsion}',I_0$, where $I_{\torsion}'$ is $I_{\torsion}$ 
conjugated by $\frac{\bar\alpha_c}{\normklpt{I_c}}$ (so the product is 
compatible).
The correctness of this algorithm can be seen from \Cref{lemma:borelKLPT}, which 
clarifies that searching an element in $\Z+I_{\torsion}'\cap \Z+I_c'$ is the 
same as searching it in $\Z+I_c'I_{\torsion}'$, which is the problem 
\IdealEichlerNorm solves.

\begin{restatable}{lemma}{lemmaBorelKLPT}
    \label{lemma:borelKLPT}
    For $I,J$ two primitive compatible ideals of $B_{p,\infty}$ (that is 
    $\O_R(I)=\O_L(J)$) of coprime norms, $(\Z+I)\cap(\Z+J)=(\Z+IJ)$.
\end{restatable}

\BorelSigningKLPT has a shorter output size and its output distribution is more 
directly linked to \SigningKLPT.
The output size is due to \IdealEichlerNorm which produces outputs of size 
$p^{3}(\torsion N_c)^3$ if the Eichler order $\Z+I_c' I_{\torsion}$ is used.
The output distribution is similar to the one of \SigningKLPT since 
\IdealEichlerNorm is used in both. The only difference between these 
distributions is that rerandomization in \BorelSigningKLPT is only done up to 
equivalence with respect to $(\Z+I_{\torsion}')$.

\begin{proposition}\label{prop:borelKLPT}
    There is a variant of \BorelSigningKLPT which produces output for 
    $\ell^e>p^{9/2}\torsion^3\epsilon$ with $\epsilon$ polylogarithmic in 
    $(p\torsion)$, under plausible assumptions.
\end{proposition}

\subsubsection{Proof-of-concept implementation.}

We implemented the KLPT variants for scalar and the optimized version for Borel 
level structures\footnote{\klpturl} based on the LearningToSQI~\cite{LearningToSQI} 
SageMath~\cite{sage} implementation of \SQIsign. It is a proof-of-concept 
implementation and has severe limitations.
In particular, it is quite slow. In addition, we did not implement the 
rerandomization from \Cref{alg:scalar_rerandomization} for Borel level 
structures, which means the output distribution is biased in that case.
Finally, we only support $\torsion$ to be a power of an odd prime, and $\ell$ to 
be 2, and the special order to be $\O_0$ (this means that the implementation 
does not adjust the order in case of larger lattice minima, but fails instead).

\subsubsection{Instantiation of \SQItorsion with 1-dimensional isogenies}

Thanks to the previous algorithms we can finally instantiate the 
\SQItorsion
signature scheme using a 1-dimensional representation of 
the response isogeny.
For efficiency, we use a Borel level structure. First, we need a prime $p$ so that ideals of norm $\ell^e$ 
can be translated to isogenies efficiently and 
the \torsion-torsion is accessible\footnote{that means 
the \torsion-torsion is defined over $\mathbb{F}_{p^4}$} 
for some prime (or prime-power) $N$ of bitsize $\secpar$.
Therefore we consider $\ell = 2$ and $p = f 2^e - 1$ such that $p-1$ has a prime factor 
$\torsion$ roughly of size $\secpar$. 
For translating ideals of norm a power of  $\ell$ we can then use the algorithms from~\cite[Appendix A]{EC:Leroux25} or 
those from~\cite{AC:OnuNak24}.\footnote{Both of these algorithms rely on higher 
dimensional isogenies to avoid requirements on the smoothness of $p+1$, but they 
are meant to be used only in the signing phase, without modification of 
the 1-dimensional verification procedure. They don't affect the simulatability of the transcripts.}
A simple search targeting NIST I security gives $p = 369 \cdot 2^{252} - 1$, where $p-1$ has a factor 
$\level$ of size $126$ bits.
With these choices, using \BorelSigningKLPT we obtain a response isogeny
of degree $2^e$ for $e = 9/2 \log(p) + 3\log(\torsion) + \log(\log(p\torsion)) 
\approx 1920$.
This gives a public key size of $186$ bytes, composed by \publicCurve,\Ppk, 
\Qpk\footnote{we can define the points over the twist to have the represented by 
\Fpsq coefficients}
and a signature size of $300$ bytes, composed by $\gamma\in\Gamma$, \chMatrix 
(compressed to $2\log(\torsion)$ bits) and \respIsogeny
(compressed to $e$ bits using techniques 
from~\cite{AC:DKLPW20,EC:CEMR24}).

\section{A state-of-the-art \SQInstructor signature} \label{sec:hd}

In this section we instantiate the \SQInstructor signature protocol
using the same state-of-the-art techniques
in~\cite{NISTPQC-ADD-R2:SQIsign24}, in particular 2-dimensional
isogeny representations.  %
We carefully analyze the Zero-Knowledge property of the identification
scheme and prove \EUFCMA of the signature scheme.  %
Finally, we implement our signature by forking the reference \SQIsign
implementation available at
\url{https://github.com/SQISign/the-sqisign} and benchmark it,
demonstrating comparable performance.%
\footnote{To gain even more performance, we could use the recent
  technique Qlapoti~\cite{qlapoti}, which applies equally well to
  \SQIsign and \SQInstructor. %
  We refrain from doing so to make the comparison with \SQIsign's
  reference code simpler.}

\subsection{The protocol in detail}

\paragraph{Parameters.}
We take a prime $p = f 2^e - 1$ such that $f$ is a small odd cofactor,
and a starting curve $E_0:y^2=x^3+x \;/\F_{p^2}$. %
Any of the primes specified in~\cite{NISTPQC-ADD-R2:SQIsign24}
works. %
We take the level to be $\torsion = 2^{\etorsion}$ for some
$\etorsion \le e$, owing to the fact that $E_0[2^e]$ is
$\F_{p^2}$-rational. %
To simplify the exposition we focus on $Z$-level structures, where
$Z = \{ \smt \lambda 0 0 \lambda \mid \lambda \text{ odd} \} \subset
\GLtwo$ is the subgroup of scalar matrices, thus, following the
analysis of \Cref{sec:soundness_general}, we take
$\etorsion \approx \frac{\secpar}{3} \approx \frac{e}{6}$.

\paragraph{Representing level structures.}
For compactness, we represent level structures on a curve $E$ as
matrices in $\GLtwo/Z$, with respect to a fixed public basis
$(P_E,Q_E)$ of $E[\level]$. %
Such bases can be computed
using~\cite[Algorithms~2.1-3]{NISTPQC-ADD-R2:SQIsign24} and then
multiplying by $2^{e-\etorsion}$.

Letting $\O_0$ be the maximal order with basis
$(1, i, (i+j)/2, (1+k)/2)$, isomorphic to $\OrderEnd(E_0)$ via
$i \mapsto \iota$ and $j \mapsto \pi$, we precompute the matrices
$(M_0^{(1)},\ldots,M_0^{(4)})$ of the action of the order's basis on
$(P_{E_0},Q_{E_0})$ as described in \Cref{rmk:basis_action}.

\paragraph{Key Generation and commitment}
are similar to
\SQIsign~\cite[Algorithm~4.1]{NISTPQC-ADD-R2:SQIsign24}: we generate
secret ideals $\secretIdeal,\commitIdeal\subset\O_0$ of prime norm
$\approx p^2$ and we evaluate the corresponding isogenies
$\secretIsogeny,\commitIsogeny$ to obtain the public key / commitment
curves $\publicCurve,\commitCurve$. %
The level structure $S_\pk$ is then simply
$Z\cdot(P_{\pk},Q_{\pk})$, with $(P_{\pk},Q_{\pk}) = 
(P_{\publicCurve},Q_{\publicCurve})$ the fixed level 
structure associated to $\publicCurve$. %
Additionally we compute and store the action matrices $M_\sk,M_\cmt'$
associated to $\secretIsogeny,\commitIsogeny$ with respect to
$(P_{E_0},Q_{E_0})$, $(P_{\pk},Q_{\pk})$ and
$(P_\com', Q_\com') = (P_{\commitCurve},Q_{\commitCurve})$. %

\paragraph{The challenge} is expressed as a matrix
$\chMatrix\in\GLtwo/Z$ with respect to the public basis
$(P_\com', Q_\com')$, i.e.
\begin{equation}
  \label{eq:com_basis}
  S_\cmt = Z\cdot (P_\cmt,Q_\cmt) := Z \chMatrix \cdot(P_\com', Q_\com')\ .
\end{equation}
In the non-interactive signature, $\chMatrix$ is computed as
$\mathsf{H}(j(\publicCurve),j(\commitCurve),\msg)$, where
$\mathsf{H}:\F_{p^2}\times\F_{p^2}\times\{0,1\}^*\to\GLtwo/Z$ is a
preimage-resistant hash function.

\paragraph{The response}
is then computed follows. %
We start by calling Algorithm $\ResponseIdeal$ on $\secretIdeal$ and
$\commitIdeal$ to obtain an ideal $I_\rho\subset\OrderEnd(E_\pk)$ such
that $\rho(P_\pk)=P_\cmt$ and $\rho(Q_\pk)=Q_\cmt$. %
Note that $M_\cmt = M_\cmt' \chMatrix^{-1}$ at Line~\ref{alg:ri:m_com}.

Then, writing $\bbE_\pk = (\publicCurve, S_\pk)$ and
$\bbE_\com = (\commitCurve, S_\com)$, we let
$I_{\rho, \curvelevel_\com} = I_\rho \cap \OrderEnd(\bbE_\pk)$. %
Next, we randomize $I_{\rho, \curvelevel_\com}$ by taking a random
equivalent $\OrderEnd(\bbE_\pk)$-ideal of bounded norm. %
For this, we take a uniformly random
$\alpha\in I_{\rho, \curvelevel_\com}$ of odd norm
$\norm(\alpha) < \degreeBound\cdot\norm(I_{\rho, \curvelevel_\com})$. %
An algorithm to sample $\alpha$ is described
in~\cite[Algorithm~3.7]{NISTPQC-ADD-R2:SQIsign24}. %
If no such $\alpha$ is found, we restart the signing procedure by
sampling a new commitment. %
The parameter $\degreeBound$ and the restart probability will be
discussed in \Cref{sec:lattice_shapes}.

Then, we set $I_\alpha := \chi_{I_{\rho}}(\alpha)$ (see \Cref{eq:Xi})
and we compute a response isogeny $\respIsogeny$ such that
$I_{\respIsogeny} = I_\alpha$. %
By \Cref{eq:Chi_change_matrix}, the matrix $M_{\respIsogeny}$ wrt $(P_\pk,Q_\pk)$ and
$(P_\cmt,Q_\cmt)$ is $\frac{1}{\deg \rho}\tomatrix{\alpha}{\level}^{\adjoint} = \gamma\Id$, where
$\gamma=\frac{1}{\norm(I_\rho)} \cdot \frac{\tr(\alpha)}{2}$.
Finally we construct a 2-dimensional isogeny representation of
$\respIsogeny$, like
in~\cite[Section~4.4.3]{NISTPQC-ADD-R2:SQIsign24}. %
For this, we construct an auxiliary isogeny
$\auxIsogeny:\commitCurve\to\auxCurve$ such that
$\deg(\respIsogeny) + \deg(\auxIsogeny) = 2^a < 2^e$ and we construct
the Kani diamond on $\respIsogeny$ and $\auxIsogeny$, as in
\Cref{fig:hd_response_diagram}, defining a $(2^a,2^a)$-isogeny
$\Phi:\publicCurve\times \auxCurve \to \commitCurve \times \auxCurve'$. %

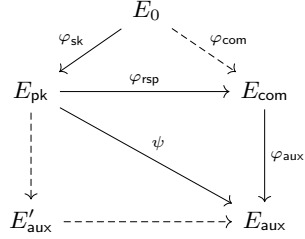
\begin{figure}
    \[
    \begin{tikzcd}
    & E_0 \arrow[ld, "\secretIsogeny"'] \arrow[rd, "\commitIsogeny", dashed] &  \\
    \publicCurve \arrow[rr, "\respIsogeny"] \arrow[dd, dashed] \arrow[rrdd, "\psi"] &                                                                 & \commitCurve \arrow[dd, "\auxIsogeny"] \\
    & & \\
    \auxCurve' \arrow[rr, dashed] & & \auxCurve                             
    \end{tikzcd}
    \] 
    \caption{Protocol diagram for the HD version of the protocol.}
    \label{fig:hd_response_diagram}
\end{figure}

The signature then contains the following data:
\begin{itemize}
\item The degree  $\deg(\respIsogeny)$ of the response,
\item The auxiliary curve $\auxCurve$,
\item The kernel of the $(2^a,2^a)$-isogeny $\Phi$, encoded with
  respect to public bases of $\publicCurve[2^e]$ and $\auxCurve[2^e]$,
\item The challenge matrix $\chMatrix$,
\item Optionally: the scalar $\gamma = \frac{\tr(\alpha)}{2 \, \norm(I_\rho)}$.
\end{itemize}

\paragraph{Verification} recomputes the $(2^a,2^a)$-isogeny $\Phi$
from its kernel and evaluates it to recover $\commitCurve$ and $\respIsogeny$, checks that
$\chMatrix = \mathsf{H}(j(\publicCurve),j(\commitCurve),\msg)$ and
that
\[
  \varphi_\rsp\begin{pmatrix}P_{\publicCurve}\\Q_{\publicCurve}\end{pmatrix} =
  \gamma\chMatrix\begin{pmatrix}P_{\commitCurve}\\Q_{\commitCurve}\end{pmatrix}\ .
\]
If $\gamma$ is not included in the signature, the verifier can
recompute it via a generalized discrete-logarithm.

\subsection{Lattice shapes and response degree} 
\label{sec:lattice_shapes}

The response computation described above may fail and restart the
signature when no $\alpha\in I_{\rho, \curvelevel_\com}$ of odd norm
$\norm(\alpha) < \degreeBound\cdot\norm(I_{\rho, \curvelevel_\com})$ 
is found. %
Though we set this failure probability to be so low as to be
practically irrelevant, it is not practically negligible and it thus
affects the distribution of responses and consequently the
Zero-Knowledge property. %
We note that \SQIsign also has failure cases that force it to restart,
as analyzed in~\cite{NISTPQC-ADD-R2:SQIsign24,qlapoti}, and those
failure cases also affect \SQInstructor, however this failure event is
specific to the design described here, thus we analyze it exclusively.

Because $I_{\rho, \curvelevel_\com}$ is isometric to
$\Hom(\curvelevel_\pk,\curvelevel_\com)$ up to a multiplicative factor $\deg \rho$, the probability of failing to find
$\alpha$ is the same as the probability of failing to find an isogeny
of odd degree $< \degreeBound$ in $\Hom(\curvelevel_\pk,\curvelevel_\cmt)$. %
Let $(\alpha_1,\alpha_2,\alpha_3,\alpha_4)$ be the Minkowski-reduced basis of
$\Hom(\curvelevel_\pk,\curvelevel_\com)$ and  $(\lambda_1,\lambda_2, \lambda_3, 
\lambda_4)$ be the norms of the $\alpha_i$'s, i.e. the squares of the successive minima\footnote{recall that we are using the norm specified 
in \Cref{eq:scalar_product_norm_length}.}
It is not possible that all the minima are even, otherwise all 
the elements of the lattice would be even, contradicting 
\Cref{lem:surjectivity_torsion}.
Thus, the probability of not finding an odd norm vector in the intersection with 
the ball of radius \degreeBound is lower-bounded by $\Pr(\lambda_4 > 
\degreeBound)$.
We estimate this probability via a quantitative analysis of the distribution of 
the minima, for which we provide experiments\footnote{\dataurl}.
We also refer to~\cite[Section~9.3.2]{NISTPQC-ADD-R2:SQIsign24} for a similar 
analysis.
Recall that:
\begin{enumerate}
\item \ComputeSolutions returns a left $\OrderEnd$-ideal $I_\rho$
  where $\OrderEnd =  \OrderEnd(\publicCurve)$. %
\item
  $I_{\rho, \curvelevel_\com} = I_\rho \cap (\Z + \torsion \OrderEnd)$. 
\item $\OrderEnd$ is statistically close to the uniform distribution
  over the set of maximal orders of $B_{p,\infty}$
  (\Cref{tab:deuring}), by the choice of $\deg(\secretIsogeny)$.
\item if $C$ is
  $\covolume(\Hom(\curvelevel_\pk,\curvelevel_\cmt)) =
  [M_2(\Zmod):\spanned{\Gamma}] \cdot \tfrac 14 p$, for scalar level
  structures $C = \tfrac p4 N^3$, then by Minkowski's second theorem we have 
	\\
  $ \frac{1}{55} C^2 \le \frac{2^4}{9\pi^4} C^2 \leq \prod_{i=1}^4 \lambda_i < \frac{2^{10}}{\pi^4} C^2 < 11 C^2$.
\end{enumerate}
In our experiments, we observe an inverse linear correlation between the first 
and the last minima, that means we can approximate 
$\lambda_1 \lambda_4 \approx 
k_1 C$
for a constant $k_1$. Our experiments suggest $k_1 \approx 0.314$.
We thus focus on the cumulative distribution of  $\lambda_1$,
\ie $\Pr(\lambda_1 \leq x)$,
that is the probability that the shortest vector of $\Hom(\curvelevel,\curvelevel')$ has norm at most $x$.

Following the same estimates as in~\cite{NISTPQC-ADD-R2:SQIsign24}, we know that there are 
$\psi(x) = x \prod_{l|x} (1 + \tfrac 1l)$ primitive left $\OrderEnd(\curvelevel)$-ideals 
(i.e. cyclic isogenies mapping curves with level structure)
of norm $x$, that means that the number of primitive ideals of odd norm at most $x$ $\sum_{2r+1 \leq x} \psi(2r+1)$, which asymptotically lies between $ \frac 14 x^2$ and $x^2$.
To count also minimal isogenies $\varphi \colon \bbE_\pk \to \bbE'$  of even degree, (here minimal just means that it is not a multiple of another isogeny $ \bbE_\com \to \bbE'$) we notice that such maps, since they map scalar structures to multiples of scalar structures, at the level of elliptic curves (without level structures) are of the form $\varphi = 2^\delta \varphi'$ with $\delta \le \etorsion$ and $\varphi'$ a cyclic, odd-degree isogeny. Moreover if we choose $\varphi'$, then $E'$ is determined and the level structure is determined up to $2^{3\delta}$ choices (or $\tfrac{3}4 2^{3\delta}$ if $\delta = e_\tor$). Then, fixing $\deg \varphi = 2^{2\delta} \deg \varphi' = 2^{3\delta} (2r+1)$, then the number of such $\phi$ is (up to a small constant in the case $\delta = e_\tor$) $2^{3\delta} \psi(2r+1)$. 
Summing over possible $\delta$ and $2r+1$ such that $4^\delta (2r+1) \le x$, we see that the possible minimal isogenies $\bbE_\com \to \bbE'$ of degree at most $x$ is equal to $\alpha x^2$ for some $\alpha$ bounded from below by $\tfrac 14$ and also bounded from above by some constant (in particular $\alpha \le 2$ if $x/N^2$ is large enough).

The total number of ideal classes, or equivalently the number of possible choices for $\bbE' = \bbE_\com$
up to isomorphism, is approximately $\tfrac{p}{12} N^3 = \tfrac C3$ for 
scalar level structures.
Thus, if all ideal classes are equally likely and it never happens that two 
ideals of norm at most $x$ are in the same class, for $x\gg N^2$ we have
\begin{equation}
    \label{eq:lambda1_approx}
    \Pr \left( \lambda_1 \leq x \right) = k_2 \left(\frac{\sqrt C}{x}\right)^2 %
     \ , \text{ for } k_2  \in \left[ \frac 32, 6\right] \ .
\end{equation}
Thus, combining the last estimate with our correlation hypothesis, we have that
\begin{equation}
    \label{eq:lambda4_approx}
    \Pr\left(\lambda_4 > x\right) = \Pr\left(\lambda_1 \leq \frac{k_1 
    C}{x}\right)
    \approx k_3 \left(\frac{\sqrt C}{x}\right)^2 \text{ for } k_3 = k_1^2 k_2
    \ .
\end{equation}
Our experiments corroborate this estimate, for $k_3 \approx 0.553794$. Thus, we 
can conclude that via fixing $\degreeBound = \kappa \sqrt C = \kappa \sqrt p 
\torsion^{3/2}$ we can bound the restart probability by $\tfrac{k_3}{\kappa^2}$. 
Thus, we have three possible strategies. The aggressive one is to fix $\kappa$ 
to a small $\bigO(1)$ constant, having a restart probability of $\bigO(1)$, 
getting a more efficient signature. The conservative one is to fix $\kappa = 
\sqrt[4] p$, getting a negligible restart probability in $\secpar$.
The intermediate one is to aim to a restart probability of $2^{-64}$, in line with~\cite{NISTPQC-ADD-R2:SQIsign24} and
NIST's recommendations that the adversary has access to no more 
than $2^{64}$ signing queries~\cite{NIST-call}, thus we fix $\kappa = 2^{32}$, 
implying
    \begin{equation}
        \label{eq:degree_bound_intermediate}
        \degreeBound = 2^{32} \sqrt p \torsion^{3/2} \ ,
    \end{equation}
that implies, recalling the definition of \torsion, that $\log(\degreeBound) = 
32 + \tfrac 12 (\log(p) +  \secpar)$.

\subsection{\EUFCMA Security}
\label{sec:hd_zk}

We are finally ready to discuss the \EUFCMA security of our \SQInstructor signature.
From now we denote by $\Prot_\SQItorsion$ the identification protocol from 
\Cref{sec:sqisign_protocol_new}.
As noted repeatedly, 
see~\cite{EC:DLRW24,AC:BDDLMP24,NISTPQC-ADD-R2:SQIsign24,PKC:BBCCIL25}, we cannot sample 
efficient representations of  large degree isogenies without relying on the 
knowledge of the endomorphism ring of the starting curve.
This means that the straightforward simulator for the \SQIsign protocol 
described above cannot be instantiated in polynomial time.

To overcome this problem in~\cite{EC:DLRW24}, and following works on \SQIsign 
variants~\cite{AC:BDDLMP24,AC:DupFou24}, the authors introduced the notion of 
\emph{isogeny oracles}: oracles that on input a supersingular elliptic curve and 
some parameters, return an efficient representation of a large degree isogeny 
from the curve following a given distribution.
\begin{remark}
    In a follow-up the authors of~\cite{C:ABDPW25} provide a different 
    security proof for the \SQIsign signature scheme, avoiding the use of
    isogeny oracles, but instead introduce the
    notion of \emph{hint-assisted} zero knowledge, where the simulator is given 
    hints sampled from an isogeny distribution. A similar proof can be adapted to
    our protocol, but is very technical and would take all our available space.
    We focus here on the more intuitive, yet rigorous proof based on oracles,
    and give the hint-assisted proof in \Cref{sec:proofs_sqisign_protocol_new} for completeness.
\end{remark}

As for~\cite{AC:BDDLMP24}, we consider two variants of the oracle:
\begin{definition}
    \label{def:uto}
    A $\Gamma$-\emph{uniform target oracle} (\UTO) 
    for $\Gamma \leq \GLtwo$ is an oracle taking 
    as input a supersingular elliptic curve with $\Gamma$-level structure $\bbE$ defined 
    over $\Fpsq$ and an integer $\degreeBound$
    and outputs a random isogeny $\phi : \bbE \to \bbE'$
    (in efficient representation) such that:
    \begin{enumerate}
        \item $\bbE'$ is distributed sampled from the uniform distribution on 
            the set of supersingular elliptic curves over \Fpsq with level 
            structure so that there exists an odd-norm isogeny in 
            $\Hom(\bbE,\bbE')$ of degree $\leq \degreeBound$;
        \item The conditional distribution of $\varphi$ given $\bbE'$ is uniform 
            over isogenies $\varphi : \bbE \to \bbE'$ of odd degree $\leq\degreeBound$
    \end{enumerate}
\end{definition}

\begin{definition}[Definition 23~\cite{AC:BDDLMP24}]
    \label{def:fidio}
    A \emph{fixed degree isogeny oracle} (\FIDIO) is an oracle taking as input a 
    supersingular elliptic curve $E$ defined over $\Fpsq $ and an integer $N$, 
    and outputs a uniformly random isogeny $\phi : E \to E'$ (in efficient 
    representation) with domain $E$ and degree $N$.
\end{definition}

Using these oracles, we can enunciate and prove the following theorem:
\begin{theorem}
    \label{thm:sqisign_hd_sec_oracles}
    The $\Prot_\SQItorsion$ identification protocol is statistically 
    \emph{non-aborting} honest-verifier zero-knowledge in the \UTO and \FIDIO model.
    In other words, there exists a polynomial time simulator $\Sim$ with access 
    to a \UTO and a 
    \FIDIO that produces random transcripts which are statistically indistinguishable 
    from honest non-aborting transcripts.
\end{theorem}

\begin{proof}
    Let \publicCurve be a random supersingular elliptic curve over $\Fpsq$ and 
    ${\Ppk},{\Qpk}$ be a random basis of the $\torsion$-torsion.
    We consider the following simulator \Sim that using the \UTO and \FIDIO 
    returns a transcript $(\commit,\ch,\resp)$:
    \begin{enumerate}
        \item \Sim sets  $ \chLevelStructure \gets Z \cdot 
            \vv{\Ppk}{\Qpk}$, then calls the \UTO with input 
            $(\publicCurve, \chLevelStructure) $ and $\degreeBound$ to get the 
            isogeny
            $\respIsogeny :  (\publicCurve, \chLevelStructure) \to 
            (\commitCurve, \commitLevelStructure)$;
            \label{item:sim_commitment_hd}
          \item \Sim computes the matrix $\chMatrix \gets \GLtwo/Z$ such that
            $S_\cmt = ZM \cdot (P_{\commitCurve}, Q_{\commitCurve})$
            and sets $\ch := \chMatrix$;
            \label{item:sim_challenge_hd}
          \item \Sim sets $\commit := \left(j(\commitCurve)\right)$;
            \label{item:sim_commit_hd}
        \item  Let $d = \deg(\respIsogeny) $, $a = \lceil \log ( d ) \rceil$, \Sim 
            calls the \FIDIO with input $\commitCurve, 2^a - d$ to get 
            $\auxIsogeny : \commitCurve \to \auxCurve$;
            \label{item:sim_aux_hd}
        \item Let $\psi = \auxIsogeny\circ \respIsogeny$, \Sim sets 
            $(\Paux,\Qaux) = ([d^{-1}]\psi(\Ppk), [d^{-1}]\psi(\Qpk))$; \label{item:sim_aux_points_hd}
        \item \Sim sets $\resp := (a,\auxCurve, \Paux, \Qaux)$;
            \label{item:sim_resp_hd}
    \end{enumerate}
    It is clear from the definition of the oracles and the steps that the 
    response consists of a valid isogeny representation distributed as in the 
    honest transcript that can be used to re-compute $\respIsogeny : 
    \publicCurve \to \commitCurve$.
    Finally, note that the \UTO only returns isogenies for curves with 
    level structure that admit odd degree isogenies from the input curve with 
    level structure, that is the same condition required for the honest
    transcript to not restart.
    \qed
\end{proof}

Since the commitment set, \ie the set of supersingular elliptic curves with 
torsion points, is a superset of \SQIsign commitment set -- just supersingular 
elliptic curves -- the min-entropy of the commitment in \SQItorsion is also 
negligible in the security parameter by~\cite[Lemma~4.1]{C:ABDPW25}.
This, in combination with from \Cref{thm:sqisign_hd_sec_oracles}, 
\Cref{thm:soundness}, immediately implies the \EUFCMA security of \SQInstructor 
in the \UTO, \FIDIO and random oracle models.
More precisely: if the signature restarting probability is negligible in the 
security parameter this is an immediate implication of~\cite[Lemma~3.5]{EC:AABN02}.
If not, assuming that the signing procedure is repeated enough times to ensure 
negligible failure probability, the \EUFCMA security follows 
from~\cite[Theorem~1]{IWSEC:BerRou18}.

%

%
%
%
%
%

%
%

%

%
%
%
%

\subsection{Implementation and benchmarks}
\label{sec:hd_parameters_and_benchmarks}

We implemented the protocol described in this section by forking
\SQIsign's reference implementation. %
The bound $\degreeBound$ was set according to the ``intermediate''
choice of \Cref{eq:degree_bound_intermediate}. %
We compare the performances of \SQInstructor and \SQIsign in
\Cref{tab:benchmarks_hd}. %
We observe very similar performances across the board, with a slight
increase in signature size and verification time, both attributable to
the longer $(2,2)$-isogeny chain in \SQInstructor. %
We remark that both \SQInstructor and \SQIsign's signatures could be
slightly compressed, however we deem this optimization minor.

\begin{table}
    \centering
    \begin{tabular}{l@{~~~}l >{\raggedleft}p{15mm} >{\raggedleft}p{15mm} >{\raggedleft}p{15mm} >{\raggedleft}p{13mm} >{\raggedleft\arraybackslash}p{13mm}}
      \toprule
             &      & \multicolumn{3}{c}{(Mcycles)}& \multicolumn{2}{c}{(bytes)} \\
      Scheme & Security & KeyGen & Sign & Verify & pk & sig      \\
      \toprule
                  & NIST I   &  33.5   &  78.9  &  6.0 &  65 & 174    \\
      \SQItorsion & NIST III & 102.9   & 240.1  & 20.8 &  97 & 272    \\
                  & NIST V   & 166.2   & 420.2  & 39.0 & 129 & 355    \\
      \midrule
               & NIST I   &  32.1  &  76.8  &  4.3 &  65 & 148        \\
      \SQIsign & NIST III & 101.3  & 237.7  & 14.4 &  97 & 224        \\
               & NIST V   & 166.9  & 398.6  & 27.7 & 129 & 292        \\
         \bottomrule
    \end{tabular}
    \smallskip
    \caption{Comparison between \SQItorsion and \SQIsign. Timings are
      the median of 100 runs on a laptop CPU AMD Ryzen 7000.}
    \label{tab:benchmarks_hd}
\end{table}

\bibliographystyle{splncs04}
\bibliography{additional,cryptobib/abbrev3,cryptobib/crypto}

@electronic{Leroux,
  author = "Antonin Leroux",
  title = "Quaternion algebras and isogeny-based cryptography",
  howpublished = {\url{https://www.lix.polytechnique.fr/Labo/Antonin.LEROUX/manuscrit_these.pdf}},
  year = "2022",
  note = "Last accessed March 23, 2023",
}

@electronic{LearningToSQI,
  author = {Maria {Corte-Real Santos} and Giacomo Pope},
  title = {Learning to SQI},
  subtitle = {Implementing SQISign in SageMath},
  howpublished = {\url{https://learningtosqi.github.io/}},
  year = {2022},
  note = {Last accessed June 11, 2023},
}

@book{quaternion_book,
  title={Quaternion algebras},
  author={Voight, John},
  year={2021},
  publisher={Springer Nature}
}

@inproceedings{arpin2024generalized,
  title={Generalized class group actions on oriented elliptic curves with level structure},
  author={Arpin, Sarah and Castryck, Wouter and Eriksen, Jonathan Komada and Lorenzon, Gioella and Vercauteren, Frederik},
  booktitle={International Workshop on the Arithmetic of Finite Fields},
  pages={171--190},
  year={2024},
  organization={Springer}
}

@inproceedings{qlapoti,
  author       = {Giacomo Borin and
                  Maria {Corte{-}Real Santos} and
                  {Jonathan Komada} Eriksen and
                  Riccardo Invernizzi and
                  Marzio Mula and
                  Sina Schaeffler and
                  Frederik Vercauteren},
  editor       = {Goichiro Hanaoka and
                  Bo{-}Yin Yang},
  title        = {Qlapoti: Simple and Efficient Translation of Quaternion Ideals to
                  Isogenies},
  booktitle    = {Advances in Cryptology - {ASIACRYPT} 2025 - 31st International Conference
                  on the Theory and Application of Cryptology and Information Security,
                  Melbourne, VIC, Australia, December 8-12, 2025, Proceedings, Part
                  {IV}},
  series       = {Lecture Notes in Computer Science},
  volume       = {16248},
  pages        = {174--205},
  publisher    = {Springer},
  year         = {2025},
  url          = {https://doi.org/10.1007/978-981-95-5113-2\_6},
  doi          = {10.1007/978-981-95-5113-2\_6}
}

@book{silverman2009arithmetic,
  title={The arithmetic of elliptic curves},
  author={Silverman, Joseph H},
  volume={106},
  year={2009},
  publisher={Springer}
}

@article{de2017mathematics,
  title={Mathematics of isogeny based cryptography},
  author={De Feo, Luca},
  journal={arXiv preprint arXiv:1711.04062},
  year={2017}
}

@manual{sage,
    label        = {Sag95},
    author       = {{The Sage Developers}},
    title        = {{S}age{M}ath, the {S}age {M}athematics {S}oftware {S}ystem},
    url          = {https://www.sagemath.org},
    version      = {10.5},
    year         = {2025},
    note         = {DOI 10.5281/zenodo.6259615},
}

@article{boneh2020graduate,
  title={A graduate course in applied cryptography},
  author={Boneh, Dan and Shoup, Victor},
  journal={Draft 0.5},
  pages={14},
  year={2020}
}

@misc{codogni2025spectraltheoryisogenygraphs,
      title={Spectral Theory of Isogeny Graphs}, 
      author={Giulio Codogni and Guido Lido},
      year={2025},
      eprint={2308.13913},
      archivePrefix={arXiv},
      primaryClass={math.NT},
      url={https://arxiv.org/abs/2308.13913}, 
}

@misc{forensic,
      title={Forensic categories: a framework for {SQI}sign-like primitives}, 
      author={Andrea Basso and Luca {De Feo} and Sikhar Patranabis and Ilinca Radulescu and Benjamin Wesolowski},
      year={2026},
      howpublished={Private communication} 
}

@misc{NIST-call,
      author = {{National Institute of Standards and Technology (NIST)}},
      title = {{Call for Additional Digital Signature Schemes for
the Post-Quantum Cryptography Standardization Process}},
      year = {2022},
      url = {https://csrc.nist.gov/csrc/media/Projects/pqc-dig-sig/documents/call-for-proposals-dig-sig-sept-2022.pdf}
}

@misc{sigma_protocols,
        author = {Damgard},
        title = {On $\Sigma$-protocols},
        howpublished = {\url{http://www.cs.au.dk/%7eivan/Sigma.pdf}},
        year = {2010},
        accessed = {2026-02-03},
}

@Article{klpt,
  author    = {Kohel, David and Lauter, Kristin and Petit, Christophe and Tignol, Jean-Pierre},
  title     = {On the quaternion-isogeny path problem},
  journal   = {LMS Journal of Computation and Mathematics},
  year      = {2014},
  volume    = {17},
  number    = {A},
  pages     = {418--432},
  publisher = {Cambridge Univ Press},
}

@article{arpin-sarah-borel-lvl-structures,
  title={Adding Level Structure to Supersingular Elliptic Curve Isogeny Graphs},
  author={Sarah Arpin},
  journal={Journal de th{\'e}orie des nombres de Bordeaux},
  year={2022},
  url={https://api.semanticscholar.org/CorpusID:247292410}
}

@string{springer =              "Springer"}

\appendix

\section{Pseudocode}\label{sec:pseudocode}

\subsection{Algorithms from other authors}
To make our paper self contained, 
in this section we recall some known algorithms from the \SQIsign literature, and 
more specifically from Leroux thesis~\cite{Leroux}.
In our algorithms from \Cref{sec:klpt} we use the following:

\begin{itemize}
    \item \EquivalentPrimeIdeal, that on input a fractional ideal $I$ returns the 
        smallest left equivalent ideal $I'$ of prime norm and $\alpha \in I$ such 
        that $I' = \chi_I(\alpha)$.
        We refer to~\cite[Algorithm 6]{Leroux} for details; 
        \label{alg:EquivalentPrimeIdeal}
    \item \FullRepresentInteger, that on input an integer $M$ returns a 
        quaternion in $\O_0$
        of norm $M$ in $\B$ (if it exists). We refer to~\cite[Algorithm 
        4]{Leroux} for details; \label{alg:FullRepresentInteger}
    \item $\FullStrongApproximation_N$ (\cite[Algorithm 5]{Leroux}), that on input an integer $M$ and a 
        pair $C,D \in \Z$ returns $\mu \in \O_0$ such that 
        $n(\mu) = N$ and $2\mu = \lambda (jC + kD) + M \mu_1$ for some $\mu_1 
        \in \O_0$ and $\lambda \in \Z$.
        \label{alg:FullStrongApproximation}
    \item \IdealSuborderEichlerNorm~\cite[Algorithm 
        14]{Leroux}, that we recall in \Cref{alg:IdealSuborderEichlerNorm}, 
        with notations slightly adapted to the use it has in~\Cref{sec:squares};
    \item \IdealEichlerNorm, that we recall in \Cref{alg:IdealEichlerNorm};
    \item \GenericKLPT~\cite[Algorithm 33]{Leroux}, which we include for comparison to our variants, with the original notations. 
    \item \SigningKLPT~\cite[Algorithm 11]{Leroux} which we include for comparison to our variants, with the original notations. 
    It relies on $\mathsf{RandomEquivalentEichlerIdeal}$ for rerandomization, which is explained (without pseudocode) in \Cref{sec:rerandomization_details}.
\end{itemize}

\begin{algorithm}
    \caption{$\IdealSuborderEichlerNorm_{\ell^e}$} \label{alg:IdealSuborderEichlerNorm}
    \begin{algorithmic}[1]
        \Require An integer $N$, two left $\O_0$-ideals $I,J$ of inert prime 
        norm $N_I,N_J$, with $\gcd(N, N_I,N_J) = 1$,
        \Ensure $\beta \in \Z + DI \cap J$ of norm $N_J \ell^e$. 
        \State {Select a random class $(C_2 : D_2) \in \mathbb{P}^1(\Z/D\Z)$. 
        \label{line:output_ideal_suborder_eichler_norm}}
        \State Compute $\gamma = \FullStrongApproximation_{N_J \ell^{e_1}}(N, C_2, D_2)$. If 
        the computation fails go back to Step 
        \ref{line:output_ideal_suborder_eichler_norm}.
        \State Compute $(C_0 : D_0) = \mathsf{EichlerModConstraint}(I, \gamma, 1)$.
        \State Compute $(C_3 : D_3) = \mathsf{IdealModConstraint}(J, \gamma)$.
        \State Sample a random $D'_2$ in $\Z/N\Z$, compute $C'_2 = -D'_2 C_2 
        (D_2)^{-1} \mod N$.
        \State Compute $C_1 = \CRT_{N_I,N,N_J}(C_0, C'_2, C_3)$, $D_1 =
        \CRT_{N_I,N,N_J}(D_0, D'_2, D_3)$.
        \State Compute $\mu = \FullStrongApproximation_{N_J \ell^{e-e_1} / n(\gamma)}(N_I 
        N N_J, C_1, D_1)$. If it fails, go back to step
        \ref{line:output_ideal_suborder_eichler_norm}.
        \State return $\beta = \gamma \mu$.
    \end{algorithmic}
\end{algorithm}

\begin{algorithm}
    \caption{$\IdealEichlerNorm_\mathcal{N} (I,J)$} \label{alg:IdealEichlerNorm}
    \begin{algorithmic}[1]
        \Require Two left $\O_0$-ideals $I,J$ of inert prime norm $N_I,N_J$, with 
        $\gcd(N_I,N_J) = 1$,
        \Ensure $\beta \in (\Z + I) \cap J$ of norm in $\normklpt{J} \mathcal{N}$.
        \State Compute $L = \EquivalentPrimeIdeal(J)$, $L = \chi_J(\delta)$ for 
        $\delta \in J$. Set $N = n(L)$.
        \State Compute $\gamma = \FullRepresentInteger_{ N_J \mathcal{N} }$.
        \State Compute $(C_0 : D_0) = \mathsf{IdealModConstraint}(L, \gamma)$.
        \State Compute $(C_1 : D_1) = \mathsf{EichlerModConstraint}(I, \gamma, \delta)$.
        \State Compute $C = \CRT_{N_J,N_I}(C_0, C_1)$ and $D = \CRT_{N_J,N_I}(D_0, D_1)$.
        \State Compute $\mu = \StrongApproximation_{N_J \mathcal{N} / \normklpt{ \gamma }}(N_I 
        N_J, C, D)$.
        \State return $\beta = \gamma \mu \delta$.
    \end{algorithmic}
\end{algorithm}

\subsection{Pseudocode for \BorelSigningKLPT}

in \Cref{alg:borel_klpt}, we give full pseudocode for the more efficient KLPT variant for Borel level structure mentioned  in \Cref{sec:klpt_others}.

\begin{algorithm}
    \caption{$\BorelSigningKLPT_{\ell^e}(I,I_{\torsion})$}
    \label{alg:borel_klpt}
    \begin{algorithmic}[1]
        \Require  $I$ and $I_{\torsion}$ left ideals of the same maximal order $\torsion$ in $\Bpinf$, with $\torsion\gets\normklpt{I_{\torsion}}$ prime or power of prime of known factorization and $e$ such that $\ell^e$ large enough
        (See \Cref{prop:borelKLPT}).
        \Ensure $J\sim I$ such that $J=\chi_I(\beta)$ with $\beta\in(\Z+I_{\torsion})\cap I$ and $\normklpt{J}\mid 2^e$
        \State $\alpha,I_c' \gets \EquivalentPrimeIdeal(I_c)$ 
        \mycomment{$I_c' = \chi_{I_c}(\alpha)$}
        \State Let $I_1 \gets \frac{1}{\alpha} I\alpha$
        \State Let $I_{\torsion}' \gets \frac{1}{\alpha} I_{\torsion}\alpha$
        \State 
        $\alpha_I,I_0\gets$\InitialRerandomization$(I_1,\Z+I_{\torsion}',I_c')$ \mycomment{See \Cref{sec:klpt_rerandomization}}
        \State $\beta\gets$\IdealEichlerNorm$_{\ell^e}(I_c'I_{\torsion}',I_0)$
        \State $J_0 \gets \chi_{I_0}(\beta)$ and $J_a\gets[I_c']^*J_0$ and 
        $J\gets\alpha J_a \frac{1}{\alpha}$
    \State\Return $J$ 
    \end{algorithmic}
\end{algorithm}

\begin{algorithm}
    \caption{$\GenericKLPT_{\mathcal{N}}(\O_1,I)$}
    \label{alg:generic_klpt}
    \begin{algorithmic}[1]
    \Require $\O_1$ a maximal order and $I$ left $\O_1$-ideal
    \Ensure $J\sim I$ of norm in $\mathcal{N}$
    \State Compute $I_0=I(\O_0,\O_1)$ \Comment{$I_0$ primitive connecting ideal of $\O_0$ and $\O_1$}
    \State Compute $K=\textsf{Random}\EquivalentPrimeIdeal(I_0)$ and $\alpha$ such that $K=I_0\alpha$
    \State Compute $I'=\alpha^{-1}I\alpha$ and $L=[K]^*I'$
    \State Compute $\mu=\IdealEichlerNorm_{\mathcal{N}}(K,L)$
    \State Set $\beta=\alpha\mu\alpha^{-1}$
    \State\Return $\chi_I(\beta)$
    \end{algorithmic}
\end{algorithm}

\begin{algorithm}
    \caption{$\SigningKLPT_{2^e}(I_{\tau},I)$}
    \label{alg:signing_klpt}
    \begin{algorithmic}[1]
    \Require $I_\tau$ a primitive left $\O_0$ and right $\O_1$-ideal of prime norm $N_\tau$
    \Require $I$ a left $O_1$-ideal
    \Ensure $J\sim I$ of norm $2^e$
    \State Compute $K=\mathsf{RandomEquivalentEichlerIdeal}(I,N_\tau)$
    \State Compute $K'=[I_\tau]^*K$ and set $L=\EquivalentPrimeIdeal(K')$, $L=\chi_{K'}(\delta)$ for $\delta\in K'$ with $N=n(L)$. Sample $e_0\in e_0(N)$.
    \State Compute $\gamma=\FullRepresentInteger_{ND2^{e_0}}$
    \State Compute $(C_0:D_0)=\textsf{IdealModConstraint}(L,\gamma)$
    \State Compute $(C_1:D_1)=\textsf{EichlerModConstraint}(I_\tau,\gamma,\delta)$
    \State Compute $C=\textsf{CRT}_{N_\tau,N}(C_0,C_1)$ and $D=\textsf{CRT}_{N_\tau,N}(D_0,D_1)$
    \State $e_0 \gets v_2(n(\gamma))$ and $e_1\gets e-e_0$
    \State Compute $\mu=\FullStrongApproximation_{2^{e_1}}(NN_\tau,C,D)$. If it fails, go back to step 3.
    \State Set $\beta=\gamma\mu$
    \State\Return $J=[I_\tau]_*\chi_L(\beta)$
    \end{algorithmic}
\end{algorithm}

\section{Generalization of \ScalarSigningKLPT to inputs of non-square norm modulo $\torsion$} \label{sec:squares}

In \Cref{sec:klpt} we assumed the input to \ScalarSigningKLPT to be an ideal $I$ of norm a square modulo $\torsion$. 
In this section, we first explain that equivalences in $\Z+N\O$ preserve the square or non-square character of ideal norms modulo $\torsion$ (or to say it short, the squarity of ideals). 
It is therefore clear that the input to \ScalarSigningKLPT  has the same squarity than the input to \IdealSuborderEichlerNorm  in this algorithm.
Next, we analyze why these algorithms only work for ideals with square norm modulo $\torsion$ (which we'll call square ideals for short).
Finally, we explain in which cases and how it is possible to adjust inputs to \ScalarSigningKLPT so that it can also be used for non-square input ideals.

For conciseness, we use the word "squarity" for integers to designate the character of being or not being a square modulo $\torsion$.
For quaternion ideals, their squarity means the squarity of their norms.
Similarly, if an ideal is said square, we mean that its norm is a square modulo $\torsion$.

\subsection{Equivalence-taking in $\Z+\torsion \O$ preserves squares}

For $\O$ a maximal ideal in $\Bpinf$, taking equivalences in $\Z+\torsion \O$ cannot alter the squarity of an $\O$-ideal, so that the classes in Cl($\Z+\torsion \O$) have either only or no element whose norm is a square modulo $\torsion$.
This fact can be understood in two ways, the elementary and the matrix way.

\begin{lemma}\label{lem:square_preserving}
    Let $\O$ a maximal order in $\Bpinf$ and $I$ a left $\O$-ideal. Let $N$ be a positive integer coprime to $\normklpt{I}$.
    Let $\alpha\in I\cap(\Z+N \O)$ of norm coprime to $N$ and $J=I\frac{\bar\alpha}{\normklpt{I}}$.
    Then either both $\normklpt{I}$ and $\normklpt{J}$ are squares modulo $N$, or none of them is a square modulo $N$.
\end{lemma}
\begin{proof}
    Since $\alpha\in I\cap(\Z+N \O)$, $\alpha=\norm{I}n+N\beta$ for some $n\in \Z$ and $\beta\in \O$ with $\tr(\beta)=0$.
    Therefore $\normklpt{\alpha}=\normklpt{I}^2n^2+N^2\normklpt{\beta}$. 
    Since $\alpha\in I$, $\normklpt{I}\mid \normklpt{\alpha}$ so $\normklpt{I}$ also divides $\normklpt{\beta}$.
    As $N$ and $\normklpt{I}$ are coprime,  there exists $m\in Z$ such that $\normklpt{\alpha} =\normklpt{I}^2n^2+N^2m\normklpt{I}$.
    Therefore we obtain $\normklpt{J}=\normklpt{\alpha}\normklpt{I}= \normklpt{I}n^2+N^2m$, which gives $\normklpt{J}=\normklpt{I}n^2\mod N$.
    $n^2$ is an integer square and therefore a square modulo $N$.
    As a consequence, $\normklpt{J}$ and $\normklpt{I}$ are either both squares or both non-squares modulo $N$.
    \qed
\end{proof}
\begin{proof}
    This can also be seen by the matrices from \Cref{lem:surjectivity_torsion}. 
    Compute $M_I$ and $M_J$ the matrices associated to (the isogenies associated to) the ideal $I$ and $J$ from~\Cref{lem:square_preserving}.
    We know that these matrices are the same up to multiplication by a scalar matrix, since the ideals are equivalent in $(\Z+N \O)$.
    Scalar matrices have square degree, so the determinants of $M_I$ and $M_J$ are either both squares or both non-squares modulo N.
    These determinants are, modulo N, equal to the norms of $I$ and $J$.
    \qed
\end{proof}

\Cref{lem:square_preserving} gives a first idea of why we stated \ScalarSigningKLPT for square input ideals only.
It makes clear that any algorithm taking equivalences in $\Z+\torsion \O$ like \ScalarSigningKLPT has the inherent limitation that it cannot create an output ideal whose norm has a different squarity than the input. 
Given that our \ScalarSigningKLPT  has a square target norm (the exponent $e$ is even), it can obviously only treat square inputs. 

\subsection{Details of \IdealSuborderEichlerNorm applying only to squares}

Knowing that equivalences respecting $\Z+\torsion \O$ preserve squarity, it seems natural to use the exact same algorithm \ScalarSigningKLPT, but with a non-square output norm, and a non-square input ideal.
However, this approach fails due to inherent limitations of its \IdealSuborderEichlerNorm subroutine, whose pseudocode can be found in \Cref{alg:IdealSuborderEichlerNorm} or in Leroux' thesis~\cite[Algorithm 14]{Leroux}.
More precisely, with the notations from the pseudocode of~\Cref{alg:IdealSuborderEichlerNorm} and the notation $\mu_0=j(C_1i+D_1)$, one notices the following points which show that the output norm $\ell^e$ must be a square modulo $\torsion$:

\begin{itemize}
\item For $\lambda$ to exist, we need $\ell^e$ and $\normklpt{\mu_0}$ of same squarity
\item $\ell^e$ has the same squarity as the product $\normklpt{\gamma}\normklpt{\mu_0}$
\item $\normklpt{\gamma}$ and $\normklpt{\mu_0}$ have the same squarity, so their product is a square modulo $\torsion$
\end{itemize}

In consequence, the output of \IdealSuborderEichlerNorm has square norm modulo \torsion.
Since \ScalarSigningKLPT takes all equivalences to the input in $\Z+\torsion \O$, \Cref{lem:square_preserving} means that the input ideal to  \ScalarSigningKLPT  must have square norm modulo $\torsion$.

\subsection{Workaround for non-squares when a solution exists}

Despite this obstacle, suitable input transformations allow to use \ScalarSigningKLPT for non-square input ideals, as long as the output is expected to be a power of a non-square.
Denote $\ell^e$ the required output norm with $\ell$ not a square modulo $\torsion$ and $e$ odd, and $I$ an input ideal of norm not a square modulo $\torsion$. 
Denote $\O$ the left order of $I$.
Then we can use a random right $\O$-ideal $I_r$ of non-square norm $\ell$ to 
create an alternative input $I_rI$ which has square norm modulo $\torsion$.
We can then call \ScalarSigningKLPT on $I_rI$, with output norm $\ell^{e-1}$ which is square modulo $\torsion$ as required.
Finally the obtained output $J$ of \ScalarSigningKLPT is multiplied by $\bar I_r$ to get the ideal $\bar I_rJ$ of norm $\ell^e$ which is equivalent to $I$ in $\Z+\torsion \O$.
\Cref{alg:gen_scalar_klpt} is  obtained by applying these step to all non-square input ideals (and output norms).

\begin{algorithm}
    \caption{$\text{\GeneralizedScalarSigningKLPT}_{\ell^e}(I,\torsion)$}
    \label{alg:gen_scalar_klpt}
    \begin{algorithmic}[1]
        \Require $\torsion$ prime or power of prime of known factorization, $I$ 
        left ideal of a maximal order $\O$ in $B_{p,\infty}$ such that $
        \ell^e\normklpt{I}\mod\torsion$ is a square, $e\in\Z$ even such that $\ell^e$ 
        large enough while $\ell$ is a small prime (See \Cref{prop:scalarKLPT}).
        \Ensure $J\sim I$ such that $J=\chi_I(\beta)$ with $\beta\in(\Z+\torsion 
        \O)\cap I$ and $\normklpt{J}\mid 2^e$
        \If{$\normklpt{I}$ is a square modulo $\torsion$}
            \State $J\gets\ScalarSigningKLPT_{\ell^e}(I,\torsion)$
        \Else
            \State Let $I_r$ a uniformly random right $\O$-ideal of norm $\ell$ non such that $I_rI$ is primitive
            \State $J_r\gets\ScalarSigningKLPT_{\ell^{e-1}}(I_rI,\torsion)$
            \State $J\gets\bar I_rJ_r$
            \If{$J$ not primitive} \State Restart\EndIf
        \EndIf
    \State\Return $J$ \end{algorithmic}
\end{algorithm}

Using this additional input treatment, \GeneralizedScalarSigningKLPT can solve all solvable squarity combinations for smooth norms.
The cases where the in- and output norms have different squarity don't have a solution, since ideals of such norms cannot be equivalent in $\Z+\torsion \O$ as seen in~\Cref{lem:square_preserving}.

\section{Proofs} \label{sec:proofs}
In this appendix we collect the detailed proofs of the statements that we
omitted in the main body of the paper.
For the sake of readability we restate the statements before their proofs.

\subsection{Proofs from \Cref{sec:level_all}}
\label{sec:proofs_level_structures}
\label{sec:proof_deuring_with_level_structure}

%
%
%
%
%
%
%
%
%
%
%
%
%
%
%
%
%
%
%
%
%
%
%
%
%

%
%
%
%
%
%
%
%
%
%
%
%
%
%
%
%
%
%
%
%
%
%
%
%
%
%
%
%
%
%
%
%
%
%
%
%
%
%
%
%
%
%
%
%
%
%
%
%
%
%
%
%
%
%
%
%
%
%
%
%
%
%
%
%
%
%
%
%
%
%
%
%
%
%
%
%
%
%
%
%
%
%
%
%
%
%
%
%
%
%

\DeuringWithLevelStructure*

\begin{proof}
For the first part we start by proving that if $I$ is an ideal 
equivalent to a certain $I_{\rho, \bbE'}$, then $I = I_{\rho', \bbE'}$ for some 
$\rho' \in \Hom(\bbE, \bbE')$. By hypothesis we know that $I = I_{\rho, \bbE'} 
\cdot \tfrac{\alpha}{Q}$ for some $\alpha\in \OrderEnd(\bbE)$ and some positive 
integer $Q$, hence by \eqref{eq:equivalent_ideals} it is enough finding a $\rho' 
$  in $\OrderEnd(\bbE)$ such that $\tfrac{\alpha}{Q} = \tfrac{\hat\rho \circ 
\rho'}{\deg(\rho)}$ or equivalently $\tfrac{\rho \circ \alpha}{Q} = \rho'$. In 
other words this is equivalent to proving that $\rho \circ \alpha$ lies in the 
sublattice $Q \cdot \Hom(\bbE, \bbE')$ of $\Hom(\bbE, \bbE')$. To prove this 
pick an element $\phi \in \Hom(\bbE, \bbE')$ of degree coprime to $Q$. Then by 
definition $\hat \phi \rho$ is an element of $I_{\rho, \bbE'}$ hence 
$\tfrac{1}{Q}\hat \phi \rho \alpha$ is an element of $I \subset \OrderEnd(\bbE)$ 
(we are supposing $I$ is a proper ideal, not only a fractional ideal) hence we 
can find an element $\beta \in \OrderEnd(\bbE)$ such that
\[
    \deg(\phi) \cdot \rho\alpha = Q \phi \cdot  \tfrac{1}{Q}\hat \phi \rho 
    \alpha = Q\phi \beta \in Q \cdot \Hom(\bbE, \bbE')
\]
and since multiplication by $\deg \phi$ is invertible in $\Hom(\bbE, \bbE')/Q 
\Hom(\bbE, \bbE')$ we deduce that also $\rho\alpha$ lies in $Q \Hom(\bbE, 
\bbE')$.

Now let $I$ be any (non-zero) locally principal ideal of 
$\OrderEnd(\bbE)$, for which we want to prove that it is of the form 
$I_{\rho',\bbE'}$.  First of all we prove that $I$ is equivalent to some ideal 
whose index in $\OrderEnd(\bbE)$ is coprime to $Np$. Indeed $I$ is a lattice 
isomorphic to $\Z^4$ and by the chinese remainder theorem, for each exponent $e$ 
the natural map \[
    I \lto \prod_{\ell \mid Np} \frac{I}{\ell^e I}
\]
is surjective, hence a fortiori, if we denote $(I \bmod Q)$ the image of $I$ in 
$\OrderEnd(\bbE)/Q\OrderEnd(\bbE)$, the following map is surjective
\[
    I \lto \prod_{\ell \mid Np}(I \bmod \ell^e) \ .
\]
In other words the map \[
    I \lto \prod_{\ell \mid Np}(I \otimes \Z_\ell) \ ,
\]
has dense image. If we now fix a generator  $\alpha_\ell$ of $I \otimes \Z_\ell$ 
for each $\ell$ dividing $Np$, then we for each precision  $e$ we can find an 
element $\alpha$ of $I$ which is congruent to $\alpha_\ell$, modulo $\ell^e$ for 
all $\ell\mid Np$. Taking $e$ large enough this implies that $\alpha$ differs 
multiplicatively from $\alpha_\ell$ by a unit of $\OrderEnd(\bbE) \otimes 
\Z_\ell$ (indeed set $U \subset \OrderEnd(\bbE) \otimes \Z_\ell$ of elements 
congruent to $1$ modulo $\ell$ is an open containing units, and for $e$ large 
enough $\alpha$ lies in the neighborhood $U \alpha_\ell$ of $\alpha_\ell$) 
implying that $\alpha$ generates $I\otimes \Z_\ell$ for all $\ell\mid Np$. Now 
take such an $\alpha$ and write $\deg(\alpha) = Q_1 Q_2$, with $Q_1$ only 
multiple of primes dividing $Np$ and $Q_2$ coprime to $Np$: then the ideal $I 
\cdot  \tfrac{\hat\alpha}{Q_1}$ is equivalent to $I$ and has index in 
$\OrderEnd(\bbE)$ coprime to $Np$.

By what we have seen before, to prove that $I$ is of the form 
$I_{\rho',\bbE'}$ we can change $I$ to an equivalent one, hence we can suppose 
that $[\OrderEnd(\bbE):I]$ is coprime to $Np$. Then we claim that $$
\ker(I) = \cap_{\alpha \in I} \ker\alpha \,.
$$
is a subgroup of $E$ or order $\sqrt{[\OrderEnd(\bbE):I]}$: indeed it is enough 
to prove that for all primes $\ell$, the $\ker(I) \cap E[\ell^{\infty}]$ has 
order equal to the $\ell$ part of $\sqrt{[\OrderEnd(\bbE):I]}$, namely order 
equal to $\sqrt{[\OrderEnd(\bbE)\otimes \Z_\ell:I \otimes \Z_\ell]}$; this is 
clear for the $\ell$'s that do not divide the index of $I$ and for the primes 
dividing such an index, which are coprime to $pN$ it is a consequence of the 
following isomorphism $$\OrderEnd(\bbE)\otimes \Z_\ell = \End(E)\otimes \Z_\ell 
\cong M_2(\Z_\ell)$$
which keeps track of the action of endomorphism on the $\ell^\infty$-torsion of 
$E$ (hence on the kernels of the endomorphisms when intersected with the 
$\ell^\infty$-torsion) by taking  basis $(\tilde P, \tilde Q)$ of the Tate 
module $T_\ell E$ and sending an endomorphism $\alpha$ to the matrix $M\in 
M_2(\Z_\ell)$ such that $(\alpha(\tilde P), \alpha(\tilde Q)) = M\cdot (\tilde 
P, \tilde Q)$.
Then we can consider the isogeny $\gamma\colon E \to E/\ker(I)$ and the 
$\Gamma$-structure $\gamma(S)$ on $E/\ker(I)$, which is well defines since 
$\gamma$ has degree coprime to $N$: the ideal $I_{\gamma, (E/\ker(I), 
\gamma(S))}$ contains all elements of $\OrderEnd(\bbE)$ whose kernel contains 
$\ker(I) = \ker(\gamma)$ hence contains $I$ and it is actually equal to $I$ 
since by the previous computations has the same index in $\OrderEnd(\bbE)$.

Now the second part of the statement. One implication is clear using  \Cref{eq:equivalent_ideals} and the fact that if $u\colon \bbE' \to \bbE'$ is an isomorphism, then
$I_{u\circ \rho'} = 
\circ u \circ \rho' = 
\Hom(\bbE', \bbE) \circ \rho'  = I_{\rho'}$. 

For the other implication, up to using \Cref{eq:equivalent_ideals}, we can suppose that $ I_{\rho',\bbE'} = I_{\rho'',\bbE''}$.  First of all we notice that the kernel of $\rho'$ can be extracted from the 
ideal: indeed by definition $\ker \rho' $ is contained in the intersection 
$\cap_{i\in I} \ker i$, and, to see that the two are equal, we choose elements 
$\phi_1, \phi_2\in \Hom(\bbE_A,\bbE)$ such that $\deg(\rho')$ is coprime to 
$\deg(\phi_1)$ and both are coprime to $\deg(\phi_2)$, so that $i_1 = \hat\phi_1 
\circ \rho'$ and $i_2 = \hat\phi_2 \circ \rho'$  are two elements in $I$ such 
that $\ker i_1 \cap \ker i_2 = \ker \rho'$ (the existence of the isogenies 
$\phi_j$ is a consequence of \Cref{cor:HomLarge}.

Applying the same argument on $\rho''$  we deduce that $\ker \rho' = \ker 
\rho''$. In particular is $E''$ is the elliptic curve underlying $\bbE''$, 
there exists and isomorphism $u \colon E' \to E''$ such that $\rho'' = 
u\rho'$. It remains to prove that $u$ is an isomorphism also when we take 
into account the level structure, i.e. that $u(S') = S''$ where we wrte 
$\bbE'' = (E'', S'')$.
Applying \Cref{lem:surjectivity_torsion}, take $\phi_1$ an element $\phi'\in 
\Hom(\bbE, \bbE') $ sending exactly $S$ to $S'$, hence having degree coprime 
to $N$. Since $ I_{\rho',\bbE'} = I_{\rho'',\bbE''}$ we deduce that there 
exists $\phi'' \in  \Hom(\bbE, \bbE'')$ such that $$\hat \phi' \circ \rho' = 
\hat \phi'' \circ \rho'' = \hat \phi'' \circ u  \rho' , $$
which implies, after taking the dual, that $\phi' = \hat u \phi''$, hence 
$\phi'' = u\circ \phi''$. Since the degree of $\phi'$ is coprime to $N$, the 
same is true for $\phi''$, hence the matrix $\tomatrix{\phi''}{N}$ 
representing $\phi''$ in basis $(P,Q)\in S$ and  $(P'',Q'')\in S''$  is a 
matrix in $\GLtwo$. Moreover, since $\phi''$ lies in $\Hom(\bbE,\bbE'')$, 
the same matrix lies in $\spanned{\Gamma}$. Hence by our hypothesis we have 
that $\tomatrix{\phi''}{N}$ lies in $\spanned{\Gamma} \cap \GLtwo = \Gamma$, 
hence
\[
    S'' = \phi''(S) = u \circ \phi'(S) = u(S') \ ,
\]
i.e. $u$ is also an isomorphism between $\bbE'$ and $\bbE''$.
    \qed
\end{proof}

\subsection{Proofs from \Cref{sec:klpt}}
\label{sec:proofs_klpt}

\lemmaScalarKLPT*

\begin{proof}[Proof of \Cref{lemma:scalarKLPT}]
    Let's assume $I$ is a right $O$-ideal. 

    1. Since $I$ is a group and $N\in\Z$, $NI\subset I$. Since $I\cap 
    \O_R(I)=O$, $NI\subset NO$ so $(\Z+NI)\subset(\Z+I)\cap(\Z+NO)$.
    
    2. Let $z\in(\Z+I)\cap(\Z+NO)$. Then there are integers $\lambda,\mu$ such 
    that $x=z-\lambda\in I$ and $y=z-\mu\in NO$.
    Note $n_I=$min$(I\cap \Z)$ and observe that $N=$min$(NO\cap \Z)$.
    Since $\normklpt{I}$ divides $n_I$ and $\gcd(\normklpt{I},N)=1$, $n_I$ and 
    $N$ are coprime.
    By applying the Chinese remainders theorem to $(\lambda,n_I)$ and $(\mu,N)$ 
    we get $\nu\in \Z$ such that $z-\nu\in NO\cap I$.
    This proves that $(\Z+I)\cap(\Z+NO)\subset (\Z+(NO\cap I))$.

    3. Let us now show that $I\cap NO\subset NI$. Let $z\in I\cap NO$. Write 
    $z=Nx$ with $x\in O$. Since $\gcd(N,\normklpt{I})=1$, there are integers 
    $\lambda,\mu$ such that $\lambda N + \mu \normklpt{I}=1$.
    Then $\lambda N x+\mu \normklpt{I}x=x$. However, $\lambda N x\in I$ since 
    $Nx\in I$ and $\mu\normklpt{I}x\in I$ since $x\in O$ and $\normklpt{I}\in 
    I\cap \Z$. Therefore $x\in I$ and $z=Nx\in NI$. This implies $NO \cap 
    I\subset NI$

    4. With 2. and 3. we get $(\Z+I)\cap(\Z+NO)\subset (\Z+NI)$. Together with 
    1. this proves the result.
    \qed
\end{proof}

\propositionScalarKLPT*

\begin{proof}[Proof of \Cref{prop:scalarKLPT}]
    Following Leroux~\cite[Section 3.4]{Leroux}, in order to produce output under plausible assumptions,
    \IdealSuborderEichlerNorm requires to get $e$ such that $\ell^e$ is larger 
    than $p^2 (\torsion^2 N_c)^3 \normklpt{I}^2$ (multiplied additionally with 
    some logarithmic factor).

    Given how we compute $I$ in \Cref{sec:klpt_inputs}, we have $\normklpt I$ 
    smaller than $\torsion^2\lambda_4$ times some logarithmic factor, where 
    $\lambda_4$ is the forth minimum of the pullback of the rerandomized input 
    to $\O_0$. 
    
    Usually, this $\lambda_4$ is as large as $\sqrt{p}$. If it is 
    larger, one can replace $\O_0$ by another special extremal order (and 
    recompute a connecting ideal), which will result in different minima. 
    Similarly, we can assume that $N_c$ is of norm $\sqrt{p}$ up to logarithmic 
    factors, and use another special extremal order otherwise.

    The technique to use alternative special extremal orders when the forth minimum of the 
    $\O_0$-ideal is not small enough is used in the KLPT-based ideal-to-isogeny algorithm in the 
    first round SQIsign NIST submission~\cite[Section 2.5.2.1]{NISTPQC-ADD-R1:SQIsign23}.
    In another context, a similar argument is used in the second-round SQIsign NIST 
    submission~\cite[Section 3.2.5]{NISTPQC-ADD-R2:SQIsign24}
    and that variant is justified heuristically in~\cite[Section 9.3]{NISTPQC-ADD-R2:SQIsign24}.

    By substituting into $p^2 (\torsion^2 N_c)^3 \normklpt{I}^2$, we obtain $p^2\torsion^6\sqrt{p}^3\torsion^4\sqrt{p}^2$
    which equals $p^{\frac{9}{2}}\torsion^{10}$ and therefore proves the proposition.
    \qed
\end{proof}

\lemmaKLPTRandomization*

\begin{proof}[Proof of \Cref{lemma:klpt_randomization}]
    1. According to \cite[Lemma 5.3.1]{Leroux}, there is an $\omega_0\in O$ such 
    that $N_c$ is inert in $\Z[\omega_0]$. Let $\omega=\torsion\omega_0$. Then 
    $\omega\in\Z+\torsion O$ and $N_c$ inert in $\Z[\omega]$.
    
    2. For $\gamma$: Let $\gamma_0\in I$ of norm coprime to $N_c$ (exists since $\normklpt{I}$ coprime to $N_c$). 
    Then $\gamma = \torsion \gamma_0\in I\cap \Z+\torsion O$, and $\gamma$ has norm $\torsion^2\normklpt{\gamma_0}$ not divisible by $N_c$, so that $\gamma\not\in N_c\O_R(I)$.
    \qed
\end{proof}

\lemmaBorelKLPT*

\begin{proof}[Proof of \Cref{lemma:borelKLPT}]
    $I$ is included in $\O_R(I)=\O_L(J)$ and $J$ in $\O_L(J)=\O_R(I)$, we have 
    $IJ$ included in both $I$ and $J$, and therefore 
    $(\Z+IJ)\subset(\Z+I)\cap(\Z+J)$.
    
    For the converse, note that $\Z+K=\O_L(K)\cap \O_R(K)$ for $K$ a primitive 
    quaternion ideal of a maximal order $\O_L(K)$.
    Since we obviously have $(\O_L(I)\cap\O_R(I))\cap 
    (\O_L(J)\cap(\O_R(J)))=\O_L(IJ)\cap\O_R(IJ)\cap\O_L(J)\subset 
    \O_L(IJ)\cap\O_R(IJ)$, we obtain $(Z+I)\cap(\Z+J)\subset \Z+IJ$.
    
    In conclusion, $ (\Z+IJ)=(\Z+I)\cap(\Z+J)$.
    \qed
\end{proof}

\subsection{Details on adapting Leroux' justification of rerandomization for \Cref{sec:klpt_rerandomization}}
\label{sec:rerandomization_details}

In \cite[Section 2.3]{Leroux}, Leroux uses the following notations and notions:

\begin{definition}
The ideal class set Cl$(\O_2)$ of an order $\O_2$ is the set of equivalence classes of its left ideals under the equivalence relation $K\sim J$ if and only if there is an invertible quaternion $\alpha$ such that $K=J\alpha$.
\end{definition}

\begin{definition}
Let $\O_2$ a maximal order in the quaternion algebra, and $O_\cap$ an Eichler order contained in $\O_2$.
Let $I_\cap$ the unique primitive left $\O_2$-ideal such that $\O_\cap=\Z+I_\cap$, and $N_\cap$ its norm.

For $K,J$ left $\O_2$-ideals of norm coprime to $N$, say that $K\sim_{\O_\cap} J$ if and only if $(K\cap \O_\cap)\sim (J\cap \O_\cap)$ in the sense of equivalence as left $\O_\cap$-ideals.

The ideal class set of $O_2$ for $O_\cap$ is denoted Cl$_{\O_\cap}(\O_2)$ and defined as the set of equivalence classes of the left ideals of $\O_2$ of norm coprime to $N_\cap$ for the equivalence relation $\sim_{\O_\cap}$.
\end{definition}

With these notions, he states:

\begin{proposition}~\cite[Proposition 2.3.12]{Leroux}
\label{propositionKLPTrerandomization}

Let $\O_1$ a maximal order in the quaternion algebra, and $\O_E$ an Eichler suborder of $\O_1$.

For $C$ in Cl$(\O_1)$, take $L\in C$ and define $\O_C:=\O_R(L)$. 

If $\O_C^\times=\langle \pm 1\rangle$, then for any $\gamma \in L\setminus N\O_C$ and quadratic order $s=\Z[\omega_s]$ of discriminant $\Delta_S$ inside $\O_1$ in which $N$ is inert, the map:

$$ \Theta:\mathbb{P}^1(\Z/N\Z)\longrightarrow \text{Cl}_{\O_E}(C)$$

$$(A:B)\longmapsto \chi_L((C+\omega_sD)\gamma)$$

is a bijection. In particular, $|Cl_{\O_E}(C)|=N+1$.

\end{proposition}

From this Proposition, Leroux then derives the randomization algorithm~\cite[Algorithm 34]{Leroux} used in SQIsign~\cite{AC:DKLPW20}.
There, $(A:B)$ is uniformly randomly sampled in $\mathbb{P}^1(\Z/N\Z)$ and then used to map an ideal of a maximal order $\O_C$ (of norm coprime to $N_E$ the norm of the unique primitive left $O_2$-ideal $I_E$ with $\Z+I_E=\O_E$) to an equivalent ideal which is in a uniformly random equivalent class for equivalence in $\O_E$.

From this rerandomization in \SQIsign, we can build the one for \SQInstructor with  1-dimensional response representation.
Recall our goal in \Cref{sec:klpt_rerandomization}: 

\begin{enumerate}
    \item Randomize a left ideal $I$ of a maximal order $\O$ with respect to the Eichler order $\O\cap \O_0=\Z+I_c$, where $\O_0$ is the usual maximal order and $I_c$ the unique primitive ideal connecting $O_0$ and $\O$.
    \item In addition, take the equivalent of $I$ in such a way that the equivalence preserves the class of $I$ with respect to $\Z+\torsion\O$. 
\end{enumerate}

The work of Leroux summarized above and \Cref{lemma:klpt_randomization} allow to achieve this:

Point 1. is exactly what \cite[Algorithm 34]{Leroux} does using his proposition 2.3.12 explained above (with $\O$ corresponding to $\O_C$, $\O_0$ to $\O_1$, $\Z+I_c=\O_0\cap\O$ to $\O_e$ and $I$ is like $L$ a representative of a class $C$).

Point 2. can be achieved by choosing $\gamma$ and $\omega_S$ such that $(C+\omega_sD)\gamma\in(\Z+\torsion\O)$. 
This is possible thanks to \Cref{lemma:klpt_randomization}.

The condition of coprimality of $I$ to the norm of $I_c$ is not a concern, since $I_c$ is modified before the rerandomization so that its norm is a large prime, which is very unlikely to divide the norm of $I$.

\section{\EUFCMA Security for \Cref{sec:hd} protocol using non-interactive assumptions}
\label{sec:proofs_sqisign_protocol_new}

In~\cite{C:ABDPW25} the authors  give a reduction of \SQIsign \EUFCMA security 
to non-interactive assumptions.
To do this they:
\begin{itemize}
    \item introduce the notion of \emph{hint-assisted} zero knowledge, where the 
        simulator is given hints sampled from a distribution \hint{E}.
    \item consider two distributions and conjecture them to be 
        undistinguishable. One is simulating the \SQIsign distribution of the 
        responses and challenges, while the other is sampling uniformly over the 
        set of isogenies of degree $d(2^{e'} -d)$ and codomain $E$, for $d$ in a 
        fixed range, then pushing them via another uniform isogeny of smooth 
        degree. We refer to~\cite[Problem 3]{C:ABDPW25} for the exact
        definition of the distributions and the aforementioned assumption.
\end{itemize}
Using this framework,  \SQIsign, as 
instantiated in~\cite{NISTPQC-ADD-R2:SQIsign24}, \EUFCMA security reduces to the 
aforementioned
indistinguishability assumption and the following generalization of 
the one endomorphism problem used already in \Cref{sec:soundness_general}.

\begin{problem}[$q\text{-}\hintOneEndp$, Problem 4~\cite{C:ABDPW25}]
    \label{prob:oneendhint}
    Given a curve $E$ sampled from the stationary distribution on 
    the set of supersingular elliptic curves over \Fpsq,
    and $q$ hints $h_1, \ldots, h_q \leftarrow \hintunif{E}$, find 
    an endomorphism in $\End(E) \setminus \Z$ in efficient representation.
\end{problem}

Here we define for \SQItorsion also two distributions \hintunif{E} and 
\hintsim{E} as in \Cref{fig:hint_dist} and give the same result for 
\SQInstructor.

Consider the following indistinguishability assumption between the two 
distributions:
\begin{problem}[$q\text{-}\hintdist$]
    \label{prob:q_hintdist}
    Let $E$ be a curve obtained from the KeyGen of \SQItorsion.
    Let $(h_1, \ldots, h_q)$ be 
    sampled with probability $\tfrac 12$ all
    from $\hintsim{E}$ and with probability $\tfrac 12$ all from
    $\hintunif{E}$.
    Given $E$ and the tuple, distinguish between the two distributions.
\end{problem}
For an algorithm $\calA$ we denote its advantage on 
\Cref{prob:q_hintdist} by
\begin{equation}
    \Adv^{q\text{-}\hintdist}(\calA) =
    \left| \prob[\calA(E,x) = 1 |\hintsim{E}] 
    - \prob[\calA(E,x) = 1 | \hintunif{E}] \right| \ .
    \label{eq:adv_q_hintdist}
\end{equation}

Using the above hint definitions and the framework from~\cite{C:ABDPW25} we
can prove:

\begin{restatable}{theorem}{hdEUFCMA}
    \label{thm:hd_eufcma}
    For any PPT adversary \calA against the {\sf \EUFCMA} security of the 
    signature scheme \SQItorsion (\Cref{sec:sqisign_protocol_new}), there exists 
    polynomial time algorithms \calB,\calD such that:
    \begin{multline*}
        \Adv_{\SQItorsion}^{\sf \EUFCMA}(\calA) \le
        (q+1) \cdot \left( \Adv^{s\text{-}\hintOneEndp}(\calB) +
          \right.  \\ \left. +
              \Adv^{s\text{-}\hintdist}(\calD) + 2^{-\secpar} + 
          \frac{1}{2\sqrt{p}} \right) 
        + (2q + s + 2)\frac{s}{\sqrt{p}} \ ,
    \end{multline*}
    where $q$ and $s$ are bounds in the number of queries that \calA makes to
    the random oracle and the signing oracle respectively.
\end{restatable}

We give the formal proof for the case in which the restarting probability of the 
signature procedure is negligible and we leave a generalization for the general case 
as a future work.
The proofs follows a similiar strategy to the one in~\cite{C:ABDPW25}, we 
given in \Cref{sec:proofs_sqisign_protocol_new} a detailed adaptation of the proof.
We refer to~\cite{C:ABDPW25} for a more in depth discussion on the security 
of \Cref{prob:oneendhint} and its relation to the Endomorphism Ring Problem 
(\Cref{prob:endring}).
We discuss the indistinguishability assumption and its differences with the one 
from~\cite{C:ABDPW25}.

\begin{figure}
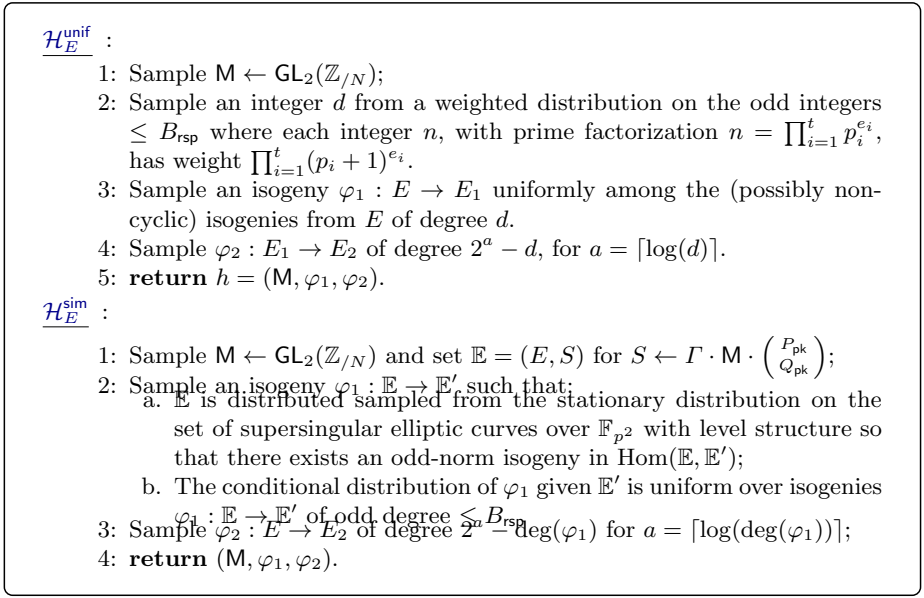

    {\centering
\begin{tcolorbox}[colframe=black, colback=white, boxrule=0.2mm, 
    width=\linewidth, fontupper=\small]
    {\footnotesize
    \begin{description}
        \item[\underline{\hintunif{E}}]:\label{dist:hint_unif} %
            \footnotesize
            \begin{algorithmic}[1]
                \State Sample $\Mat \leftarrow \GLtwo$;
                \State Sample an integer $d$ from a weighted distribution on the 
                odd integers $\leq \degreeBound$ where each integer $n$, with 
                prime factorization $n = \prod_{i = 1}^{t} p_i^{e_i}$, has 
                weight $\prod_{i=1}^{t} (p_i + 1)^{e_i}$.
                \State Sample an isogeny $\varphi_1: E \rightarrow E_1$ uniformly 
                among the (possibly non-cyclic) isogenies from $E$ of degree 
                $d$.\label{line:sample_isog}
                \State Sample $\varphi_2: E_1 \rightarrow E_2$ 
                of degree $2^{a}-d$, for 
                $a = \lceil \log(d) \rceil$.
                \State\Return $h = (\Mat, \varphi_1, \varphi_2)$.
            \end{algorithmic}

        \item[\underline{\hintsim{E}}]:\label{dist:hint_sim} %
    \begin{algorithmic}[1]
        \State Sample $\Mat \leftarrow \GLtwo$ and set $\curvelevel = (E,S)$ for  
        $S \gets \Gamma \cdot \Mat \cdot \vv{\Ppktor}{\Qpktor}$;
        \vspace{-1mm}
        \State Sample an isogeny $\varphi_1 : \bbE \to \bbE'$ such that:
        \vspace{-2mm} 
        \begin{enumerate}[label=\alph*.]
            \item $\bbE$ is distributed sampled from the stationary distribution on 
    the set of supersingular elliptic curves over \Fpsq with level structure so 
    that there exists an odd-norm isogeny in $\Hom(\bbE,\bbE')$;   
    \label{item:level_structure_rejection}
            \item The conditional distribution of $\varphi_1$ given $\bbE'$ 
                is uniform over isogenies $\varphi_1 : \bbE \to \bbE'$ of odd degree 
                $\leq\degreeBound$
        \end{enumerate}
        \vspace{-2mm}
        \State Sample $\varphi_2 : E \to E_2$ of degree $2^a - 
        \deg(\varphi_1)$ for $a = \lceil \log(\deg(\varphi_1)) \rceil$;
        \State\Return $(\Mat, \varphi_1, \varphi_2)$.
    \end{algorithmic}

\end{description}
}

\end{tcolorbox}
}
\vspace{-3mm}
    \caption{The hint distribution \hintunif{E} and \hintsim{E}. For the sake of 
    simplicity we added the sampling of \Mat in both the distributions, but it 
    is only relevant for \hintsim{E}.}
    \label{fig:hint_dist}
\end{figure}

From now we denote by $\Prot_\SQItorsion$ the identification protocol in which 
the challenge is a random element of $\Mattwo$ and $\sigprot$ 
the signature scheme obtained by applying the Fiat-Shamir transform to
$\Prot_\SQItorsion$ and modelling the hash function as a random oracle.
We denote by $\idprotoriginal$ the \SQIsign identification protocol from 
\cite{C:ABDPW25}.

For the sake of completeness we report here the hint-specific definitions from 
\cite[Section 3]{C:ABDPW25} we need for our proof.

\begin{definition} Let $R$ be a relation and $\Prot = 
    (P,V)$ a $\Sigma$-protocol for $R$.
    \begin{enumerate}
        \item A \emph{hint distribution $\hintdistr$ for $R$} is a collection of 
            distributions $\hintdistr = \{ \hintdistr_x\}_{ x \in R }$ with 
            codomains efficiently
            representable in $\lvert x \rvert$.
            We call the pair $(R, \hintdistr)$ a relation with hints.
        \item We say that $\Prot$ is \emph{$\hintdistr$-hint-assisted wHVZK} if 
            there exists a PPT algorithm, called the simulator $\Sim$, such that 
            for all $q = \poly(\secpar)$ and all PPT algorithms $\calA$ the 
            advantage in distinguishing a distribution of $q$ real transcripts 
            of $\Prot$ and $q$ simulated transcripts of $\Prot$ using $q$ 
            hints
            is negligible in $\secpar$.
            We refer to this advantage as 
            $\Advantage^{\sf hint-wHVZK}_{\Prot, \hintdistr_x, \Sim}(\calA,q)$.
        \item
            We say that $(R,\hintdistr)$ is a \emph{hard relation with hints} 
            if, in the following game, it holds for all PPT algorithms $\calA$ 
            and all $q=\poly(\lambda)$ that \begin{equation*}
                \Adv^{\sf hintrel}_{R, \hintdistr}(\calA,q)
                :=
                \probcond{(x, w^*) \in R}
                {(x, w) \gets \Gen_R(1^\secpar),\\
                h_1, \dots, h_q \gets \hintdistr_x, \\ w^* \gets \calA(x, h_1, 
            \dots, h_q)} = \negl(\secpar).
        \end{equation*}
\end{enumerate}
\end{definition}

We then recall the main theorem from~\cite{C:ABDPW25} that we use to
prove \Cref{thm:hd_eufcma}.
\begin{theorem}[\cite{C:ABDPW25}, Theorem 1]
    \label{thm:abdpw_main}
    Let $(R, \hintdistr)$ be a relation with hints.
    Let $\Prot$ be a $\Sigma$-protocol for the relation $R$ that has 
    challenge space $\mathcal{C}$ and is special-sound with respect to a 
    soundness relation $R$.
    Additionally, let $\Sim$ be a $\hintdistr$-hint-assisted simulator for 
    $\Prot$.

    For any PPT algorithm $\adv$ against the \EUFCMA of $\SIG[\Prot]$, there 
    exists a PPT algorithm $\bdv$ and an expected polynomial time algorithm 
    $\calD$ such that 
    \begin{multline}
        \Advantage^{\EUFCMA}_{\sigprot}(\adv) \leq 
        \left(q+1\right)\cdot 
        \left(
                \Adv^{\sf hintrel}_{R, \hintdistr}(\calD,s)
                 +  
                \Advantage^{\sf hint-wHVZK}_{\Prot, \hintdistr_x, \Sim}(\calB,s)
        +\frac{1}{\lvert \mathcal{C} \rvert}\right) \\
        + (q + s + 1)s \cdot \mathsf{MinEnt}( \Prot ),
    \end{multline}
    where $q$ and $s$ are upper-bounds on the number of queries that \adv makes to 
    the random oracle and to the signing oracle, respectively, and 
    $\mathsf{MinEnt}( \Prot )$ is the minimum entropy of the protocol $\Prot$.
\end{theorem}

We can use the previous theorem to finally prove the \EUFCMA security of \sigprot.

\begin{proof}[Proof of \Cref{thm:hd_eufcma}]
    The proof is a straightforward adaption of the proof 
    of~\cite[Theorem~3]{C:ABDPW25} in which we use a slightly different 
    protocol and hint distribution.
    Let $R_{\OneEndp}$ be the relation \Cref{eq:rel-oneend}, in which the 
    instance is a supersingular elliptic curve and the witnesses are non-scalar 
    endomorphisms.
    The steps of the proof consists in the following:
    \begin{itemize}
        \item Prove that \idprot satisfies the hypothesis 
            of~\Cref{thm:abdpw_main}, in particular that it has 
            \hintsim{E}-hint-assisted wHVZK, thus reducing the security of 
            \sigprot to the relation with hints $(R_{\OneEndp}, \hintsim{E})$.
        \item As done in \cite[Theorem~3]{C:ABDPW25}, we 
            finally reduce the security of \sigprot to the security of 
            \Cref{prob:q_hintdist} and \Cref{prob:oneendhint} via a game-hopping 
            argument.
    \end{itemize}

    \fakepar{First step}
    We check the hypothesis from \Cref{thm:abdpw_main}. %
    By \Cref{thm:soundness} we know that \idprot is special-sound with respect 
    to $R_{\OneEndp}$, also  the challenge space is by definition of size 
    $[\GLtwo : \Gamma] \geq 2^\secpar$.
    Also, since the  generation of the commitment curve \commitCurve is the same 
    as in the \idprotoriginal protocol, by \cite[Lemma~4.1]{C:ABDPW25} we 
    have that $\mathsf{MinEnt}{\idprot} \leq {p}^{-\tfrac 12}$.
    Finally, we need to prove that \idprot has $\hintsim{E}$-hint-assisted 
    wHVZK, for this we provide a simulator $\Sim_{\idprot}$ in 
    \Cref{alg:simulator_hints} and prove that its output is undistinguishable 
    from the real protocol.
    It is clear that the challenge matrix is uniform in both the real and the
    simulated protocol.
    Without considering the rejecting of commitment without short odd degree 
    isogenies, by \cite[Proposition 2.4]{C:ABDPW25}, the distribution of 
    $j(\commitCurve)$ are statistically close. The torsion points 
    $\Pcomtor,\Qcomtor$ are sampled uniformly in both cases (note that in the 
    real case they are randomized via \commitMatrix and in the simulated one via 
    the uniform matrix \Mat), also we are rejecting  in the both cases the same 
    subset of level structures, i.e. the ones such that $\Hom(\bbE,\bbE')$ has 
    no odd degree isogeny.
    Thus, these distributions are still statistically close, with distance 
    bounded by $s \frac{1}{\sqrt{p}}$.
    With respect to the response information we remark that it is equivalent to 
    the information given by the response isogeny \respIsogeny and the uniformly 
    random auxiliary isogeny $\auxIsogeny' : \publicCurve \to \auxCurve'$. 
    Note that in both cases they are both sampled from the same distribution: 
    for the former the odd degree isogenies in $\Hom(\bbE,\bbE')$ while for the 
    latter they are uniformly random between the isogenies of the same degree.
    We conclude that $\Advantage^{\sf hint-wHVZK}_{\idprot, \hintsim{E}, 
    \Sim_{\idprot}}(\calA,q) \leq s \frac{1}{\sqrt{p}}$ and we can apply 
    \Cref{thm:abdpw_main}.

   \begin{algorithm}[t]
        \caption{Simulator $\Sim_{\idprot}$ for \idprot} \label{alg:simulator_hints}
        \begin{algorithmic}[1]
            \Require The public parameters for \idprot, the curve $\publicCurve$ and a hint 
            $h = (\Mat, \respIsogeny, \auxIsogeny') \in \hintsim{\publicCurve}$ 
            with $\respIsogeny: \publicCurve \rightarrow \commitCurve$
            and $\auxIsogeny: \commitCurve \rightarrow \auxCurve$.
            \Ensure Transcript of $\idprot$.
            \State $\Ppk, \Qpk$ a deterministic basis of $\publicCurve[\torsion]$;
            \State sample $\gamma \gets \Gamma$;
            \State $\vv{\Pcomtor}{\Qcomtor} \gets (\gamma \cdot \Mat) \cdot \respIsogeny \vv{\Ppk}{\Qpk}$;
            \State $\commit := \left(j(\commitCurve),\Pcomtor,\Qcomtor\right)$ and $\ch := \Mat$.
            \State $d \gets \deg(\respIsogeny) )$, $a \gets \lceil \log ( d ) \rceil$;
            \State $\varphi \gets [\respIsogeny]_*(\auxIsogeny') \circ 
            \respIsogeny : \publicCurve \to \auxCurve$;
            \State $(\Paux,\Qaux) \gets ([d^{-1}]\varphi(\Ppk), [d^{-1}]\varphi(\Qpk))$;
            \State $\resp := (a,\auxCurve, \Paux, \Qaux)$;
            \State\Return $(\commit,\ch,\resp)$.
        \end{algorithmic}
    \end{algorithm}

    \fakepar{Second step} As a consequence of the first step and by
    \Cref{thm:abdpw_main}, we have that for any PPT algorithm $\adv$ against the 
    \EUFCMA of $\sigprot$, there exists an expected polynomial time algorithm 
    $\calE$ such that
    \begin{multline*}
        \Advantage^{\EUFCMA}_{\sigprot}(\adv) \leq \left(q+1\right)\cdot 
        \Adv^{\sf hintrel}_{R_{\OneEndp}, \hintsim{E}}(\calE,s) +
        \\
        + (2q + s + 2)\frac{s}{\sqrt{p}} + \frac{q+1}{2^\secpar}\ .
    \end{multline*}
    We can then repeat the same game-hopping argument 
    from~\cite[Theorem~3]{C:ABDPW25} to conclude the proof.
    For this we construct \calB from \calE by simulating the hints from 
    \hintsim{E} with \hintunif{E} instead. Any difference in the success 
    probability between \calB and \calE is bounded by the advantage in solving
    \Cref{prob:q_hintdist} with $s$ queries.
    Thus we imply
    \begin{equation*}
    \begin{aligned}
        \Adv^{\sf hintrel}_{R_{\OneEndp}, \hintsim{E}}(\calE,s) \leq
        & 
        \Adv^{s\text{-}\hintdist}(\calA) +
        \Adv^{\sf hintrel}_{R_{\OneEndp}, \hintunif{E}}(\calB,s)
         \\
        \leq 
        &
        \Adv^{s\text{-}\hintdist}(\calA) +
        \Adv^{s\text{-}\hintOneEndp}(\calB) + \frac{1}{2\sqrt{p}} \ .
    \end{aligned}
    \end{equation*}
    The last inequality is due to the fact that for \hintOneEndp Problem we are 
    considering a curve sampled from the stationary distribution instead so we 
    need to account for the statistical distance between the two distributions, 
    as done also in the proof of~\cite[Theorem~3]{C:ABDPW25}.
    \qed
\end{proof}

\paragraph{Hardness of distinguishing the hint distributions.}
    \label{rem:distinguishing_hints}
    As in \cite{C:ABDPW25}, we argue that \Cref{prob:q_hintdist} is 
    computationally hard. Both the hint distributions,
    \hintsim{E} and \hintunif{E}, provide a matrix \Mat and two isogenies 
    $\varphi_1, \varphi_2$ such that their degree sum to a power of $2$. 
    Actually both \Mat and $\varphi_2$ are sampled uniformly in both the two 
    distributions, so the only difference lie in the sampling of $\varphi_1$.
    The distribution of $\varphi_1$ in both cases depends on the bound 
    \respDeg, and for a large enough bound, the statistical distance between the 
    two curves with level structure is negligible since we consider a level 
    structure containing the scalar 
    matrices~\cite{codogni2025spectraltheoryisogenygraphs}.
    Moreover, note that the set of curves with level structure $\bbE'$
    excluded from \hintsim{E} is the same of the codomains of $\varphi_1$ in 
    \hintunif{E} since for them we have by definition the existence of an odd 
    degree isogeny from $\bbE$ to $\bbE'$.

\begin{remark}[Differences $q\text{-}\hintdist$ in~\cite{C:ABDPW25}]
    In~\cite[Problem 3]{C:ABDPW25} the authors consider an hint 
    distribution statistically close to the real one and the pushforward of 
    hints sampled from a uniform hint distribution \hintunif{E} through cyclic 
    powersmooth degree isogenies.
    Since for our protocol there is no challenge isogeny involved, only a level 
    structure, there is no third curve involved and we can directly consider the 
    uniform distribution \hintunif{\publicCurve} to sample the hints.
    Also, we argue that the use of a level structure does not fundamentally 
    change the hardness of distinguishing the two distributions, as discussed in 
    \Cref{sec:homsets_and_lvl_structures_love_story} the lattice $\Hom(\bbE, 
    \bbE')$ is a sublattice of $\Hom(E, E')$ enjoying similar properties. 
    Moreover, as already noted the curve with level structures rejected in 
    \Cref{item:level_structure_rejection} can be made negligible.
\end{remark}

\end{document}